\documentclass[reqno]{amsart}

  \input amssymb.sty

\usepackage{epsfig,latexsym,amsfonts,amssymb,amsmath,amscd,graphics,epic}
\usepackage{amsfonts,amssymb,amsmath,amscd,amsthm}
\usepackage[mathscr]{eucal}
\usepackage{mathrsfs}
\usepackage{oldgerm,units}
\usepackage{wrapfig,epsfig}

\usepackage{ifthen}
\usepackage{mathbbol}

\usepackage{amsthm}
\usepackage[mathscr]{eucal}
\usepackage{mathrsfs}
\usepackage{mathbbol}
\usepackage{oldgerm,units}
\usepackage{wrapfig}

\newtheorem{theorem}{Theorem}[section]
\newtheorem{proposition}[theorem]{Proposition}
\newtheorem{definition}[theorem]{Definition}

\newtheorem{lemma}[theorem]{Lemma}

\newtheorem{notation}[theorem]{Notation}

\newtheorem{example}[theorem]{Example}
\newtheorem{remark}[theorem]{Remark}

\newtheorem{thm}[theorem]{Theorem}
\newtheorem*{thm*}{Theorem}
\newtheorem*{dig*}{Digression}
\newtheorem{cor}[theorem]{Corollary}

\newtheorem{lem}[theorem]{Lemma}
\newtheorem{rem}[theorem]{Remark}
\newtheorem{prop*}{Proposition}

\newtheorem{prop}[theorem]{Proposition}
\newtheorem{defn}[theorem]{Definition}
\newtheorem*{examp*}{Example}
\newtheorem*{examples*}{Examples}
\newtheorem*{remark*}{Remark}
\newtheorem{Note}[theorem]{Note}
\newtheorem*{defn*}{Definition}

\newcommand{\Ref}[1]{(\ref{#1})}

\newcommand{\Comp}{\mathbb C}
\newcommand{\Real}{\mathbb R}
\newcommand{\Rati}{\mathbb Q}

\newcommand{\Net}{\mathbb N}

\newcommand{\Fld}{\mathbb K}

\def\N{\mathbb N}
\def\Q{\mathbb Q}
\def\Z{\mathbb Z}
\def\R{\mathbb R}
\def\T{\mathbb T}

\def\supp{\operatorname{supp}}

\newcommand{\one}{\mathbb{1}}
\newcommand{\zero}{\mathbb{0}}

\newcommand{\Trop}{\mathbb T}

\newcommand{\trop}[1]{\mathcal{#1}}
\newcommand{\Amb}{\trop{A}}

\newcommand{\tA}{\trop{A}}

\newcommand{\tG}{\trop{G}}

\newcommand{\tI}{\trop{I}}

\newcommand{\tM}{\trop{M}}

\newcommand{\tP}{\trop{P}}

\newcommand{\tT}{\trop{T}}

\newcommand{\tZ}{\trop{Z}}

\newcommand{\To}{\longrightarrow }

\newcommand{\tUniS}{-\infty}

\newcommand{\uuu}[1]{{#1}^\nu}

\newcommand{\Om}{\Omega}

\newcommand{\al}{\alpha}
\newcommand{\bt}{\beta}
\newcommand{\gm}{\gamma}

\newcommand{\dl}{\delta}

\newcommand{\lm}{\lambda}
\newcommand{\Lm}{\Lambda}
\def\la{\lambda}
\def\La{\Lambda}

\newcommand{\nVal}{Val}

\newcommand{\Var}{V}

\newcommand{\TrS}{\oplus}
\newcommand{\TrP}{\odot}

\newcommand{\OP}{\left(}
\newcommand{\CP}{\right)}

    \ifx\proof\undefined
    \newenvironment{proof}{
    \smallskip
    \noindent\emph{Proof.}}{\hfill\(\Box\)
    \bigskip
    } \fi

\newcommand{\vvMat}[9]{\OP \begin{array}{ccc}
  #1 & #2 & #3\\
  #4 & #5 & #6\\
  #7 & #8 & #9\\
\end{array}\CP}

\newcommand{\bfem}[1]{\textbf{\emph{#1}}}

\newcommand{\ifdef}[3]{\ifthenelse{\equal{#1}{true}}{#2}{#3}}

\numberwithin{equation}{subsection}
\numberwithin{equation}{section}

\input xy
\xyoption{all}

\def\tgm{\overline{\gm}}
\def\oD{\overline{D}}
\def\oC{\overline{C}}

\def\lgm{\overleftarrow{\gm}}
\def\rgm{\overrightarrow{\gm}}
\def\sF{{F^{\times}}}
\def\zReal{\Real_{-\infty}}

\def\tTz{\tT_\zero}

\def\nug{>_{\nu}}
\def\nuge{\ge_{\nu}}
\def\nul{<_{\nu}}
\def\nule{\le_{\nu}}
\def\trF{f_{_\Trop}}
\def\reF{f_{_\Real}}

\def\kF{f_{_\Fld}}
\def\nVal{\operatorname{val}}
\def\Cor{\operatorname{Cor}}
\def\tng{{\operatorname{tan}}}
\def\gst{{\operatorname{ghost}}}
\def\ply{{\operatorname{poly}}}
\def\nucong{\cong_\nu}

\def\RGnu{(R,\tGz,\nu)}

\def\FGnu{(F,\tGz,\nu)}

\def\TMGnu{(T(M),\tGz,\nu)}

\def\pipeGS{{\underset{\operatorname{\, gs }}{\mid}}}
\def\lmod{\mathrel   \pipeGS \joinrel\joinrel \joinrel =}

\def\lmodg{\lmod}
\def\lmodgla{\lmod}
\def\lmodgLa{\lmod}

\newcommand{\dual}[1]{{#1}^\wedge}

\newcommand{\tsqrt}[1]{\sqrt[{\operatorname{trop}}]{#1}}

\newcommand{\etype}[1]{\renewcommand{\labelenumi}{(#1{enumi})}}
\def\eroman{\etype{\roman}}
\def\ealph{\etype{\alph}}

\def\Var{X}

\def\Skip{\vskip 1.5mm}
\def\pSkip{\vskip 1.5mm \noindent}

\def\cltG{\cl{\tG}}
\def\cltM{\cl{\tM}}
\def\tGz{ \tG_{\zero}}

\def\epsc{\preceq_{\operatorname{comp}}}

\def\Fun{{\operatorname{Fun}}}
\def\sFun{{\Fun}^\times}

\def\FunSGz{\Fun (S,\tGz)}
\def\FunGz{\Fun (F^{(n)},\tGz)}

\def\CFun{\operatorname{CFun}}

\def\FunR{\Fun (R^{(n)},R)}
\def\FunF{\Fun (F^{(n)},F)}
\def\sFunF{\sFun (F^{(n)},F)}

\def\CFunR{\CFun (R^{(n)},R)}

\def\sCFunF{\CFun ((\sF)^{(n)},F)}

\newcommand{\tGinf}{\tG_{\tUniS}}

\newcommand{\cltGz}{\cl{\tG}_{\zero}}

\newcommand{\cl}[1]{\bar{{#1}}}

\def\eBin{binomial}

\def\clF{\cl{F}}
\def\clR{\cl{R}}
\def\clM{\cl{\tT }}

\def\inf{-\infty}

\def\val{value function}

\newcommand{\per}[1]{|#1 |}

\def\inf{-\infty}
\def\a{\alpha}
\def\sig{\sigma}

\def\eqR{\overset{_e}{\sim}}
\def\eqRnu{\overset{_{e, \nu}}{\sim}}
\def\neqR{\overset{_e}{\nsim}}

\def\rDiv{\, |_{_e} \: }

\def\one{\mathbb{1}}
\def\zero{\mathbb{0}}

\def\rone{{\one_R}}
\def\rzero{{\zero_R}}
\def\dzero{{\zero_{\dual{R}}}}
\def\fone{{\one_F}}
\def\fzero{{\zero_F}}

\def\singular{ordinary}

\def\individual{ordinary}

\def\essn{{\operatorname{es}}}

\def\ef{f^{\essn}}
\def\eg{g^{\essn}}
\def\eh{h^{\essn}}
\def\eq{q^{\essn}}

\def\vf{f^{\tng}}

\def\gf{f^{\gst}}

\def\iif{f^{\operatorname{in}}}
\def\ig{g^{\operatorname{in}}}

\def\iq{q^{\operatorname{in}}}

\def\ra{a}
\def\bfa{\textbf{\ra}}
\def\rb{b}
\def\bfb{\textbf{\rb}}
\def\rc{\textsl{c}}

\def\bfc{\textbf{\rc}}
\def\bfi{ \textbf{i}}
\def\bfj{\textbf{j}}
\def\bfk{\textbf{k}}
\def\bfh{\textbf{h}}

\def\TrT{\odot} 

\def\Rnu{\mathbb R^\nu}

\def\pl{\lfloor} \def\pr{\rfloor}

\begin{document}

\title[Supertropical algebra] {Supertropical algebra }
\author[Z. Izhakian]{Zur Izhakian}\thanks{The first author was supported by the
Chateaubriand scientific post-doctorate fellowships, Ministry of
Science, French Government, 2007-2008.}
\thanks{\textbf{Acknowledgement} The authors would like to thank Professors J. Bernstein, D. Haile, M. Knebusch, D. Speyer, S. Shnider, and B. Sturmfels  for
helpful suggestions concerning earlier versions.}
\address{Department of Mathematics, Bar-Ilan University, Ramat-Gan 52900,
Israel} \address{ \vskip -6mm CNRS et Universit´e Denis Diderot
(Paris 7), 175, rue du Chevaleret 75013 Paris, France}
\email{zzur@math.biu.ac.il, zzur@post.tau.ac.il}
\author[L. Rowen]{Louis Rowen}
\address{Department of Mathematics, Bar-Ilan University, Ramat-Gan 52900,
Israel} \email{rowen@macs.biu.ac.il}


\date{\today}




\begin{abstract} We develop the algebraic polynomial theory for
``supertropical algebra,''  as initiated earlier over the real
numbers by the first author. The main innovation there was the
introduction of ``ghost elements,'' which also play the key role
in our structure theory. Here, we work  somewhat more generally
 over an ordered monoid, and develop a theory which
contains the analogs of several basic theorems of classical
commutative algebra. This structure enables one to develop a
Zariski-type algebraic geometric approach to tropical geometry,
viewing tropical varieties as sets of roots of (supertropical)
polynomials, leading to an analog of  the Hilbert Nullstellensatz.

Particular attention is paid to
 factorization of polynomials. In one indeterminate, any polynomial can be factored into linear and quadratic
 factors, and unique factorization holds in a certain sense. On the
 other hand,
 the failure of
  unique
 factorization   in several indeterminates
is explained by geometric phenomena described in the paper.
\end{abstract}

\maketitle






\section{Introduction}
\numberwithin{equation}{section}

One of the goals of algebra is to find the ``correct'' algebraic
structure with which to frame some mathematical theory. The
underlying motivation of this paper is to provide a direct
algebraic approach to the rapidly developing theory of tropical
mathematics. Tropical geometry has been the subject of intensive
recent research, including some remarkable applications in various
areas of mathematics, such as combinatorics, polynomials (Newton's
polytopes), linear algebra, and algebraic geometry; cf.~\cite{IMS}
and~\cite{SturmfelsSolving}. Before bringing in our structure, let
us review briefly how one passes from ``classical'' algebraic
geometry to tropical geometry.

For any complex affine variety $W = \{ ( z_1, \dots, z_n): z_i \in
\Comp\} \subset \Comp^{(n)},$ and any small $t$, one could define
its \textbf{amoeba}, cf. \cite{Gelfand94},
$$\tA(W) = \{ (\log_t| z_1|, \dots, \log_t|  z_n|):
( z_1, \dots, z_n) \in W \} \subset \zReal^{(n)},$$ where $\zReal
:= \Real \cup \{ -\infty \}$. Note that $\log_t| z_1 z_2 | =
\log_t| z_1| +\log_t| z_2 |$, and the limiting case $t \to 0$
degenerates to a polyhedral complex, i.e., \textbf{non-Archimedean
amoeba}, in~$\zReal^{(n)}$ where now $\zReal$ is given the
structure of the \textbf{max-plus} algebra, for which the new
addition is defined as the maximum, multiplication is taken to be
the original addition in $\Real$, and the zero element is
$-\infty$. Passing from the original algebraic variety to this
``tropical variety'' preserves various geometric invariants
involving intersections, and has been used to simplify proofs of
deep results from algebraic geometry.
 As developed in
\cite{AHK1998,Akian2005,Butkovic2003,Cuninghame79,gaubert-2004-2,Gaubert199,MikhalkinEnumerative,pin98,Shustin2005,Speyer4218,Speyer2003,SpeyerSturmfels2004},
the  max-plus algebra (or dually, the min-plus algebra) lies at
the foundation of ``tropical algebra'' and ``tropical geometry.''
A survey can be found in~ \cite{Litvinov2005}, and
\cite{Gathmann:0601322} provides a fine explanation of how one
arrives at tropical geometry defined over the max-plus algebra.

Although many ideas of tropical geometry can be found in the
pq-webs of~\cite{AHK1998}, researchers in tropical geometry have
focused on definitions of tropical varieties arising from complex
analysis and symplectic geometry.  In the simplicial geometric
approach of \cite{MIK07Intro}, a finite polyhedral complex is said
to be of \textbf{pure dimension} $k$ if each of its faces of
dimension $< k$ is contained in a $k$-dimensional face -- called a
\textbf{top-dimensional} face. A \textbf{$k$-dimensional tropical
variety} $\Var\subset \Real^{(n)}$ is a finite rational polyhedral
complex of pure dimension $k$ whose top-dimensional faces $\delta$
are equipped with positive integral \textbf{weights} $m(\delta)$
such that, for each face $\sigma$  of codimension $1$ in $\Var$,
the following condition is satisfied, called the \textbf{balancing
condition}:
\begin{equation}\label{eq:balanceCond} \sum_{\sigma\subset\delta}m(\delta)n_\sigma(\delta)=0\
,\end{equation}
where $\delta$ runs over all $k$-dimensional faces of $\Var$
containing $\sigma$, and $n_\sigma(\delta)$ is the primitive unit
 vector normal to $\sigma$ lying in the cone centered at $\sigma$
and directed by $\delta$. Accordingly, a tropical hypersurface,
i.e., an $(n-1)$-dimensional tropical variety in $\Real^{(n)}$,
must have (topological) dimension $n-1$.

An alternative approach, more algebraic in nature,  is to
 consider tropical polynomials as piecewise linear functions
$\reF : \Real^{(n)} \to \Real$; then  the \textbf{corner locus},
denoted $\Cor({\reF})$, is defined as the domain of
non-differentiability of  $\reF$, or, in other words, the set of
points on which the evaluation of $\reF$\ is attained by at least
two of its monomials. Yet, this notion has no pure algebraic
framework over the max-plus algebra $\zReal$, and our structure
aims for such a framework.

 There is a direct passage
from (classical) affine algebraic geometry to tropical geometry,
in which algebraic varieties are transformed to polyhedral
complexes. Namely, the max-plus algebra appears as the  target of
a non-Archimedean valuation $ \nVal : \Fld \to \zReal$ of   the
field $\Fld$ of locally convergent \emph{Puiseux series} of the
form $ p(t) = \sum_{\tau \in T} c_{\tau} t ^{\tau},$ where $c_\tau
\in \Comp$ and  $T \subset \Rati$ is bounded from below, where
\begin{equation*}\label{eq:valPowerSeries}
  \nVal(p(t))\  := \
\left\{%
\begin{array}{ll}
    - \min \{\tau \in R \ : \; c_{\tau} \neq 0 \}, &  p(t) \in
    \Fld^\times,
    \\[1mm]
    -\infty, & p(t) = 0. \\
\end{array}%
\right.
\end{equation*}

Given a polynomial $\kF = \sum_{\bfi \in \Om} p_{\bfi} \lm_1^{i_1}
\cdots \lm_n^{i_n}$, for $p_{\bfi} = p_{\bfi}(t)$,  over $\Fld$
with zero set $\tZ(\kF) \subset \Fld^{(n)}$,  the non-Archimedean
\textbf{amoeba} $\widetilde \tA(\kF) \subset \zReal^{(n)}$ is now
defined be the closure  $\overline{\nVal(\tZ(\kF))}$ of
$\nVal(\tZ(\kF))$, where the valuation is taken coordinate-wise.

\begin{theorem}[Kapranov, \cite{IMS}]\label{thm:Kapranov}  $\widetilde \Amb(\kF)$
is contained in the corner locus of the tropical function
\begin{equation}\label{eq:Kapranov}
\reF(\bfa)  = \max_{\bfi \in \Om}( \langle \bfi,\bfa \rangle +
\nVal(p_{\bfi})), \qquad \bfa = (a_1, \dots, a_n)\in \Real^{(n)},
\end{equation}
where $\langle \, \cdot , \cdot \, \rangle$ stands for the
standard scalar product. (Note that the term $a_1^{i_1}\cdots
a_n^{i_n}$ is evaluated as~$\langle \bfi,\bfa \rangle$ in the
max-plus algebra.) Equality holds when $\nVal$ is onto.
\end{theorem}
\noindent  Kapranov's  theorem implies not only that every
non-Archimedean amoeba  is a corner locus of a tropical
polynomial, but also that any corner locus of a tropical
polynomial
$\reF$
is a non-Archimedean amoeba. In \cite{RST},  tropical varieties
were defined as   non-Archimedean amoebas
$\overline{\nVal(\tZ(I))}$, where $I \triangleleft \Fld[\lm_1,
\dots, \lm_n]$. (However, there exist balanced polyhedral
complexes of codimension $> 1$ that cannot be described as
non-Archimedean amoebas.)

$ $From a categorical perspective, one would like to study these
tropical varieties directly, in terms of the underlying algebraic
structure. \begin{defn} A \textbf{semiring} $(R,+,\cdot \, ,
\rzero, \rone)$, is a set $R$ endowed with binary operations $+$
and $\cdot \, $ and distinguished elements $ \rzero$ and $\rone)$,
 such that  $(R,\cdot \, ,\rone)$ and $(R,+,\rzero)$ are
monoids satisfying distributivity of multiplication over addition
on both sides, and such that $\rzero \cdot r = r \cdot \rzero =
\rzero$ for every $r\in R$. \end{defn} Semirings have attracted
interest because of their impact on computer science, and we use
\cite{golan92} as a general reference. Occasionally we need the
more general notion of a \textbf{semiring without zero}, which
satisfies all the axioms of semiring except those involving the
element $\rzero.$ For $R$ any semiring without zero, we obtain a
semiring by formally adjoining the element $\rzero$ and
stipulating that $a + \rzero = \rzero + a = a$ and $a \cdot \rzero
= \rzero \cdot a = \rzero$ for each $a\in R$.

 The
algebraic structure of the max-plus algebra is that of a semiring
without zero, which becomes a semiring when we formally adjoin the
element $-\infty$. The complications in utilizing the max-plus
algebra as the underlying structure in tropical geometry crop up
almost immediately. Unfortunately, the max-plus algebra has no
additive inverse (even after one adjoins $-\infty$), and thus its
algebraic structure as a semiring is handicapped.

Consequently, the direct algebraic-geometric development of the
category of tropical varieties has lagged behind. For example, one
could define the algebraic set of a polynomial $f$  to be the corner
locus of the
 function $\zReal^{(n)} \to \zReal$ determined by $f$. This formulation is more cumbersome than
the classical formulation in algebraic geometry that $f(\bfa)=0$,
$\bfa = (a_1,\dots,a_n)$, and its awkwardness becomes apparent the
moment one starts to work with algebraic sets. The alternative
definition used in \cite{RST} for the algebraic set of $f$, namely
the set of points on which $f$ is not differentiable, works well
from the perspective of differential geometry, but is even more
difficult to apply in various algebraic situations. Consequently,
much current research relies heavily on passing back and forth
frequently from ``classical'' algebraic geometry to tropical
geometry.

Furthermore, although  any non-Archimedean valuation $\nVal$
satisfies $\nVal(pq) =  \nVal(p) + \nVal(q)$ as well as
$$\nVal(p+q) = \max \{\nVal(p), \nVal(q)\} \quad \text{if} \quad \nVal(p) \ne
\nVal(q),$$ one does not know $\nVal(p+q)$ in the case that
$\nVal(p) = \nVal(q).$ Thus, from the point of view of Kapranov's
Theorem, not only is the max-plus algebra a difficult structure to
study, but in some sense it may not even be the right structure.

The first author \cite{zur05TropicalAlgebra} had addressed these
issues by introducing
 \textbf{extended tropical arithmetic} $\mathbb T$, the disjoint union of two copies of $\mathbb R,$
 denoted respectively as $\mathbb R$ and~$\mathbb R^\nu =
\{ a^\nu : a \in \mathbb R\},$ together with a formal element
$-\infty$. One defines the map $\nu:  \mathbb T \to \zReal^\nu$
to be the identity on $\zReal^\nu := \Rnu \cup \{\inf\},$ and to
satisfy $\nu (a) = a^\nu$ for each $a \in \mathbb R$. (As presently defined,  $\nu$  is 1:1.)

  $\mathbb T$ is also endowed with the two operations
$\TrS$ and $\TrP$, satisfying the following axioms (using the
generic notation that $a,b\in \R,$ $x,y \in \mathbb T$):
\begin{enumerate}
\Skip \item $-\infty \oplus x = x \oplus -\infty = x;$
\Skip \item $x \TrS y = \max \{x,y\}$ unless $\nu(x) = \nu(y);$
\Skip \item $a\TrS a = a^\nu \TrS a^\nu = a  \TrS a^\nu = a^\nu
\TrS a = a^\nu;$
\Skip
\item $\inf \TrP x = x \TrP \inf =
\inf;$
\Skip
\item $a \TrP b = a+b$ for all $a,b \in \R;$
\Skip
\item $ a^\nu \TrP b
= a \TrP b ^\nu = a^\nu \TrP b^\nu = (a+b)^\nu$.
\end{enumerate} $(\T, \TrS, \TrP, -\infty, 0)$  is seen in \cite{zur05TropicalAlgebra} to
have the structure of a (non-idempotent) commutative  semiring.
The verification is a special case of Lemma~\ref{semiring} below.
Our motivating example is $\T $ with this notation, which we call
\textbf{logarithmic notation}, where the zero element
$\zero_\Trop$ is~$\inf$.

\begin{defn} A semiring
homomorphism $\nu : R \to R$ is \textbf{idempotent} if $\nu^2 =
\nu.$\end{defn}

 Note that $\R^\nu$ (with the max-plus operations) is
a sub-semiring without zero of $\T$ isomorphic to the usual
max-plus algebra, and the map $\nu: \T \to \zReal^\nu $ is an
idempotent semiring  homomorphism.  Moreover, $\zReal^\nu$ is a
semiring  ideal of $\Trop$.   In this sense, $\T$ is a ``cover''
of the max-plus algebra (and its role is similar to that of a
covering space). In applying $\T$ to tropical geometry, one
focuses on the first copy of $\R$, which we call the set of
\textbf{tangible elements}, while elements of $\zReal$  are called
\textbf{ghost elements}; $ \zReal^\nu$ is called the \textbf{ghost
ideal}.

The lack of additive inverses is bypassed by
 identifying all ghost elements in some sense as ``zero''; this leads
to a much more malleable structure theory, which is also
compatible with tropical geometry. The intuition here is that the
second component $\R ^\nu$ is a ``shadow'' of the tangible
component $\Real$, with respect to which  a ghost element $a^\nu$
could be interpreted as the interval from $-\infty$ to $a$, in the
sense that there is an uncertainty and one does not know which
element in this interval to choose. Thus, its elements often act
as ``noise,'' especially with regard to multiplication, and one is
led to treat this ghost component the same way that one would
customary treat the zero element in commutative algebra.

It is surprising how well the use of the ghost ideal enables one
to overcome the shortcomings of the general structure theory of
semirings. Also, as we shall see in this paper, non-tangible
elements also have their own special properties of independent
interest.

Polynomials over $\Trop$  are defined as formal sums
$$\bigoplus _{i \ge 0} \ \a _i \!\TrT\! \la ^i$$ where almost all
$\a _i = \zero_\Trop;$ addition (denoted $\TrS$) and
multiplication (denoted $\TrP$) of polynomials are defined in the
usual manner.
 In order to simplify the notation, we write
polynomials in the usual notation, understanding that $+$ now
means $\oplus$, and $\cdot$ now means $\TrT;$   for example, over
$\Trop$,  the computation
$$(\la \oplus 7)\TrT(\la \oplus 3) = (\la \TrT \la) \oplus (7 \! \oplus\!
3)\! \TrT \!\la \oplus (7\! \TrT\! 3) = (\la \TrT \la) \oplus
7\!\TrT \! \la \oplus 10 $$ is rewritten as
$$(\la +7) (\la +3) = \la ^2 + 7 \la +10.$$

Note that the polynomial semiring $\Trop [\la]$ is not a max-plus
algebra since, for example,
$$(\la +2) + (2\la +1) = 2\la +2.$$  We shall cope with this
difficulty shortly.

In this paper we generalize the structure of $\T $ to the more
abstract setting of a \textbf{supertropical semiring} $R = \RGnu$,
in which  $\tGz := \tG \cup \{ \rzero \}$ is an ideal, called the
\textbf{ghost ideal}, and $\nu: R \to \tGz$ is an idempotent
semiring homomorphism.  The ``supertropical'' structure defined in
\S\ref{Supsem} gives an axiomatic description of the extended
tropical arithmetic  $\mathbb T$. Our overall  objective is to
cover the max-plus algebra by an algebraic structure that we call
the \textbf{supertropical semiring}, which has a more reasonable
structure theory, and in whose language many basic concepts of
tropical geometry can be described more intrinsically. The main
structures for us are the \textbf{supertropical domain}
(Definition~\ref{dom}) in which $\tT = R\setminus \tGz$ is a
monoid comprising the \textbf{tangible elements} (which provide
the link to tropical geometry), and especially the special case of
a \textbf{supertropical semifield}
(Definition~\ref{def:semifield}) in which  $\tG$ is an ordered
Abelian group.

A few words about our terminology {\it supertropical} and its
interpretation. Usually ``super'' in mathematics means graded by
the additive group $(\mathbb Z _2,+).$ However, here our structure
is ``graded'' by the multiplicative monoid $(\mathbb Z _2,\cdot)$
(viewing the tangibles as the 1-component and the ghosts as the
0-component), since the product of elements of degree $i$ and $j$
is an element of degree~$ij$. Our focus is on the tangible
elements, which provide the link the usual tropical theory.
Nevertheless, at times it is useful to view the supertropical
semiring as a ``cover'' of the max-plus algebra, via the ghost map
$\nu$.

As noted earlier, the structure of a polynomial semiring over a
supertropical semiring is no longer supertropical, so, in order to
study polynomials over a supertropical semiring, we introduce a
somewhat weaker algebraic structure, that of a \textbf{semiring
with ghosts}, which also enables us to handle matrices. Viewing
the algebraic theory from this perspective, one can carry over
much of the classical theory of commutative algebra and linear
algebra.

The \textbf{roots} of a polynomial $f \in R[\la _1, \dots, \la
_n]$ over a supertropical semiring $R$ are defined as those
$n$-tuples $\bfa = (a_1, \dots, a_n) \in R^{(n)}$ such that
$f(\bfa)$ is a ghost element. (We call them roots even when $n>1$,
since the more customary terminology ``zeroes'' seems misleading
in this context.) The geometric object of interest to us is the
set of tangible roots of a supertropical polynomial, denoted as
$\tZ_{\tng} (f)$. This definition encompasses other formulations
in tropical geometry, as we see in \S\ref{compare}, and  is
considerably neater than the customary definition of tropical root
described above; especially when one needs to add and multiply
polynomials.

This definition  permits us to describe tropical varieties  as in
classical algebraic geometry. The tropical variety $\widetilde
\Amb(\kF)$ arising from the original algebraic variety $W =
\tZ(\kF)$ should be written as the set of roots of $\trF\in
\Trop[\lm_1, \dots, \lm_n]$, suitably interpreted in our new
structure, as to be made explicit in \S\ref{rootsofpol} below.
This approach provides a clear-cut categorical framework for a
direct algebraic study of tropical varieties, without constantly
referring back to classical algebraic geometry, much in the spirit
that one can study the category of Lie algebras without always
referring back to Lie groups.
 Our approach also yields the extra dividend of providing new, previously inaccessible, examples in tropical
geometry,  such as subvarieties having the same dimension as the
original variety (as exemplified in Figure \ref{fig:2dimset}; also cf.~Example~\ref{tangex0}).

\begin{figure}
\begin{minipage}{0.9\textwidth}
\begin{picture}(10,140)(0,0)
\includegraphics[width=5.5 in]{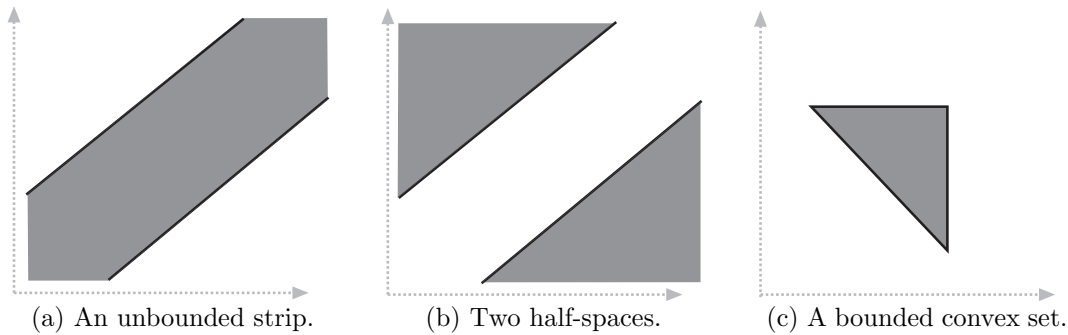}
\end{picture}
\end{minipage}
\begin{minipage}{0.25\textwidth}
\center{(a) An unbounded  strip. }

\end{minipage}\hfil
\begin{minipage}{0.22\textwidth}
\center{(b) Two half-spaces.}

\end{minipage} \hfil
\begin{minipage}{0.25\textwidth}
\center{(c) A  bounded convex set.}

\end{minipage}
\caption{\label{fig:2dimset} Illustrations of supertropical
varieties in 2-dimensional space .}
\end{figure}


 Our main result in this paper is a tropical version of the
Hilbert Nullstellensatz (Theorem~\ref{Null2}). This part of the
theory is rather delicate, because the connection between algebra
and geometry is more subtle than in the classical case -- here,
radical semiring ideals correspond to components of the
complements of root sets.

One needs to study factorization of polynomials to facilitate the
computation of roots, but this is a delicate matter. Much of the
difficulty in   factorizing of polynomials arises from the fact
that polynomials that look quite different may behave as the same
function from $R^{(n)}$ to~$R$. Thus, strictly speaking, we should
study the natural image of the polynomial semiring in the semiring
of functions from $R^{(n)}$ to~$R$. This leads to equivalence
classes of polynomials which we call $e$-\textbf{equivalent}, and
representatives of a specific form, which we call \textbf{full}
polynomials.

Let $\Net$ denote the positive natural numbers. It is not
difficult to show when the supertropical semifield $F$ is
$\Net$-divisible, that every polynomial that is not a monomial has
a tangible root.

Since factorizations of polynomials respect the roots, we consider
factorization of polynomials (up to $e$-equivalence). In the case
of one indeterminate, one already has the analog
\cite{Cuninghame79} of the fundamental theorem of algebra, that
every tangible polynomial can be factored (as a function) uniquely
as a product of linear tangible polynomials, stated in the context
of supertropical algebras as Propositions~\ref{fundam} and
\ref{prop:emptyzeroset}. In general, we have a full description of
factorization of a polynomial $f$ in one indeterminate (as a
product of linear and quadratic factors) in terms of the tangible
roots of $f$; cf.~Theorems~\ref{factor1} and ~\ref{fullfact2} and
Proposition~\ref{mingho}.

Although something like unique factorization holds in one
indeterminate, it fails miserably in several indeterminates.
However, its failure should be interpreted geometrically as the ability to partition a tropical variety in different ways as a union of irreducible subvarieties. All non-unique factorizations that we know are consequences of such geometric ambiguities.
From a more positive viewpoint,
every polynomial divides a product of binomials
(Theorem~\ref{permprime}), which has the geometric consequence
that every algebraic set is embedded naturally into a finite union
of hyperplanes; also, there is a way to obtain the minimal
such product, as illustrated in Example~\ref{exm:2}. This latter
result is best understood in terms of Laurent polynomials (whose
root sets match those of polynomials), since we can extend the
natural algebraic duality between the max-plus algebra and the
min-plus algebra (given by sending an element to its inverse) to
the Laurent polynomial semiring without zero,
 thereby yielding a geometric duality.

 One bonus of viewing polynomials (and Laurent
polynomials) as functions is the surprising result reminiscent of
the Frobenius automorphism (Corollary~\ref{Prop:powOfPols}):
$$\bigg(\sum f_i\bigg)^n = \sum f_i^n$$ for any natural number $n$.

Tangible polynomials provide the (affine) varieties familiar from
tropical geometry;   yet, nontangible polynomials yield new and
interesting examples of varieties. Consequently, our theorems
about polynomials often are stated for arbitrary supertropical
polynomials, even though the formulations and proofs may  be
shorter in the tangible case. Also, there is a polyhedron which
yields the correspondence of supertropical polynomials  with
Newton polytopes (Proposition~\ref{convex}); this is analogous to
the \textbf{grid} in~\cite{Akian2005}.

 Although the algebraic definitions given in this paper can be generalized even
further,  permitting different ``layers'' of ghosts, we feel that
the rich theory described above justifies the presentation of the
structure theory at the current level of generality. This theory
also is useful in describing matrices and solutions to equations.
In   subsequent papers including \cite{IzhakianRowen2008Matrices},
we
 develop the matrix theory, including the description of nonsingular matrices in terms
of the tropical determinant (which is really the permanent), and a
supertropical version of the Hamilton-Cayley theorem. Resultants
of supertropical polynomials are studied in
\cite{IzhakianRowen2009Resulatants}.


\section{Valued monoids}

In  \S\ref{Supsem}, we define our main algebraic structure:
Supertropical domains, and supertropical semifields. Since their
definitions could seem technical at first, we motivate them with a
preliminary structure that provides our major example, as well as
a transition to the supertropical
 theory.

 A
 monoid $M$ is
\textbf{ordered} if $M$ as a set has a total order $\le$ such that
$$\text{ $ab \le ac$ and $ba \le ca$  \quad  for all $b \le c$ and $a$ in
$M$}.$$ Given any  ordered monoid $(\tG, +)$, one may adjoin the
formal element $\inf$ to $\tG$ by declaring $$\inf < g, \qquad
\forall g \in \tG,$$ and define $(\inf) + g= g+ ( \inf)  = -\infty
$, $\forall g \in \tG .$  We denote this new ordered monoid
$\tGinf := \tG \cup \{ \inf \}$, declaring $\inf + \inf = \inf$.
  Of course $\tGinf$ is
not a group, even if $\tG$ is a group.

Recall that a monoid homomorphism   from $(M,\cdot)$ to $(\tG,+)$
is a function $$\varphi: M \To \tG$$ such that $\varphi(\one_M) =
0_\tG$ (the neutral element of $\tG$) and $\varphi(a b) =
\varphi(a)+ \varphi(b)$, $\forall a,b \in M$.

\begin{defn}\label{def:valuedMonoid} A monoid $(M,\cdot \;)$ is \textbf{valued with
respect to an ordered cancellative monoid~$(\tG,+)$} if there is a  monoid
homomorphism $v : M \to \tG$. We notate this set-up as the
\textbf{triple} $(M,\tG,v)$; $v$ is called the \textbf{\val} for
$(M,\tG,v)$.

 Given any triple  $(M,\tG,v)$, where $\tG = (\tG,+)$, we
define the \textbf{extended semiring} $T(M,\tG, v)$  to be the
triple $(T(M),\tG_{\inf}, \nu)$ where $T(M)$ is the disjoint union
$M \cup \tGinf,$ whose value function $$\nu : T(M) \To \tGinf$$
extends the original \val\ $ v$ by putting $\nu(g) = g$, $\forall
g\in \tGinf.$ Furthermore, $T(M)$ is made into a semiring, where
multiplication is defined by incorporating the given monoid
operations of $M$ and $\tGinf$, and also defining \Skip
\begin{enumerate}
\item $(\inf)\odot x = x\odot (\inf) = \inf$ for all $x \in T(M)$;
\Skip \item  $a \odot g = v(a) +g$ and $g\odot a = g+v(a) $ for
all $a \in M, \ g \in \tGinf$; \Skip
\end{enumerate}
addition $\TrS$ on $T(M)$ is defined as follows, for  $x,y \in
T(M)$:

\begin{equation}\label{addition1}  \ x \TrS y =
\begin {cases} x  \qquad \ \text{if} \quad \nu(x) > \nu(y);\\
y  \qquad \ \text{if} \quad \nu(x) < \nu(y);\\
\nu(x)    \quad \text{if} \quad \nu(x) =
\nu(y).\\
\end{cases}\end{equation}
\end{defn}

\begin{lem}\label{semiring} $T(M)$ is a semiring.\end{lem} \begin{proof}  The operation $\TrS$ is clearly commutative; to check that
$\TrP$ is distributive over $\TrS$, one wants to verify that
\begin{equation}\label{dist1}  x \TrP (y \TrS z) = (x \TrP y) \TrS (x\TrP z).\end{equation} This is clear
if one of the entries is $\inf$. If $\nu(y) \ne \nu(z),$ then by hypothesis $\nu(x \TrP y) \ne \nu(x \TrP z)$,
so again \eqref{dist1} holds.

 So
assume that $\nu(y) = \nu(z).$ Then $$x \TrP (y \TrS z) = x \TrP
\nu(y) = \nu(x)+ \nu(y) =\nu(x \TrP y) = (x\TrP y) \TrS (x \TrP
z),$$ as desired. Associativity of addition is checked in a
similar fashion. Associativity of multiplication is
clear.\end{proof}

Note that the zero element of $T(M)$ is $\inf$, whereas the
original unit element $\one_M$ of~$M$ is also the unit element of
$T(M)$, in view of the following verifications:
$$\one_M \TrP (\inf) = (\inf) \TrP \one_M = \inf, \qquad \one_M
\TrP a = a \TrP \one_M = a$$
for all $a\in M$, by definition, whereas, for all $g \in \tGinf,$
$$\one_M \TrP g = v(\one_M)+g = (-\infty) +g = g.$$

\begin{rem}\label{add1} The monoid $(T(M), \TrP)$ is valued in $\tGinf$, with
respect to the \val\ $\nu$. Indeed, let us check that   $$\nu
(x\odot y) = \nu(x) + \nu(y), \quad \forall x,y \in T(M).$$ If $x$
and $y$ are both in $M$ or in $\tGinf$, then this is true by
definition, so suppose $x\in M$ and $y \in \tGinf.$ Then
$$\nu(x\odot y) = \nu(v(x)+y) = \nu(v(x)) +\nu(y)= v(x) + \nu(y)= \nu(x) + \nu(y).$$
\end{rem}

\begin{example} \label{Exmp:valuedMonoid} Here are some examples of extended semirings (of valued
monoids).
\begin{enumerate} \eroman
    \item
 For any ordered monoid $(\tG,+ \ )$, we  have the
 extended semiring  $D(\tG):= T(\tG,\tG, 1_{\tG})$, where $M = \tG$ and $1_\tG$ is the identity map. The semiring $D(\R)$ is
 the extended tropical arithmetic, in the
sense of
  \cite{zur05TropicalAlgebra}.  Note that in $D(\tG)$, $\nu$ restricts to a 1:1 map $\nu
_{M} : M  \to \tG.$ \pSkip
  \item Recall that $\R^\times$ denotes $\R\setminus \{0 \},$ and
  ${\R}^+$ denotes the positive real numbers.
 The group  $(\R^\times,\cdot \; )$, with
its absolute value, yields the triple $(\R^\times,{\R}^+,
|\phantom{w}|)$. We often refer back to this example for intuition
in the case that $v$ is not 1:1.   Likewise, for any ordered field
$F$, we have the valued monoid $(F^\times,{F}^+, |\phantom{w}|).$
\pSkip
\item  If $F$ is a field with valuation $v : F \to \tG,$ then
 $T(F^\times, \tG,v)  $ is an extended semiring. (In particular, we could
 take $F$ to be the field of Puiseux series.)
 \pSkip
\item Algebraic groups over ordered fields, or over fields with
valuation, can be valued by
 means of the determinant.

 \end{enumerate}
 \end{example}

\begin{rem}\label{add2} Given a triple $(M,\tG,v)$, one can
also define the \textbf{dual extended semiring}
$\dual{T}(M,\tG,v)$, where addition is defined by reversing the
order in the formula \eqref{addition1}; namely $x\oplus y$ equals
$y$ if $\nu(x) > \nu(y),$ and equals $x$ if $\nu(x) < \nu(y).$ As
before, we formally adjoin the element $-\infty$. This duality
will be explained algebraically in \S\ref{duality}.
 \end{rem}


\section{Supertropical semirings}\label{Supsem}

 We start  this section by laying out the basic algebraic notion
of a supertropical domain, showing how it is just a reformulation
of a valued monoid. Then, having made the transition to semirings,
we can then bring in related semirings such as the
 semiring of functions of \S\ref{sec:func} and the polynomial
semiring. In this paper, we assume throughout that all of our
semirings are commutative (under multiplication as well as
addition), although in \cite{IzhakianRowen2008Matrices} we need to
drop this assumption in order to deal with matrices.

 Having already constructed our main object $R=T(M)$, together
with the operations $\TrS$ and $\TrP$, let us first describe it
more intrinsically in the language of semirings.

\begin{Note}\label{ignore}   In line with the customary algebraic notation for semirings, the zero
element $\rzero$ of $R$   replaces
 what we originally called
$-\infty$. Under this notation, we write $\tG_\zero $ instead of
$\tGinf.$ Likewise, multiplication in $R$ is taken to subsume the
original monoid operation of $\tG$, so  $\nu (\rone) $ is the
neutral element of $\tG.$

 But
 in order to emphasize the tropical aspect,
 we often revert to
what we have called {\it logarithmic notation} when discussing our
motivating example $R = T(\R);$ in these instances we retain the
usage of $-\infty$ for the zero element and $0$ for the
multiplicative unit.
\end{Note}

\subsection{Semirings with a designated ghost ideal}

All of our structures fit into the framework of a semiring $R$
with a designated ideal $\tGz := \tG \cup \{ \rzero \},$ called
throughout the \textbf{ghost ideal}. Recall from \cite{golan92}
that an \textbf{ideal}  of a semiring $R$, denoted $A\triangleleft
R$, is a submonoid~$A$ of the monoid $(R,+)$ such that $ra$ and $
ar \in A$ for all $r \in R$ and $a\in A$.

In the following definition we consider $\tGz$ as a semiring in its own right, with neutral element  $\nu (\rone) $.

\begin{defn}\label{def:superTropical0} 
A \textbf{semiring with ghosts}   $\RGnu$ is a semiring $R$ (with
zero element   denoted as  $\rzero$) together with a semiring
ideal $\tGz $, called the \textbf{ghost ideal}, and
 an idempotent semiring homomorphism
$\nu : R  \to \tGz,$  called the \textbf{ghost map}, satisfying
\begin{equation}\label{eq:ghostMap}  a + a = \nu (a), \quad
\forall a\in R.
\end{equation}
\end{defn}

 From now on we formulate tropical concepts in
the language of supertropical semirings, in order to draw from the
structure theory of semirings (together with its parallels in ring
theory).

\begin{rem}  The notion of ideal is  standard in semiring theory.
   \cite[Proposition 9.10]{golan92} shows that an ideal $A$ of
$R$ is a kernel of a suitable homomorphism iff $A$ is
\textbf{subtractive}, which means that for every $a,b \in R$ such
that $a \in A$ and $a+b \in A,$ we must have $b\in A.$ Whereas one
often goes on to define a congruence and a quotient structure,
cf.~\cite[p. 68]{golan92}, this approach is not relevant to our
theory here, specifically for $\tGz$. Indeed, for any element
$a\in R$, we have both $2a = a+a \in \tGz$ and $2a+a \in \tGz,$ so
from this point of view, there is only one coset of $\tGz$, which
is all of $R$. The ghost ideal $\tGz$ is far from subtractive, and
we must abandon this aspect of classical semiring theory; the main
feature of this research is an alternative structure theory
utilizing the ghost ideal in a fundamental role.
\end{rem}

We are finally ready for the main definition of this paper.

\begin{defn}\label{def:superTropical} 
A \textbf{supertropical semiring} is a semiring with ghosts
$\RGnu,$ satisfying the extra properties, where we write $a^{\nu
}$ for $\nu(a)$:
\begin{enumerate} \ealph
 \item (Bipotence) $a+b  \in \{a,b\},\ \forall a,b \in R \; \text{such that } \; a^{\nu }
\ne b^{\nu };$ \pSkip
 \item (Supertropicality) $a+b   =  a^{\nu } \quad \text{if}\quad a^{\nu } =
 b^{\nu}$.
\end{enumerate}
\end{defn}

Note that Equation \eqref{eq:ghostMap}, a special case of
supertropicality,  implies that the ghost map $\nu$ is given by
$\nu(a) = a+a$.

\begin{rem} \label{tropprop}$ $
\begin{enumerate} \eroman
    \item It follows from Equation~\Ref{eq:ghostMap} that $ (ab)^{ \nu}
= ab + ab = (a+a)b = (a^\nu) b,$ and likewise $ (ab) ^{ \nu}  = a
(b^\nu).$ Thus, $(ab)^\nu  = (a^\nu b)^\nu  = a^\nu b^\nu.$ In particular, $b^\nu = 1^\nu b.$
\pSkip \item If $a^\nu = \rzero,$ then $a = a+ \rzero = \rzero^\nu
= \rzero.$ It follows that $\nu(R\setminus \{ \rzero \}) \subseteq
\tG.$
\pSkip \item The fact that $\tGz$ is an ideal of the supertropical
semiring $R$ is a formal consequence of the properties of the map
$\nu$.
 Indeed, if $a \in R$
and $b\in \tGz,$ then $ab = a (b^{ \nu})  = (ab) ^{ \nu}   \in
\tGz,$ and likewise $ba \in \tGz.$ ($\tGz$ is closed under
addition, by bipotence.) \pSkip \item If $a+b = \rzero$, then $a =
b = \rzero$. Indeed, if $a^ \nu \ne b^\nu ,$ then $a+b \in \{ a,
b\}$; let us assume that $a+b = a.$ Then $a=a+b = \zero,$ so $a+b
= b,$ and $a= b= \zero$.

We have shown that   $a^ \nu =  b^\nu $; but then $a^\nu =
a+b =  \rzero, $ implying $a = \rzero,$ by (ii).
\end{enumerate}
\end{rem}

Strictly speaking, one could have $\tGz = R,$ with $\nu$ the
identity map. In this case $R$ is an additively idempotent
semiring, such as the usual max-plus algebra. However, we view
this case as degenerate, and are much more interested in the case
where $\tGz$ is a proper ideal of $R$.

\begin{rem}[Universal characteristic] Supertropicality implies that $\tGz \supseteq \{nr: r \in R\}$, (where $nr = r + \cdots + r$ repeated $n$ times), for every
natural number $n
>1;$ more precisely, $a +a = a+a+a = \cdots =  a^\nu \in \tGz$,
$\forall a \in R$. Thus, $R$ might be expected simultaneously to
have properties of rings of every positive characteristic.
\end{rem}

This leads to a surprising fact.

\begin{prop}\label{lem:powOfPol0}
If $R$ is a supertropical semiring and $a,b \in R  $, then
 \begin{equation}\label{pow1}(a + b)^m  = a^m +  b^m, \quad \forall  m \in \Net.\end{equation}
\end{prop}

\begin{proof} We need to show that the only terms needed to compute $(a+b)^m $
are $a^m$ and $b^m$. Write
 \begin{equation*}(a + b)^m   = a^m  +
  \rone^\nu a^{m-1} b + \cdots +  \rone^\nu a^{1} b^{m-1} + b^m;\end{equation*}
then
\eqref{pow1} is clear if $a\nucong b$, since each side of
Equation \Ref{pow1} is then $(a^m)^\nu$, so we may assume that $a
\nug b $. Then  $ a^m \nug a^i b^j$ whenever $i+j = m.$ This means
that the single dominating term in the expansion of $(a + b)^m $
is $ a^m$; i.e.,
$$(a + b)^m  =a^m = a^m + b^m,$$ as
desired.
\end{proof}

\begin{rem} A suggestive way of viewing this proposition is to note that for any $m$
there is a semiring endomorphism $R\to R$ given by $a \mapsto
a^m,$ strongly reminiscent of the Frobenius automorphism in
classical algebra. This plays an important role in our theory.
\end{rem}

\subsection{Supertropical domains and semifields}

Suppose $\RGnu$ is a supertropical semiring.
\begin{defn}\label{dom} A \textbf{supertropical domain} $\RGnu$ is a supertropical semiring
 for which the following extra properties hold:

\begin{enumerate} \ealph
 \item  $\tT :=  R \setminus
\tGz$ is a (multiplicative) Abelian monoid; i.e., is closed under multiplication. \pSkip
  \item  The   restriction $\nu _{\tT} $ of $\nu$ to $\tT $ is onto; in other words, every element of $\tG$
has the form $a^\nu$ for some $a\in \tT $.

   \end{enumerate}
\end{defn}

Note that the ghost ideal $\tGz$ of a supertropical domain $R=
\RGnu$ determines both $\tT = R \setminus \tGz$ and $\tG = \tGz
\setminus \{ \rzero \}.$

\begin{rem} One can  obtain   condition (b) by replacing $R$ by
$\tT \cup \nu(\tT)\cup \{\rzero\}.$
\end{rem}

We call  $\tT$ the set of \textbf{tangible elements}; these
comprise one of our main focuses, since they lead us back to
tropical geometry. Ironically, for supertropical semirings in
general, the tangible elements are more complicated to define than
the ghost elements. In an arbitrary semiring with ghosts, the
definition of $\tT$ is much subtler, but we do not consider that
issue in this paper.
 Two elements of $R$ have the same \textbf{parity} if they are both
ghosts or both tangible.

\begin{notation}\label{nott:nu}
Viewing $\tGz$ as an ordered monoid, we write $a
 \nug b$ (resp.~ $a\nuge b$) to denote $a^\nu > b^\nu$ (resp.~ $a^\nu \ge b^\nu$ );
 we write $a \nucong b$ to denote that $a^\nu = b^\nu.$ We say $a$
 is $\nu$-maximal in $S \subset R$ if $a \nuge s$ for all $s \in
 S$.
\end{notation}

A major question in algebra is when two elements are equal.
Normally in a ring one determines whether $a=b$ by checking if
$a-b = 0.$ This simple procedure is no longer available in general
semirings, but in our supertropical setting we note  for tangible
$a,b \in R$ that $a \nucong  b$  iff $a+b \in \tG.$ This point of
view provides the motivation for the supertropical theory. We also define an
equivalence $\equiv$ by the rule $a\equiv b$ iff $a,b$ have the same parity with $a \nucong b$.
This means that either $a=b$ or $a,b \in \tT$ with $a \nucong b.$ (Thus, when $\nu$ is 1:1,
this reduces to equality.) Equivalent elements
are interchangeable in the sense that if $a \equiv b$ then $ac\equiv bc$ and $a+c \equiv b+c$ for all elements $c$.

Much of the theory can be carried out for $\tT$ not necessarily
Abelian, but this assumption is useful when we consider
factorization of polynomials (and is even more crucial for
studying matrices later on).

The mild condition that $\tT $ is a multiplicative monoid has some impressive
consequences.

\begin{rem}\label{cancell} Suppose that $R$ is a supertropical
domain.
\begin{enumerate} \eroman
    \item  $R$ is $\nu$-cancellative, in the sense that $ca \nucong cb $
for $c \ne \rzero$ implies $a \nucong   b .$ Indeed, since
$\nu_{\tT}$ is onto, we may assume that $a,b,c \in \tT.$ But then
$c(a+b) = ca+cb \in \tG$ by supertropicality, which contradicts
the fact that $\tT$ is a monoid unless $a+b \in \tG$; i.e., $a
\nucong b .$

In particular, the monoid $\tG$ is cancellative. \pSkip

    \item  If $ca \nucong db $ and $c \nucong d  \ne \rzero$, then $a \nucong   b .$ Indeed,
    $(ca)^\nu = (db)^\nu = d^\nu b^\nu = c^\nu b^\nu,$ so we are done by (i). \pSkip

 \item  If $ca \nug   cb,$
then $a\nug  b.$ Indeed, again we may assume that $a,b,c \in \tT.$
If $a \nucong   b,$ then $a+b \in \tG$, implying $ ca+cb =c(a+b)
\in \tG$,   contradicting the fact that $ca+cb = ca
 \in \tT$.
\pSkip

\item  The same argument shows that $R$ has cancellation over
$\Net,$ in the sense that $a^n \nucong b^n$ implies $a \nucong b.$
Indeed, again we may assume that $a,b \in \tT.$ But
Proposition~\ref{pow1} implies that $(a+b)^n = a^n + b^n \in \tG,$
implying $a+b \in \tG;$ i.e., $a \nucong b$. \pSkip

\item  It follows from (i) and (iv) (by applying $\nu$) that the
monoid $\tG$ is cancellative and also has cancellation over
$\Net$.\pSkip

\item $R$ is a commutative semiring. Indeed, any two elements of
$\tT$ commute, by definition; hence, $$a(b^\nu) = (ab)^\nu =
(ba)^\nu = (b^\nu) a; \qquad a^\nu b^\nu = (ab)^\nu = (ba)^\nu =
b^\nu a^\nu.$$
\end{enumerate}
\end{rem}

Let us tie supertropical domains to
the preliminary notions of the previous section.

\begin{rem}\label{ord1} Any valued Abelian monoid $(M,\tG,v)$ (as in Definition~\ref{def:valuedMonoid})  gives rise to the
supertropical domain $\TMGnu;$ cf.~Remark~\ref{add1}.

  Conversely, given a supertropical domain $\RGnu,$
we can recover the valued monoid $M = \tT = R\setminus \tGz$, and
$v = \nu|_{\tT} : \tT  \to \tG$ provides the \val. We view $\tGz$
as a cancellative ordered monoid ({\it but now with its operation being
written as multiplication}), under the following order:
$$g \geq h \text{ in }\tGz \quad \text{
iff }\quad g+h = g \text{ in }R.$$
To verify that ``\,$\ge$'' defines an order, note that the
properties of identity and antisymmetry are immediate; to check
transitivity, we note that if $g_1 \geq g_2$ and $g_2 \geq g_3$,
then the following equalities hold in $R$:
$$g_1 + g_3 = (g_1+g_2)+ g_3 = g_1 + (g_2+g_3) = g_1 + g_2 =
g_1.$$

We want to identify $(T(\tT ),  \tGz, \nu)$ with $\RGnu.$
Definition~\ref{def:superTropical0} gives us
Definition~\ref{def:valuedMonoid}, and Remark~\ref{add1} is seen
case by case, assuming $v(a)>v(b)$:
\begin{enumerate} \ealph
    \item
  $ a^{\nu }+  b^{\nu } = a^{\nu },$ so
 $$  (a+b)^\nu  =
a^{\nu } + b^{\nu } =  a^{\nu}  .$$
 Also $a+b=a,$ for if $a+b = b$, then $v(a) = a^{\nu }
=(a+b)^\nu  = b^{\nu }= v(b),$ contrary to assumption.
 \pSkip

    \item
Likewise, $a + b^{\nu } = a$, for if $a+ b^{\nu } = b^{\nu }, $
the same argument would show that
$$v(a) = a^\nu + b^{\nu } = \nu (a+b) = v(b),$$ contrary to
assumption.
\end{enumerate}
\end{rem}

The following computation is useful in later sections.
\begin{lem}\label{compaid} $ $
\begin{enumerate}\eroman \item If  $b ^2 = a  c$ in   a supertropical
semiring, then $a+c = a+b+c$. \pSkip

\item More generally (but over a supertropical domain), if $bc
\nucong ad$, then $ac +bd = (a+b)(c+d).$
\end{enumerate}
\end{lem}
\begin{proof} (i)  If $a \nug b$ or $c\nug   b$, then there is nothing to prove, so we may assume that
$a \le_\nu b$ and  $c  \le_\nu  b$. But then if $a <_\nu b$ or  $c
<_\nu  b$ we would have $ac <_\nu b^2$ by
Remark~\ref{cancell}(ii), contrary to hypothesis. Hence, $a
\nucong b \nucong c,$ implying
$$a+b+c  = a ^\nu = a+c .$$

(ii) By symmetry we may assume that $ad \ge_\nu bc$; we are done unless $ad \ge_\nu ac$ and $ad \ge_\nu bd,$ implying by
Remark~\ref{cancell} that $d  \nucong c$ and $a  \nucong b;$ hence
$$ac +bd = (ac)^\nu = a^\nu c^\nu = (a+b)(c+d).$$
\end{proof}

\begin{defn}\label{def:semifield}
A supertropical domain $\RGnu$ is called a \textbf{supertropical
semifield} if $\tT$ is a (multiplicative) Abelian group, i.e., if
every tangible element is invertible.
\end{defn}

Supertropical semifields play
 a basic role in our theory, analogous to the role of fields in
 linear algebra and algebraic geometry.

\begin{rem}
For any supertropical semifield $\RGnu$, $\tG$ is a group with
neutral element $\rone^\nu$ . Indeed, if $a^\nu \in \tG$ for $a
\in \tT ,$ then taking $b\in \tT $ such that $a b = \rone,$ we
have $a^\nu b^\nu = \rone^\nu;$ thus $a^\nu$ is invertible in
~$\tG$, as desired.

\end{rem}

  We usually designate a supertropical semifield as $F= \FGnu$,
 still denoting the ghost ideal as $\tGz.$

\begin{example} The   extended semirings of
Example~\ref{Exmp:valuedMonoid} all are supertropical
semifields.\end{example}

\subsection{Supertropical duality}\label{duality}
 For any supertropical domain $R = \RGnu,$ the set $R_+ = R \setminus
\{ \rzero \}$ is a semiring without zero, and one can define the
\textbf{dual semiring without zero} $\dual{R}_+$ to have the same
underlying set as $R_+$ with the same tangibles and ghosts, the
same ghost map $\nu$, and the same multiplication, but with
addition defined  by putting $a+b = b$ in $\dual{R}$ iff $a+b = a$
in~$R_+$, and $a+b = a^\nu$ in $\dual{R}$ whenever $a\nucong b$.
(This is well-defined in view of supertropicality.)

Formally adjoining a zero element $\dzero$ to $\dual{R}_+$ yields
a semiring which we call
  the
\textbf{supertropical dual} ~$\dual{R}$. The zero element $\rzero$
of $R$ has been treated separately since if we formally included
$\rzero$  in $\dual{R}$,  it would behave like $\infty$ rather
than like the zero element.
 $\dual{R}$  is a supertropical domain, seen by
combining Remarks~\ref{add1} and~\ref{add2}, and in fact is the
supertropical domain that matches the min-plus algebra in
\cite{SpeyerSturmfels2004}.


\begin{lem}\label{add4}
When $R$ is a supertropical semifield, so is its supertropical
dual $\dual{R}$, and moreover there is a semiring isomorphism
$\Phi: R \to \dual{R}$ given by $$\rzero \mapsto \dzero , \qquad
 a\mapsto a^{-1}, \quad a^\nu\mapsto (a^{-1})^\nu, \quad \text{ for all } a\in \tT.$$\end{lem}
 \begin{proof} Take $a,b \in \tT.$
If $a+b =a$, then $a \nuge b,$ so $a^{-1} \nule b^{-1},$ implying
$$\Phi(a+b) = \Phi(a)  = a^{-1} = a^{-1}+b^{-1} = \Phi(a)+ \Phi(b)
$$ in $\dual{R}.$ The same argument works with nonzero ghosts. For
$a  \in \tT,$ $b =\rzero,$ we have
$$\Phi(a+b)=a^{-1} = a^{-1} + \dzero  = \Phi(a) + \Phi(b) .$$
\end{proof}

This duality provides the reversals of polyhedral complexes in
tropical geometry, as described
in~\cite{IzhakianRowen2008Completion}.

\subsection{The divisible closure of a supertropical domain}

  Given any ordered
monoid  $(\tM,+)$ with cancellation over $\Net,$ one can form
 an $\mathbb
N$-divisible ordered monoid
$$\cltM = \left\{ \frac{a}m : a \in \tM, \ m \in \Net
\right\},$$ called the \textbf{divisible closure} of $\tM$; here $ \frac{a}m \le \frac{b}n$
iff $na \le mb.$ . The canonical map $\tM \to
\cltM$ given by $a \mapsto \frac{a}1 $ is 1:1. When
$\tM$ is a group,  $\cltM$ is then a
group which can be viewed as containing $\tM.$

We want to perform the same procedure for a supertropical domain
$R$, but now proceed using the semiring notation. We first note by
Remark \ref{cancell} that the ghost set $\tGz$ has cancellation
over $\Net.$ Viewing the ghost ideal $\tG$ as an ordered monoid as
in Remark~\ref{ord1} (with respect to multiplication), we form its
divisible closure $\cltG$, which we notate
$$\cltG = \left\{ \root m \of  a : a \in \tG, \ m \in \Net\right\}
.$$
 We formally define the $\Net$-localization
$$\clR = \bigg\{ \root m \of  a  : \ a \in R, \  m \in
\Net\bigg\};$$ here $ \root m \of  a  =  \root {m'} \of  b$ when
$a^{m'n} = b^{mn}$ for some $n\in \Net.$ (In the tropical
examples, using logarithmic notation, one would write
  $\frac am$  instead of $\root m \of a$.)

 Multiplication is
defined by
$$\root m \of  a  \root {m'} \of  b = \root {m  m'} \of  {a^{m'}b^m},$$ and
addition by
$$\root m \of  a + \root {m'} \of  b = \root {m  m'} \of {a^{m'}+b^m}.$$   We
extend $\nu : R \to \tG$ to a map $\clR \to \cltG$ by putting $\nu
(\root m \of  a ) =\root m \of {a^\nu}$, and call  $\clR$ the
\textbf{divisible closure} of~$R$.

 We say that $R$ is
\textbf{divisibly closed} if $\clR = R.$
For example,  $D(\Q)$ is
divisibly closed.

\begin{prop}\label{divcl1}  If $(R, \tG, \nu)$ is  a  supertropical domain, then $(\clR, \cltG, \nu )$ is also a supertropical domain, and there is a semiring homomorphism $R \to \clR$ given by $a \mapsto a$ (identifying  $ \root 1 \of  a$ with $a$) which is 1:1 on equivalence classes with respect to our equivalence relation $\equiv$. When $R$ is a supertropical
semifield, $\clR $ is also  a divisibly closed supertropical
 semifield.
 \end{prop}
 \begin{proof} The operations are clearly well-defined . For example, if
 $ \root m \of  a = \root {m'} \of  {a'}$ and    $ \root {n} \of  b = \root {n'} \of  {b'}$, then for some numbers $k.\ell$ we have
 $a^{m'k} = {a'}^{mk}$ and   $b^{n'\ell} = {b'}^{n\ell  }$, so $$(ab)^{m'n'k\ell} =(a'b')^{mnk\ell} ,$$
 implying  $ \root {mn} \of {a b} = \root {m'n'} \of  {a'b'}.$
 Clearly $\clR $ is the disjoint union of $\clM = \{\root m \of  a
: a \in \tT ,$ $m \in \Net\}$ and $\cltGz$, and
Proposition~\ref{lem:powOfPol0} shows that $\clR$ is a divisibly
closed supertropical domain.

It remains to show that if $ \root 1 \of  a\equiv \root 1 \of  b$ then $a\equiv b.$ But by definition $a^n \equiv  b^n$ for some $n$, implying $a \nucong b$
  by Remark~\ref{cancell}(iv), and clearly $a,b$ have the same parity, so $a\equiv b$.
\end{proof}

The reason for passing to the divisible closure is to enrich the
 structure by means of the following observation:

 \begin{rem}\label{divcl} If $R$ is divisibly closed, then $a^{m/n}$ is defined
 in $R$ for any $a\in R$ and any rational number~ $\frac m n$.
 \end{rem}

 \begin{rem}\label{loc1}  By the same token as in Proposition~\ref{divcl1},  in view of Remark~\ref{cancell}, one can formally localize
 the ghost elements of a supertropical domain to obtain a supertropical semifield, under which process  equivalence classes are preserved. This trick enables one to extend many of the results about
 supertropical semifields to supertropical domains.
 \end{rem}

\subsection{The semiring of continuous functions}\label{sec:func}

It is useful to introduce the following topology on $R$, obtained
from the order topology on $\tG$:

\begin{defn}\label{ordertop}  Suppose $\RGnu$ is a supertropical domain.
 Viewing $\tG$ as an ordered monoid, as in Remark~\ref{ord1},
we define the $\nu$-\textbf{topology} on $R$, whose open sets have
a base comprised of the \textbf{open intervals}
$$W _{\alpha, \beta} = \{ a \in R: \alpha < a^\nu <
\beta\}; \qquad W _{\a, \beta; \tT} = \{ a \in \tT: \a < a^\nu <
\beta\}, \quad \alpha, \beta \in \tG_\zero.$$ We also define
$\left[\alpha, \beta \right]= \{ a \in \tT: \a \le a^\nu \le
\beta\}.$  In other words, $\left[\alpha, \beta \right]$ is the
intersection of $\tT$ with the closure of $ W _{\a, \beta; \tT}$,
and we call it a  \textbf{tangible closed  interval}.
\end{defn}

 \begin{rem}\label{divcl1}
If $R$ is divisibly closed, then $\tT$ is \textbf{dense} in $R$ in the
sense that each nonempty open interval contains a tangible
element. (Indeed, if $\alpha, \beta \in \tT$ with $\alpha <_\nu
\beta ,$ then $\alpha^{1/2}\beta^{1/2} \in W _{\alpha, \beta},$ in
view of Remark~\ref{divcl}.)
\end{rem}

Here is an important semiring construction,  given in
\cite{golan99}.

\begin{defn} Given any set $S$ and semiring $R$, we  define the
semiring $\Fun (S,R)$ of functions from $S$ to~ $R$, under
pointwise addition and multiplication.  The \textbf{zero function}
$\zero_\Fun$ is given by $\zero_\Fun (\bfa) = \zero_R$ for all
$\bfa \in S.$ \end{defn}

\begin{rem}\label{subs} The map $f \mapsto f(\bfa)$ is a semiring homomorphism $\Fun
(S,R)\to R,$ for any fixed $\bfa\in S.$\end{rem}

\begin{rem}\label{ghost00} When $(R,\tGz, \nu)$ is a semiring with ghosts,
then the semiring $\Fun (S,R)$ also is viewed as a semiring with
ghosts, where a function $f \in \Fun (S,R)$ is said to be
\textbf{ghost} if
$$f(\bfa) \in \tGz \qquad \text{for every } \bfa
\in S.$$ The ghost ideal $\FunSGz$ of $\Fun (S,R)$ is the set of
ghost functions, and the ghost map $\nu$ is defined by $f \mapsto
f^\nu$, where by definition
$$f^\nu(\bfa) := f(\bfa)^\nu, \quad \forall \bfa \in S.$$
\end{rem}



\begin{prop}\label{lem:powOfPol}
If $f,g\in \Fun (S,R)  $, then $(f + g)^m =  f^m + g^m$ for any
positive $m \in \Net$.
\end{prop}

\begin{proof} In view of Proposition~\ref{lem:powOfPol0},
 \begin{equation}(f + g)^m (\bfa) = (f (\bfa)+ g (\bfa))^m = f (\bfa)^m + g (\bfa)^m
 = f^m (\bfa)+  g^m(\bfa)\end{equation} for each   $\bfa \in
S$.
\end{proof}

\begin{cor}\label{Prop:powOfPols}
If $f_1, \dots, f_k \in \Fun (S,R) $, then
$$\bigg(\sum _{i=1}^k f_i \bigg)^m =  \sum _{i=1}^k f_i^m ,$$
for any  positive $m\in \Net$.
\end{cor}

We also have duality:
\begin{rem}\label{add5}
The isomorphism $\Phi: F \to \dual{F}$ of Remark~\ref{add4}
extends to an isomorphism
$$\Phi_{\Fun}: \Fun (S,F) \To \Fun (S,\dual{F}) $$ given by $f \mapsto  f^\wedge,$
where $ f^\wedge(\bold a) = \Phi(f(\bold a))$. (Indeed, for
$f(\bold a), g(\bold a) \in \tT,$
$$(f^\wedge + g^\wedge)(\bold a)  = f(\bold a )^{-1} + g(\bold a
)^{-1} = (f + g)^\wedge (\bold a),$$ where ``+'' is taken in the
appropriate semiring; the other verifications are analogous.)
\end{rem}

Here is the case of special interest for us.  We write $R^{(n)}$
for the Cartesian power $R^{(n)}$ of $n$ copies of~$R$. The
semiring with ghosts $$\FunR := (\FunR, \FunGz,\nu_{\Fun})$$    is
{\textbf{not}} a supertropical semiring for $n>1$, since bipotence
fails. In analogy with Remark~\ref{subs}, we have:

\begin{prop}\label{spec1} Given $a \in R$, define  $$\Phi_a : \FunR
\to \Fun (R^{(n-1)},R)$$ by sending $f \mapsto f_a$, where
$$f_a(a_1, \dots, a_{n-1}) = f(a_1, \dots, a_{n-1},a).$$
Then $\Phi_a $ is a semiring homomorphism.
\end{prop}
\begin{proof} Write $\bfa = (a_1, \dots, a_{n-1})$ and $\bfa(a) =  (a_1, \dots, a_{n-1},a)$. Then $$\Phi_a (f+g)(\bfa) =
(f+g)(\bfa(a)) = f(\bfa(a)) +g(\bfa(a)) = \Phi_a (f) (\bfa)+\Phi_a
(g)(\bfa),$$ yielding $\Phi_a (f+g)= \Phi_a (f)  +\Phi_a (g);$ the
verification for multiplication is analogous, and  $\Phi_a
(\rzero) = \rzero$.
\end{proof}

Later on, we consider $$\ker \Phi_a = \{ f \in \FunR: f(a_1,
\dots, a_{n-1},a) = \rzero : \forall a_i \in R\}.$$ This is a
rather restrictive view of the kernel, and is to be weakened in
subsequent research.

 Let us bring in  the $\nu$-topology.

\begin{definition} $\CFunR$ is the  sub-semiring with ghosts, comprised of functions in the semiring
$\FunR$ which are continuous with respect to the $\nu$-topology of
Definition \ref{ordertop}.
\end{definition}

 $\operatorname{CFun}(R^{(n)},R )$ plays a very important role in
 this paper.

 \begin{defn}\label{box} Given  $\bfa = (a_1, \dots, a_n) \in R^{(n)}$ and
 $\bt_1, \dots, \bt_n \in \tT$ with each $\bt_i \nug \rone,$ the \textbf{closed} $\bfa$-\textbf{box}
 is defined as the  product of closed tangible intervals (cf.~Defintion~\ref{ordertop})$$
 \left[\frac {a_1}{\bt_1},\, \bt_1a_1 \right]\times   \left[\frac {a_2}{\bt_2},\,
 \bt_2a_2 \right]\times \dots \times   \left[\frac {a_n}{\bt_n },\,
 \bt_na_n\right] \subset \tT^{(n)}.$$\end{defn}

\begin{defn}\label{arch} The semifield $F$ is \textbf{archimedean}, if for every $a
\nug \fone$ and $b$ in $F$, there is suitable~$m$ such that $a^m
\nug b.$\end{defn}

 This guarantees that $F$ (and thus $\tT$) has ``large enough''
and ``small enough'' elements.

\begin{rem} If $\RGnu$ is a supertropical domain, then $R^{(n)}$ is
endowed
 with the usual product topology (obtained from the $\nu$-topology on
 $R$). When  $\RGnu$ is an archimedean supertropical semifield, the closed boxes comprise a sub-base for the closed sets of the relative
  topology on $\tT^{(n)}.$
 \end{rem}

\subsection{Radical ideals and prime ideals of semirings}

\begin{defn}  Suppose $A \subset R $.
 The \textbf{radical} $\sqrt{A}$ is defined as $\{ a \in R : a^k\in A$ for some $k \}$. An
 ideal~$A$ of $R$ is \textbf{radical} if
$A = \sqrt{A}.$
\end{defn}

\begin{rem}\label{surprise1} If $A$ is an ideal of a commutative semiring $R$, then
$\sqrt A\triangleleft R$, by the usual ring-theoretic argument.
More surprisingly, if $R$ is a commutative supertropical
\emph{semiring} and $A$ is a sub-semiring of $R$, then $\sqrt A$
is also a sub-semiring of $R$, by Proposition \ref{lem:powOfPol0};
by the same reasoning, if $W$ is a sub-semiring of
$\operatorname{Fun}(R^{(n)},R )$, then $\sqrt W$ is also a
sub-semiring of $\operatorname{Fun}(R^{(n)},R )$, by
Corollary~\ref{Prop:powOfPols}.
\end{rem}

The following definition is also lifted from ring theory.

\begin{defn}  An ideal $P$ of a semiring $R$ is \textbf{prime} if
it satisfies the following condition:
 $$\text{ $AB\subseteq P$  for $A,B
\triangleleft R$ implies $A\subseteq P$ or $B\subseteq P.$}$$
\end{defn}

\begin{prop}\label{prime5} Every radical
  ideal $A$ of a commutative semiring $R$
 is the intersection of prime  ideals. \end{prop}
\begin{proof} We copy the standard argument from commutative algebra.
For any element $b\notin A,$ take an ideal $P$ maximal with
respect to $b^k \notin P$, for each $k \in \N.$ Then $P$ is a
prime ideal, since if $a_1a_2\in P$ with $a_1,a_2 \notin P$, then,
for $i = 1,2$ the ideal $P+ Ra_i$   properly contains $P$, and
thus
  contains a power $b^{k_i}$ of $b;$ Letting $k = k_1 + k_2$ we see that
  $(P+Ra_1)(P+Ra_2)\subseteq P$ contains $b^k$, contradiction.
\end{proof}

\subsection{Ghost-closed ideals}

 \begin{defn}\label{tropid} A \textbf{ghost-closed ideal} of a semiring $R =\RGnu$
 with ghosts is a semiring
 ideal containing the ghost ideal $\tGz$.\end{defn}

Clearly, a supertropical domain $(R,\tGz, \nu)$ is a supertropical
semifield iff it has no proper ghost-closed ideals other than
$\tGz$. This is one reason why we focus on ghost-closed ideals.

\begin{example}\label{tropid1}
The ghost ideal $\FunGz$ is itself a radical ideal of the
supertropical   semiring $\FunR$.
  Indeed, if $f^m(\bold a) \in \tGz,$ then  $f(\bold a) \in \tGz.$   By
Proposition~\ref{prime5}, $\FunGz$  is the intersection of prime
ideals, each of which clearly is ghost-closed. \pSkip
\end{example}


\begin{defn} The ghost-closed ideal $\langle S \rangle$
\textbf{(classically) generated} by a set $S$ is the intersection
of all ghost-closed ideals containing $S$ (or, in other words, the
ideal generated by $S$ and $\tG$).\end{defn}

There is a weaker version of this definition, called ``tropical
generation,'' which although more appropriate to the tropical
theory is  more technical; in this paper we focus on classical
generation in order to obtain more precise information about the
ideals in question.

\subsection{Supertropical divisibility and the supertropical radical}

We say that $a=b + \, \text{ghost}$ in a semiring~$R$ with ghosts
when $a = b + c$ for some ghost element $c \in \tGz;$ in this
case, we write $a\lmodg b.$ This relation arises naturally in many
supertropical contexts, including the following.

\begin{defn}\label{superdiv} In any semiring $R$, an element $b\in R$ \textbf{divides} $a\in R$
if $a=qb$ for some
    $q\in R$.   For $R$ a   semiring with ghosts, an element $b\in R$ \textbf{supertropically divides} $a\in R$
if $a \lmodg qb$ for some
    $q\in R$.
\end{defn}

\begin{defn}  Suppose $A \subset R $.
 The \textbf{supertropical radical} $\tsqrt{A}$ is defined as
 the set $$\{ a \in R : \text{ some power}  \  a^k \   \text{is supertropically divisible by
 an element of } A \}.$$
  An ideal $A$ of $R$ is \textbf{supertropically radical} if
$A = \tsqrt{A}.$
\end{defn}

\begin{rem}\label{surprise1} If $A$ is an ideal of a commutative semiring $R$ with ghosts, then
$\tsqrt A = \sqrt{A+\tGz} \triangleleft R$.  It follows at once
that every supertropically radical ideal of a commutative semiring
$R$ with ghosts is the intersection of ghost-closed prime ideals
of $R$ (and vice versa).
\end{rem}

By the same sort of argument, in analogy with
Remark~\ref{surprise1}, if $R$ is a commutative supertropical
semiring and $A$ is a sub-semiring of $\FunR$, then $\tsqrt A$
is also a sub-semiring of $\FunR$.

\section{Polynomials}

\begin{defn}\label{def:poynomials} Given any semiring $\RGnu$ with ghosts, we define the semiring $(R\pl \la\pr , \tG \pl
\la \pr, \nu)$ of \textbf{polynomials} $$\left\{ \sum _{i \in \N }
\a _i \la ^i: \quad \a_i \in R, \quad \text{almost all }\a _i =
\rzero,\right\}$$ where we define polynomial addition and
multiplication in the familiar way:
$$\left(\sum _i \a _i\la
^i\right)\left(\sum_j \beta _j \la ^j\right) = \sum _k \left(\sum
_{i+j=k} \a _i \beta _{j}\right) \la ^k.$$
\end{defn}

We have denoted the semiring of polynomials as $R\pl \la\pr$
rather than by the familiar notation $R[\la]$. The reason is that,
as we shall see, different polynomials can take on identically the
same values as functions, and we want to reserve the notation
$R[\la]$ for the equivalence classes of polynomials (with respect
to taking on the same values as functions). But before discussing
this issue, let us develop some more notions.

 We write a polynomial  $f = \sum_{i=0}^t \a _i \la ^i$ as a sum of \textbf{monomials} $\a_i\la^i$,
where $\a _t \ne \rzero$ and $\a _i = \rzero$ for all $i>t$, and
define its \textbf{degree}, $\deg f$, to be $t$.  By analogy, we
sometimes write $\la ^\nu$ for $\rone ^\nu \la.$ A~polynomial is
\textbf{monic} if its leading coefficient is $\rone.$ A polynomial $f$ is
said to be \textbf{tangible} if its coefficients are all tangible.
 We identify $\a_0\la^0$ with $\a_0$, for each $\a_0\in R.$ Thus, we
may view $R \subset R\pl \la\pr .$ Often we use logarithmic
notation for the coefficients of polynomials over $\Trop$;
  $\la$ then means $0 \la +(-\infty).$

Since the polynomial semiring was defined over an arbitrary
  semiring, we can define inductively $R\pl \la_1,
\dots, \la _n\pr  =R\pl \la_1, \dots, \la _{n-1}\pr \pl \la _n\pr
.$ Often we write $  \Lm $ for $ \{\la_1, \dots, \la _n\} $.

\begin{defn}
In particular, we define the polynomial semiring $R\pl \Lm\pr =
R\pl \la_1, \dots, \la _n\pr $ in $n$~indeterminates over a
supertropical semiring $R$. Any such polynomial can be written
uniquely as a sum
$$f = \sum _{i_1, \dots, i_n}\a_{i_1, \dots, i_n}\la _1^{i_1}\cdots \la _n^{i_n},$$ which we
denote more concisely as $\sum _\bfi\a _\bfi \Lm ^\bfi,$ where
$\bfi$ denotes the $n$-tuple $(i_1, \dots, i_n)$ and $\Lm ^\bfi$
stands for $\la _1^{i_1}\cdots \la _n^{i_n}$. We  write $\deg_k \a
_\bfi \Lm ^\bfi = i_k$ for $1 \le k \le n$. The \textbf{support}
of $f$ is
$$\supp(f) = \{ \bfi : \a_\bfi \ne \rzero \}.$$

A \textbf{binomial} is the sum of two monomials.
\end{defn}

Binomials play the key role in this theory, because, as we shall
see, tangible roots often can be defined in terms of binomials.

We sometimes write $f(\la_1, \dots, \la _n
)$ for a polynomial $f \in R\pl \Lambda\pr,$ indicating that it
involves the $n$ indeterminates  $\la_1, \dots, \la _n.$

\begin{rem} If $F$ is a supertropical semifield, then $\{ f \in F\pl\Lm \pr: f$ is not a tangible constant$\}$ is the unique
maximal ideal of $F\pl\la_1, \dots, \la _n\pr,$ comprised of all
the noninvertible elements, and it is a ghost-closed prime ideal.
\end{rem}

\subsection{The polynomial semiring (as functions)}

A more concise way of viewing polynomials is inside the larger
semiring $\operatorname{CFun}(R^{(n)},R )$ of \S~\ref{sec:func}.

\begin{rem}\label{fun2} There is a natural
semiring homomorphism
$$\Psi : R\pl \Lambda \pr \to
\operatorname{CFun}(R^{(n)},R ),$$ obtained by viewing any
polynomial $f \in R \pl \Lambda \pr$ as the (continuous) function
sending $(a_1, \dots, a _n) \mapsto f(a_1, \dots, a _n).$

In classical commutative algebra, when $R$ contains an infinite
field, $\Psi$ is 1:1, by the easy part of the fundamental theorem
of algebra. Thus, one always can ``make'' $\Psi$ 1:1 by enlarging
$R$ suitably. But in our supertropical setting, different tropical
polynomials could always represent the same function, i.e., take
on the same values at each element of $R$.

For example, take elements $\a, \bt$ in a supertropical semifield
$R$, for which $\bt  \nug \a^2.$ The polynomials $\la^2 + \a \la +
\bt$ and $\la ^2 + \bt$ define the same function. Indeed,
otherwise there is $a\in R$ such that $\a a$ has $\nu$-value at
least both that of $a^2$ and $\bt.$ But $a^2  \le  _\nu \a a $
implies $a  \le  _\nu \a,$ and thus
$$\bt   \le _\nu \a a \le _\nu  \a ^2 ,$$
contrary to hypothesis. This argument did not depend on any other
properties of $R$, and thus shows that $\Psi$ is not 1:1 over any
semifield containing $R$, as opposed to the classical situation.

\end{rem}

$ $From now on, we work with
$$R[ \Lambda ] :=  \Psi ( R\pl \Lambda \pr);$$ i.e.,
in $\operatorname{CFun}(R^{(n)},R )$, whose ghost ideal (as
observed in \S\ref{sec:func}) is $\operatorname{CFun}(R^{(n)},\tGz
)$. Thus, $R[ \Lambda ]$ can be viewed as a semiring with ghost
ideal consisting of all polynomials which as functions take on
only ghost values.

\subsection{Equivalence of polynomials, and essential polynomials}

 As noted above, the semiring of polynomials over a supertropical
semifield is not supertropical, and, even worse, the tangible
 polynomials are not closed under multiplication; for
example $(\la +3)^2 = \la ^2 + 3^\nu \la + 6,$ which has a ghost
term. (Recall that our examples are computed in logarithmic
notation.)
Accordingly,  we need another definition to enable us to consider
polynomials over $R$ in $\CFunR,$ i.e., as continuous functions
from $R^{(n)}$ to $R$.

\begin{defn}\label{def:R-equivalent}
 Two polynomials $f,g\in R\pl \Lambda\pr $  are $\bf{e-}$\textbf{equivalent},
 denoted as $f \eqR g$,  if $f(\bfa)  = g(\bfa)$ for any
 tuple
  $\bfa = (a_1,\dots,a_n) \in
 R^{(n)}$. (In other words, polynomials $f$ and $g$ are $e$-equivalent iff
$\Psi(f) = \Psi(g)$.)

 Two polynomials $f,g\in R\pl \Lambda\pr $  are \textbf{weakly} $\bf{(\nu,e)-}$\textbf{equivalent},
 denoted $f\eqRnu g,$ if they identically take on $\nu$-equivalent values,
 i.e., $f^\nu \eqR g^\nu$, or, explicitly,
 $f(\bfa) \nucong  g(\bfa)$ for any $\bfa = (a_1,\dots,a_n) \in R^{(n)}$.
 Weakly $\bf{(\nu,e)-}$equivalent polynomials $f,g\in R\pl \Lambda\pr $ are $\bf{(\nu,e)-}$\textbf{equivalent} if
 $f(\bfa)$ and  $g(\bfa)$ have the same parity, for all $\bfa \in
 R^{(n)}.$
 \end{defn}

 \begin{Note} $ $
\begin{enumerate} \eroman
 \item  Polynomials of different degree over a supertropical
 semifield cannot identically take on $\nu$-equivalent values. Thus,
 $(\nu,e)$-equivalent polynomials (and, a fortiori, $e$-equivalent polynomials) have the same degree. \pSkip

   \item
 The difference (for tangible  polynomials) between $e$-equivalent
 and $(\nu,e)$-equivalent  only arises
 when the restriction $\nu_{\tT} $ of $\nu$ to $\tT $ is not 1:1. Since $\nu _{\tT} $ is 1:1 in the ``standard'' tropical example~$D(\tG)$, this distinction only exists in  the more
 unusual examples, such as $(R^\times, R^+, \nu)$ where
 $\nu$ is the absolute value; here $\la +2$ and $\la + (-2)$ are
$(\nu,e)$-equivalent but not $e$-equivalent.   We may resort to
this example when $(\nu,e)$-equivalence
 comes up, but  we focus on $e$-equivalence whenever possible, indicating how the
 theory simplifies when $\nu_{\tT} $ is 1:1. \pSkip

\item One can  reduce an arbitrary supertropical domain $R$ to the
case when $\nu_{\tT} $ is 1:1. Namely, $\nucong$ restricts to an
equivalence $\sim$ on $\tTz  = R \setminus \tG$;  Then $(\tTz/\!\!
\sim ) \, \cup \,\tG$ becomes a supertropical domain
 under the natural
operations of the equivalence classes (and this
 can be identified
with $D(\tG)$).
\end{enumerate}
\end{Note}

\begin{example}\label{exmp:eq} The following facts hold for all $a,b \in
\tT $, $a \neq b$: \pSkip
\begin{enumerate} \eroman
    \item
 $\lm + a \neqR \lm + \uuu{a}$,  although $\lm + a \eqRnu \lm +
\uuu{a};$ \pSkip
\item  $\lm + a \neqR \lm + b;
 $ \pSkip
\item $(\lm + a)^2 \eqR \lm^2 + a^2$ (a special case of
Proposition \ref{lem:powOfPol}).
\end{enumerate}
\end{example}

Let us introduce a natural representative for each e-equivalence
class.

\begin{defn}  The function $f\in
\FunR$ \textbf{dominates} $g$ if $ f(\bfa) \nuge g(\bfa)$ for all
$\bfa \in R^{(n)};$   $f\in \FunR$ \textbf{strictly dominates} $g$
if $ f(\bfa)\nug g(\bfa)$ for all $\bfa  \in R^{(n)}.$

(Thus, when $f$ dominates $g$, $f(\bfa) + g(\bfa) \in \{f(\bfa),
f(\bfa)^{\nu}\}$ for all $\bfa  \in R^{(n)};$ when $f$ strictly
dominates~$g$, $f  + g= f.$ )
\end{defn}

\begin{definition}\label{def:essentialPart} Suppose $f = \sum \a _\bfi \Lambda ^\bfi,$
$h = \a_\bfj \Lambda ^\bfj$ is a monomial of $f$, and write $f_h =
\sum _{\bfi \ne \bfj}\a _\bfi \Lambda ^\bfi.$ The monomial $h$ is
\textbf{inessential in} $f$ iff $f_h$ dominates $h$. The
\textbf{essential part} $\ef$ of a polynomial
 $f = \sum \a _\bfi \Lm^\bfi$ is the sum of those monomials
$\a_\bfj \Lm ^\bfj$ that are essential, while its inessential part
$\iif$ consists of the sum of all inessential monomials $\a_\bfi
\Lm ^\bfi$. The polynomial $f$ is said to be an \textbf{essential
polynomial} when $f = \ef$.
\end{definition}

 The following
equivalent formulation indicates the direction we wish to take:

\begin{rem} $ $
\begin{enumerate} \eroman
    \item A monomial $h$ is essential in a polynomial $f$ iff
$h(\bfa)\nug f_h(\bfa)$ for some $\bfa$ and thus for all $\bfa'$
in an open set $W_\bfa$ of $\bfa$. \pSkip
    \item

Any monomial $h$ of $\ef$ is essential in $\ef$. Indeed, by
definition, $ f_h(\bfa)+ h(\bfa)  \nug f_h(\bfa) $ for some
$\bfa\in R^{(n)},$ implying $ h(\bfa) \nug  f_h(\bfa) $. A
fortiori, this implies $ h(\bfa) \nug  \ef_h(\bfa) $.
\end{enumerate}

\end{rem}

We want the essential part of a polynomial $f$ to be
$e$-equivalent to $f$. Towards this end, we turn to the divisible
closure.

\begin{rem}\label{arch1} Any archimedean supertropical  semifield $F$  ( in the sense of Definition~\ref{arch})
satisfies the following property:
\\
 For any nonconstant monomials $g_1,g_2,h_1,
\dots, h_m$ and $\bfa\in F^{(n)}$ with
$$g_2(\bold a) \ \nucong \  g_1(\bold a) \ \nug \
 h_i(\bold a), \quad 1 \le i \le m,$$ and any open set $W_\bfa$ of $F^{(n)}$ containing $\bfa,$ there exists $\bold
a'\in W_\bfa$ with $$g_2(\bold {a'}) \ \nug \  g_1(\bold {a'}) \
\nug \
  h_i(\bold {a'}), \quad 1 \le i \le m.$$\end{rem}

\begin{lem} Suppose the supertropical
  semifield $F$ is  archimedean. For any monomials $g_1,\dots,
g_\ell$, $h_1, \dots, h_m$ and $\bold a \in F^{(n)}$ with
$$g_1(\bfa) \ \nucong \ g_2(\bfa) \ \nucong \ \dots  \ \nucong \ g_\ell(\bfa) \
\nug  h_i(\bfa), \quad 1 \le i \le m,$$ there exists  $\bfa' \in
F^{(n)}$ and $1<j \le \ell$ such that $$g_j(\bfa') \ \nug \
g_i(\bfa')  \quad \forall i \ne j; \qquad g_j(\bfa') \ \nug
h_i(\bfa'), \quad 1 \le i \le m.$$
\end{lem}
\begin{proof} Induction on $\ell$. By Remark \ref{arch1}, we have
$\bfa'\in F^{(n)}$ such that $$g_2(\bfa') \nug g_1(\bfa') \nug
h_i(\bfa'), \quad 1 \le i \le m.$$ Take $j$ such that $g_j(\bfa')$
is $\nu$-maximal, and expand the $h_i$ to include all $g_i$ such
that $g_j(\bfa') \nug  g_i(\bfa').$ Then we have the same
hypothesis as before, but with smaller $\ell$.
\end{proof}
\begin{prop} If $F$ is a archimedean supertropical  semifield, then $\ef \eqR f$
for any $f\in F\pl \Lambda\pr $.
\end{prop}
\begin{proof} Given any $\bfa \in F^{(n)}$, there is a monomial
$g_1$ such that $f(\bfa) \nucong  g_1(\bfa).$ We need to show that
$\ef(\bfa) \nucong g_1(\bfa).$ This is clear unless $g_1(\bfa)
\nucong g_2(\bfa) \nucong  \cdots \nucong g_\ell(\bfa) \nug
\ef (\bfa)$ for some other monomial(s) $g_2, \dots, g_\ell$ of $f$
that are inessential in $f$. But then, by the lemma, we may find
$\bfa'$ such that $g_j(\bfa')^\nu$ takes on the single largest
$\nu$-value of the monomials of $f$, for some $2 \le j \le \ell,$
contrary to $g_j$ being inessential in $f$.
\end{proof}

Now we have a new way of viewing the  polynomial semiring
$F[\Lambda]$.

\begin{rem}\label{functional} For $F$ archimedean, the  supertropical polynomial semiring $F[\Lambda]$
can be viewed as the collection of   essential polynomials, viewed
as a semiring whereby we perform the usual operations in $F\pl
\Lambda\pr$ and then take the essential part. The ghost ideal is
comprised of those essential polynomials whose coefficients are
all in $\tGz.$

If $f_1$ dominates $f_2$, then obviously $f_1 +g$ dominates
$f_2+g$ and $f_1g$ dominates $f_2g,$ for any polynomial~$g$.
Accordingly, one can discard the inessential monomials at any
stage of the computation, which shows that our new operations of
addition and multiplication in $F[\Lambda]$ remain associative and
distributive.
\end{rem}

We write $f \lmodgLa g$ when $f\eqR g+ h$ for $h \in
\tGz[\Lambda].$ Occasionally we only need $f\eqR g+ h$ as
functions on a given subset $S \subseteq R^{(n)}$; then we say $f
\lmodgLa g$ on $S$.

The following definition is easily seen to be a special case of
Definition~ \ref{ghost00}. Since we are viewing polynomials as
functions, we consider only essential polynomials.

\begin{defn} The \textbf{tangible part} $\vf$ of an essential polynomial $f =
\sum \a _\bfi \La^\bfi$ is defined as the sum of those $\a_\bfi
\La ^\bfi$ for which $\a _\bfi $ is tangible;  the \textbf{ghost
part} $\gf$ of $f$ is the sum of those $\a_\bfi \La ^\bfi$ for
which $\a _\bfi \in \tG$. \end{defn}

 Thus, any essential polynomial $f$ is
written uniquely as the sum of its tangible part $f^\tng$ plus its
ghost part $f^\gst$.  We say that a polynomial is
\textbf{essential-tangible} if its essential part is tangible.

\begin{prop}\label{thm:prodOfIndiviuals2} If  $R = \clR$, then
the   product $q = fg$ of two essential-tangible polynomials $f,g$
in $R \pl \Lambda \pr$
 is essential-tangible.
\end{prop}
\begin{proof}
Assume $q = fg$ is the product of two essential-tangible
polynomials
  $ f = \sum \al_\bfi \Lm^\bfi$ and  $ g
= \sum \bt_\bfj \Lm^\bfj$. Write $f = \ef + \iif$, and $g = \eg +
\ig$; clearly, $\iif \ig$, $\iif g$ and $f \ig$ belong to $\iq$,
and $\eq \eqR \ef\eg$. Thus, a ghost monomial $h$ of $\eq$, if it
existed, would be obtained from some two (or more) identical
products
\begin{equation}\label{eq:ghostMono} \al_{\bfi} \Lm^{\bfi} \bt_{\bfj}
\Lm^{\bfj} = \al_{\bfh} \Lm^{\bfh} \bt_{\bfk} \Lm^{\bfk};
\end{equation}
But these are dominated by $ \al_{\bfi} \Lm^{\bfi}\bt_{\bfk}
\Lm^{\bfk} +\al_{\bfh} \Lm^{\bfh} \bt_{\bfj} \Lm^{\bfj} $, in view
of Lemma~\ref{compaid}(ii), so $h$ is inessential.
\end{proof}

\section{Roots of polynomials}\label{rootsofpol}

As in classical algebra, our main interest in polynomials lies in
their roots, which are to be defined in the tropical sense. As
mentioned earlier, in our philosophy, ghost elements are to be
treated like zero.

 \begin{defn} (Compare with \cite{RST}) Suppose $R=\RGnu$ is a supertropical domain.
 An element ${\bfa }\in R^{(n)}$ is a (tropical) \textbf{root}
of a polynomial $f\in R\pl \la_1, \dots,\lm_n \pr$ if $f({\bfa})
\in \tGz$; in this case we also say $f$ \textbf{satisfies} $\bfa$.
 The root $\bfa = (a_1, \dots, a_n)$ is
\textbf{tangible} if each $a _i$ is tangible or $\rzero$; $\bfa =
(a_1, \dots, a_n)$ is \textbf{strictly tangible} if each $a _i$ is
tangible.
\end{defn}
\noindent

For example, any ghost $a^\nu$ is a root of the monomial $\la$,
and $\lm$ has no strictly tangible roots; any tangible constant
$\ne \rzero$ has no roots. On the other hand,  every element of $R$
is
 a root of all ghost polynomials.

 \begin{Note} Of course, a tangible polynomial could take on some non-tangible
 values. For example, the tangible polynomial $f = \la
 +1$ satisfies $f(1) = 1^\nu \in \tG.$ This is precisely the idea behind roots of a tangible polynomial.
  \end{Note}

 Of course, if $f \in R[\la]$ and $f(\bfa) = \rzero$, then $\bfa$ is a root of $f$.
 Although
 this is usually much too special to be of use, it does help us keep track of monomials.
 Note by Remark~\ref{tropprop}(iv) that  $f(a) = \rzero$ iff $a = \rzero$
 and $\la | f.$ (Otherwise, some monomial of $f$ would take a
 nonzero value.) Let us generalize this observation.

 \begin{prop}\label{root10} Define  $\Phi_a : R[\la_1, \dots,
 \la_n]
\to R[\la_1, \dots,
 \la_{n-1}]$ as in Proposition \ref{spec1}, sending $f \mapsto f_a$, where
$$f_a(\lm_1, \dots, \lm_{n-1}) = f(\lm_1, \dots, \lm_{n-1},a).$$ If $f\in \ker\Phi _a,$ then $\la_n$ divides
$f$.\end{prop}
 \begin{proof} We are given $f(a_1, \dots, a_{n-1}, a) = \rzero$ for
 all $a_1, \dots, a_{n-1} \in R.$ But writing $$f (\la_1, \dots, \la_{n}) =
 \sum_i
 h_i$$ as a sum of monomials, we can view
 $$f_a (\la_1, \dots, \la_{n-1}) =   \sum
 h_i(\la_1, \dots, \la_{n-1},a)$$ as a sum of monomials, yielding $$\sum_i h_i (a_1, \dots,
 a_{n-1} ,a) = \rzero, \quad \forall a_j \in R$$ which by Remark~\ref{tropprop}(iv) implies that
 each $h_i (a_1, \dots,
 a_{n-1},a) = \rzero$ for all $a_j;$ i.e., each $\Phi_a(h_i) = \rzero.$
 In other words, $\la_n| h_i$ for each $i$, implying $\la_n | f.$
\end{proof}

\begin{lem}\label{divide1} If $\la_j$ divides a polynomial $f =
\sum h_i,$ where the $h_i$ are monomials, then $\la_j | h_i$ for
each $i$.
\end{lem}
\begin{proof} Write $\overline{\phantom{w}}$\, for the specialization $\la_j \mapsto
\rzero$. Then $$\rzero = \overline{f} = \sum \overline{h_i},$$ by
Proposition~\ref{spec1}. Applying Remark~\ref{tropprop}(iv) yields
each $\overline{h_i} = \rzero,$ so $\la_j$ divides each $h_i$.
\end{proof}

\begin{prop}\label{divide2} If $\la_j$ divides $\sum g_i,$ a sum
of polynomials, then $\la_j$ divides each $g_i.$
\end{prop}
\begin{proof} Write each $g_i$ as a sum of monomials; by the
lemma, $\la_j$ divides each of these monomials, and thus divides
each $g_i.$
\end{proof}

\subsection{The fundamental theorem of supertropical algebra}

We return to our general considerations about roots.

\begin{rem} Obviously, any $e$-equivalent polynomials have precisely the same
roots.
\end{rem}

\begin{rem} Any root
$\bfa \in R^{(n)}$ of $\vf$ is a root of $f$. Indeed, $f(\bfa) =
\vf (\bfa) + \text{ghost}\in \tGz.$
\end{rem}

 One classical result, the
\bfem{Fundamental Theorem of Algebra}, has a very easy analog
here. We work over a divisibly closed supertropical semifield $F =
\FGnu;$  i.e., $F= \clF$.

\begin{lem}\label{lem:cl1} Suppose $F = \clF$. Then for
any nonconstant polynomial $f \in F\pl \lm\pr $  and for any
$\ra^\nu \ne \fzero$ in~$\tG$, there exists tangible $ r \in F$
with $f(r) \nucong \ra.$
\end{lem}
\begin{proof}  Write $f = \sum \a_i \la^i.$ For each $k>0$,
there is some tangible $r_k$ such that
$$r_k \nucong  \root k \of {\frac {\ra}{\a_k}};$$ thus,  $ \a_k
 r_k^k \nucong  \ra.$ Take $r$   such that $ r^\nu $
is minimal among these $r_k^\nu$. 
Then $ f(r)  \nucong  \ra.$
\end{proof}

\begin{prop}\label{fundam} Over any divisibly closed supertropical semifield $\FGnu$, every   polynomial $f \in F\pl
\lm\pr $,  which is not a tangible  monomial, has a strictly
tangible root.\end{prop}
\begin{proof} Suppose $f(\lm) = \sum _{i=u}^m \a_i \la ^i,$ where
$\a_u \ne \zero_R.$ Replacing $f$ by $\sum _{i=u}^m \a_i \la
^{i-u},$ and renumbering the coefficients, we may assume that
$\a_0 \ne \zero_F.$
 Write $g(\lm) =  \sum _{i=1}^m \a_i \la ^i,$ so $f(\la) = g(\la)
+ \a _0.$ If $\a_0 \in \tGz $, then we could erase $\a_0$ and
divide by $\la,$ and conclude by induction on $\deg f$. Thus, we
may assume that $\a_0 \notin \tGz$. By the lemma, there is some
tangible $r$ such that $g (r) \nucong \a_0,$ implying $ f(r)=
\a_0^\nu + \a _0 \in \tG.$
\end{proof}

This proposition, whose analog for the max-plus algebra was proved
in the sense of polynomial factorization \cite{Cuninghame79}, was
included here to give a quick positive result, but its proper
formulation in this theory is more sophisticated;
cf.~Proposition~\ref{prop:emptyzeroset} below.

\subsection{Different kinds of roots}

Varieties of tropical geometry come up in our theory as tangible
roots of tangible polynomials. However, since we have the ghost
structure at our disposal, we might as well consider roots of
non-tangible polynomials as well, thereby  enriching the geometry
and also adding insight to factorization.

Note that ghost elements are automatically roots when $f$ lacks a
constant term. Thus, our main interest is in tangible roots. A bit
of thought shows that, unlike the classical situation where the
tangible roots of a polynomial in one indeterminate are
topologically isolated, here we can have a continuum of tangible
roots. For example, every number less than 1 is a root of $\la +
1^\nu.$ Thus, we need to investigate roots of polynomials more
carefully.

\begin{rem}\label{rep1} Consider an arbitrary nonzero
polynomial $f = \sum _\bfi \a _\bfi \La ^\bfi \in R \pl \Lambda\pr
$, over a supertropical domain $R = \RGnu$. For any essential
monomial $\a _\bfi \La ^\bfi$ of $f$,  and $\bfa \in R^{(n)},$ let
us write $c_\bfi = \a_\bfi \bfa^\bfi,$ and
$$S(\bfa) = \{ c_\bfi: \bfi  \in \supp (f^\essn)\}.$$
Write $$c(\bfa)^\nu = \max \{ c_\bfi ^\nu:  c_\bfi \in S(\bfa)\},
\qquad  \text{and} \quad  \hat S(\bfa)  = \{ c_\bfi \in S(\bfa) :
c_\bfi ^\nu =  c(\bfa)^\nu\}.$$
In evaluating $f(\bfa)$, we may discard all $c_\bfi$  which are
not $\nu$-maximal. There are two cases:

\begin{description}
     \item[Case I]  $\hat S(\bfa)$ has at least two elements. \pSkip
   \item[Case II]  $\hat S(\bfa)$ has a unique element
$c_\bfj$. 

\pSkip


\end{description}
 In Case I,  we call
     $\bfa$ a \textbf{corner root}. These
      are the familiar roots in tropical geometry, i.e.,
those arising in tropical geometry in the corner locus of
polynomials over the max-plus algebra.  Note that any corner root
$\bfa$ is also a root of the binomial consisting of the sum of any
two monomials $\a_\bfj \Lambda^\bfj$ of~$f$ for which  $c_\bfj\in
\hat S(\bfa)$; this hints at the key role to be assumed by
binomials in tropical geometry.  A tangible  corner root $\bfa$ is
called  \textbf{\singular}  if, under the notation above, $\hat
S(\bfa) \subset \tT$, i.e., if the monomials determining the root
are tangible.

In Case II,     $\bfa$ is a  root  of $f$ iff $c_\bfj \in \tGz;$
we call this a \textbf{cluster root}. This is a new phenomenon
which arises from the supertropical structure, and does not occur
in the familiar theory of corner loci for  tropical geometry based
on the max-plus algebra.

\end{rem}

\begin{example} Consider the polynomial $\la^4 + 3^\nu \la^3 +
5^\nu \la^2 + 6\la + 6$ in $D(\R)\pl \la\pr .$ The tangible roots are 0 and all $a$ such
that $1 \le a \le  3.$ The corner roots are 0, (which is
\singular), and  1, 2, and 3 (which are not \singular). All the
other tangible roots are cluster roots.\end{example}

\begin{rem}\label{cont} Some immediate consequences, for $f \in R \pl
\lm\pr $:
\begin{enumerate} \eroman
    \item If $\ra^\nu \in \tG$ is ``large enough'' (for example,  for $f
= \lm^2 + \al_1\la +\al_0$, if $a \nuge \al_1$ and $\ra \a _1 \nuge
\a _0$), then $\ra^\nu$ is a  root of $f$. \pSkip
\item Likewise, if the leading coefficient of $f$ is ghost and
$\ra ^\nu $ is ``large enough,'' then $\ra$ is a  root of $f$.
\pSkip
\item If $\a _ 0 \in \tG$ and $\ra ^\nu$ is ``small enough,'' then
$\ra$ is a root of $f$.  For example, in $D(\R)\pl \la\pr $, every
$\ra \leq 7$ is a root of $\la^2 + 8\la +15^\nu$. \pSkip
\item More generally, suppose $\ra \in \tT$ is a root of $f = \sum
\a_i \la^i;$ let $c_i= \a_i \ra^i$, and (notation as in
Remark~\ref{rep1}), take
 $$
 \hat S (a)= \{ c_j \; is \;  \nu-maximal\; in \; S(a) \}.$$

If $\hat S(a)$ has only a single element $c_j$, i.e., $\ra$ is a
cluster root, then this $c_j$ must be a ghost, and thus there is
an open set containing $\ra$, all of whose elements of which are
roots of~$f$. (This could be viewed as a form of Krasner's Lemma
from valuation theory.) \pSkip
\item On the other hand, notation as in (iv), if all $c_j \in
S(a)$ are  tangible, then taking $\rb\in  \tT$ ``close" to~$\ra$,
but not equal, will make all the $\{c_i : c_i \in S(b)\}$
distinct, and thus $\rb$ is not a root of $f$.
\pSkip \item  If  $\ra \nucong \rb$ and
$\ra$ is a corner root of $f$, then $\rb$ is also a corner root.
Thus, even when $\ra \notin \tT $,  taking $\rb \in \tT $ for
which $\rb \nucong \ra$ yields
  a tangible root.
 \pSkip
 \item If $\ra \in \tT$ is a root of $f$ and $b \lmodg a$, then $b  $ is a
root of $f$ as well. \pSkip
\end{enumerate}
\end{rem}

A similar situation occurs for polynomials in $n$ indeterminates.
Nevertheless, \singular\ roots also involve extra subtleties, as
seen in Example \ref{tangex0},  differing considerably from the
situation in classical algebraic geometry.

  Remark~\ref{cont} shows that if $\ef$ is tangible, then all of
 the tangible  roots of $f$ are \singular.

\subsection{Laurent polynomials and rational Laurent functions}

Often it is convenient to consider a slightly larger semiring than
the semiring of polynomials. Let $\sF := F \setminus \{ \fzero
\}$, for a supertropical semifield~$F$. As before, we write
$\Lambda = \{ \la_1, \dots, \la _n\} $ and $\Lambda^{-1}$ for  $\{
\la_1^{-1}, \dots, \la _n^{-1}\} $. We want to consider functions
such as $\la_1^{-1}.$ Since these are not defined at $\fzero$, we
must be more careful, and define $\sFunF$ to be the semiring of
functions $(\sF)^{(n)} \to F,$ and $\sCFunF$ to be the
sub-semiring of continuous functions from $\sFunF$. $(\sF )^{(n)}$
could be called the supertropical torus.

\begin{defn}\label{rationalfunc} The semiring $F \pl \Lambda,
\Lambda^{-1} \pr$ of Laurent polynomials is the sub-semiring of
$\sCFunF$ generated by the \textbf{Laurent monomials} $\a_\bfi
\Lambda ^ {\bfi}$, where $\a_\bfi \in F$ and
$$\Lambda ^ {\bfi} = \la_1^{i_1}\cdots \la _n^{i_n}, \quad i_1,
i_2, \dots, i_n \in \Z.$$ \end{defn}

Strictly speaking, first we embed the semiring $F \pl \Lambda \pr$
into $F \pl \Lambda, \Lambda^{-1} \pr$, and then pass to $F [
\Lambda, \Lambda^{-1} ]$ inside $\sCFunF$. Explicitly, we have:

\begin{prop}\label{Laurentembed} There is a canonical 1:1 semiring homomorphism $$\Psi: F[\la_1,
\dots, \la_n] \To F[\la_1,\la_1^{-1}, \dots, \la_n,\la_n^{-1}],$$
given by $f \mapsto f.$ Every element of $F[\la_1,\la_1^{-1},
\dots, \la_n,\la_n^{-1}]$ can thereby be written in the form
$\frac f h,$ where $h$ is a suitable monomial in $F[\la_1, \dots,
\la_n]$.
\end{prop}
\begin{proof} Clearly $\Psi$ is a semiring homomorphism, which is
1:1 since $(\sF)^{(n)}$ is dense in $F ^{(n)}$ in the
$\nu$-topology. (Any two continuous functions that agree on a
dense subset are equal.) The last assertion is seen by clearing
denominators. \end{proof}

\begin{rem}\label{cancel5}  If $\bfa \in R^{(n)}$ is a tangible root of  $\al_\bfi \Lm^\bfi f$,
where $\al_\bfi$ is tangible, then $\bfa$ is also a root of
$f$.\end{rem}

 Since  multiplying by a tangible monomial does not
affect the roots of a polynomial, it is convenient to be able to
have monomials invertible, especially when $n\ge 2.$ Thus, $F
[\Lambda, \Lambda^{-1}]$ has the following important advantage
over $F [\Lambda]$:

\begin{rem}  Every tangible Laurent  monomial is invertible.
Consequently, we may always reduce Laurent polynomials to Laurent
polynomials having constant term $\fone$.
 For
example, the roots of~$\la_1 + \la_2$ in $(\sF)^{(2)}$  are the
same as those of $\la_1\la_2^{-1}+\fone$.

In fact, an element of $F [ \Lambda, \Lambda^{-1} ]$ is invertible
in $F [ \Lambda, \Lambda^{-1} ]$ iff it is  a tangible Laurent
monomial. (This is seen by a degree argument: When multiplying
Laurent polynomials that are not Laurent monomials, one gets a
highest-order term and a lowest-order term, and so the product
cannot be a Laurent monomial.)

\end{rem}

Here is an example of a proof made easier by passing to $F [
\Lambda, \Lambda^{-1} ]$.

\begin{prop}\label{prop:emptyzeroset} Suppose $\bar F = F,$ and   $f\in F [ \la_1, \dots,
\la_n ]$ is not a tangible  monomial.  Then $f$ has a tangible
root.
\end{prop}
\begin{proof} The case when $f$ is a ghost monomial is clear.
Take an essential monomial  $f_{\bfi} = \a _{\bfi} \La^{\bfi}$ of $f$.
Passing to $F \pl \Lambda, \Lambda^{-1} \pr$ and dividing by
$f_{\bfi},$ we may assume that $\fone$ is a monomial of $f$ (where
now $f$ is a sum of Laurent monomials).
   Write $f = g  + \fone$.   By an argument analogous to Lemma
\ref{lem:cl1} there exists tangible $r$ for which $g(r) \nucong
\fone$. Thus $r$ is a root of $f$.
\end{proof}

The Laurent polynomial semiring also permits us to focus the
duality of Remark \ref{add5}:

\begin{rem}\label{add6} If $\bfi = (i_1, \dots, i_n),$ write $-\bfi$ for $(-i_1, \dots,-i_n).$
The isomorphism $$\Phi_{\Fun}: \sFunF \To \sFun
((\dual{F})^{(n)},\dual{F})$$ of Remark~\ref{add5} extends to an
isomorphism
$$\Phi_{\ply}: F [\Lambda, \Lambda^{-1}] \To \dual{F} [\Lambda,
\Lambda^{-1}] $$ given by $\sum \a_\bfi \la^\bfi \mapsto \sum
\a_\bfi^{-1}\la^{-\bfi}$.
\end{rem}

\subsection{Convex sets} Suppose $F$ is a divisibly closed supertropical semifield.
In this case, $a^ t$ is defined for each $a\in F$ and $ t \in
\Rati$, by Remark~\ref{divcl}. Given $\bfa = (a_1, \dots, a_n)$,
 we define $\bfa^t = (a_1^t, \dots, a_n^t).$

\begin{defn}\label{def:convexFunction}
Suppose $F$ is divisibly closed. We define the
 \textbf{path} $\gm_{\bfa,\bfb}$ joining points $\bfa$ and $\bfb$ in $F^{(n)}$ to
 be the set $$\gm_{\bfa,\bfb} := \{\bfa ^t {\bfb}^{1-t}: t \in [0,1] \cap \mathbb Q
 \}.$$ A subset  $S \subseteq F^{(n)}$ is  \textbf{convex} if
 whenever $\bfa, \bfb \in S$ then $\gm_{\bfa,\bfb}$ lies in $S$.

 The  \textbf{left ray} (resp.~\textbf{right ray})  joining points $\bfa$ and $\bfb$ in $F^{(n)}$
 is the set $$\lgm_{\bfa,\bfb} := \{\bfa ^t {\bfb}^{1-t}: t \in (-\infty,1] \cap \mathbb Q
 \}$$
 (resp.~$\rgm_{\bfa,\bfb} :=  \{\bfa ^t {\bfb}^{1-t}: t \in [0,\infty) \cap \mathbb Q
 \}$).
  By \textbf{(closed) ray} we mean left ray or right ray. We define open (left, right) rays analogously.
 We define a \textbf{two-sided ray} to be a set of the form $\tgm_{\bfa,\bfb} := \{\bfa ^t {\bfb}^{1-t}: t \in  \mathbb Q
 \}$.

A function $\phi \in \FunF$ is called \textbf{linear} if, for any
$\bfa, \bfb \in F^{(n)},$ $$ \phi(\bfa)^t  \phi(\bfb)^{1-t} =
\phi(\bfa^t \bfb^{1-t}), \quad \forall t \in   \mathbb Q.$$

\end{defn}

\begin{rem}\label{Laurentlin}
Any Laurent monomial $h = \la_1^{i_1}\cdots \la_n^{i_n}$ is
linear; indeed write
$$\begin{array}{lllll}
   h (\bfa ^t {\bfb}^{1-t}) &  = & \al_\bfi (a_1^{t}b_1^{1-t})^{i_1} \cdots (a_n^{t}b_n^{1-t})^{i_n} & & \\
  & = & (\al_\bfi a_1^{i_1}  \cdots a_n^{i_n})^t (\al_\bfi b_1^{i_1}
\cdots b_n^{i_n})^{1-t} & = & h (\bfa) ^t h ({\bfb})^{1-t}.   \\
\end{array}$$
\end{rem}
Consequently, we have:

\begin{lem}\label{convex1} Given $\bfa, \bfb \in F^{(n)}$,
for  all  $\bfc \ne \bfa, \bfb$ in the path $\gm_{\bfa, \bfb}$
joining $\bfa$ and $\bfb$:
\begin{enumerate} \eroman
    \item If Laurent monomials $f_\bfi(\bfa ) \nuge f_\bfj(\bfa )$ and
$f_\bfi(\bfb)\nug f_\bfj(\bfb),$ then $f_\bfi(\bfc )\nug
f_\bfj(\bfc ) $;
 \pSkip

 \item If Laurent monomials $f_\bfi(\bfa ) \nug f_\bfj(\bfa )$ and
$f_\bfi(\bfb)\nuge f_\bfj(\bfb),$ then $f_\bfi(\bfc )\nug
f_\bfj(\bfc ) $;
 \pSkip

    \item If $f_\bfi(\bfa ) \nuge f_\bfj(\bfa )$ and
$f_\bfi(\bfb)\nuge f_\bfj(\bfb),$ then $f_\bfi(\bfc )\nuge
f_\bfj(\bfc ) $.
\end{enumerate}
\end{lem}
\begin{proof} (i) (and (ii)): Write $ \bfc  = \bfa ^t {\bfb}^{1-t}$ for $0 \le t
\le 1.$ Then
$$f_\bfi(\bfc )= f_\bfi(\bfa ^t {\bfb}^{1-t}) = f_\bfi(\bfa)^t f_\bfi({\bfb})^{1-t}
\nug f_\bfj(\bfa)^t f_\bfj({\bfb})^{1-t} = f_\bfj(\bfc ).$$

The proof for (iii) is analogous.
\end{proof}



\section{Supertropical geometry}

 Having the basic concepts of polynomials under our belt, we are ready to
apply them to tropical geometry. For convenience, we treat affine
geometry; the analogous discussion using homogeneous polynomials
would yield the parallel results for projective geometry. Although
the
 following definitions could be formulated over an arbitrary
 semiring with ghosts, we work throughout over a supertropical
 semifield $F= \FGnu.$

 \subsection{Supertropical root sets}

\begin{defn} For a set  $S \subset F[\La]$, we define  its \textbf{root
set}
$$\tZ(S) = \{ \bfa =  (\ra_1,\dots,\ra_n) \in F^{(n)} \ : \ f(\bfa ) \in \tGz, \  \forall f \in S \} \subseteq F^{(n)}.$$
The \textbf{complement} of $\tZ(S)$ is $F^{(n)} \setminus \tZ
(S)$.

The \textbf{tangible root set}, denoted $\tZ_{\tng} (S)$, is
$\tZ(S) \cap \tTz^{(n)}$. The \textbf{tangible complement} of
$\tZ_{\tng} (S)$ is $\tTz^{(n)} \setminus \tZ_{\tng} (S)$. When $S
= \{ f\}$ consists of only one polynomial, we write
$\tZ_{\tng}(f)$ for $\tZ_{\tng}(S)$, which is called a
\textbf{supertropical hypersurface}; we call $\tZ_{\tng}(f)$ a
\textbf{tangible primitive} when $f$ is a tangible binomial.
\end{defn}

\begin{lem}\label{hyperpl1} If two points lie in a tangible primitive, then the
two-sided ray containing them also lies in this tangible
primitive.\end{lem}
\begin{proof} Suppose $f = h_1 + h_2 $ is a sum of monomials, with
$f(\bfa), f(\bfb) \in \tGz.$ Then, in view of
Remark~\ref{Laurentlin}, \begin{equation}\label{hyperpl} h_i(\bfa
^t {\bfb}^{1-t}) = h_i(\bfa) ^t h_i(\bfb)^{1-t}
 .\end{equation}
 since also $h_1(\bfa)$ and  $h_2(\bfa)$ are tangible, we have $h_1(\bfa) \nucong
 h_2(\bfa)$,
 and likewise $h_1(\bfb) \nucong h_2(\bfb)$; thus,
 Equation~\eqref{hyperpl} yields $ h_1(\bfa
^t {\bfb}^{1-t}) + h_2(\bfa ^t \bfb^{1-t}) \in \tGz$ since also $
h_1(\bfb) + h_2(\bfb) \in \tGz$. \end{proof}

(Note that this argument does not work for ghosts: If $h_1(\bfa),
h_1(\bfb)\in \tGz$ with $h_1(\bfa) \ge_\nu h_2(\bfa)$ and
$h_1(\bfb) \ge_\nu h_2(\bfb)$, then clearly $$ h_1(\bfa ^t
{\bfb}^{1-t}) + h_2(\bfa ^t \bfb^{1-t}) = h_1(\bfa) ^t
h_1(\bfb)^{1-t} \in \tGz$$ for $0 \le t \le 1,$ but for small $t$
or large $t,$ one could have $h_2$ dominating, and thus the sum
may be tangible.)

 The tangible root sets are the geometric objects that we would
like to view as supertropical varieties. One advantage of this
approach is that elementary considerations yield the usual
correspondence between varieties and ideals of polynomials, whose
analogous formulation under other definitions (involving domains
of non-differentiability) might fail:

\begin{rem}\label{rootset1}$ $
\begin{enumerate}\eroman
\item $\tZ_{\tng}(S_1)\cap \tZ_{\tng}(S_2)= \tZ_{\tng}(S_1 \cup
S_2)$. Thus, the intersection of tangible root sets is a tangible
root set.  \item $ \tZ_{\tng}(f) \cup \tZ_{\tng}(g)=
\tZ_{\tng}(fg)$. Thus, the union of finitely many supertropical
hypersurfaces is a supertropical hypersurface.
\end{enumerate}
\end{rem}

Nevertheless we continue to use the terminology ``(tangible) root
set'' to avoid confusion with other definitions of tropical
varieties.

\begin{remark}\label{rmk:closeRootSet}  For any $S \subset F[\La]$, the root set $\tZ(S) $  of $S$ is closed in the
    $\nu$-topology; its tangible complement (as well as its  complement)
    is  open in the
    $\nu$-topology.
\end{remark}\noindent Some examples of tangible root sets are given in Examples
~\ref{linearirr}, \ref{exm:1} and \ref{exm:2}.

 For an appetizer, let us start with a sample result, reminiscent
of a ``weak Nullstellensatz.''

\begin{rem} Any finite set $S$ of non-constant  polynomials has  common
 roots. In fact, one can just take ghosts ``large enough''
so that they outweigh the constant terms in the polynomials.   On
the other hand, $S$ could have no common tangible roots; for
example, $\la +2$ and $\la +3$ have no common tangible
root.\end{rem}

\begin{defn}\label{idealofzero}  The
\textbf{ideal} $\tI(Z)$ \textbf{of a  set} $Z \subset F^{(n)}$ is
defined to be
$$ \tI(Z) = \{ f\in F[\la_1, \dots, \la_n] \ : \
f(\bfa) \in \tGz, \forall \bfa \in Z\}.$$
We call $\tI(Z)$ the \textbf{ideal of polynomials satisfying} $Z$.
\end{defn}

\begin{prop}\label{radid} For any set $Z \subset R^{(n)}$, the ideal $\tI(Z)$ is
a ghost-closed radical ideal of $F[\Lambda]$.
\end{prop}
\begin{proof} To check that $\tI(Z)$ is an ideal, note that if $f(\bfa) \in \tGz$ and $g(\bfa) \in \tGz$ for polynomials $f$ and $g$,
then $(f+g)(\bfa)= f(\bfa)+g(\bfa) \in \tGz;$ likewise if $f(\bfa)
\in \tGz$ and $g$ is any other polynomial, then $f(\bfa)g(\bfa)
\in \tGz.$

Furthermore,  $f(\bfa) \in \tGz$ iff  $f(\bfa)^k \in \tGz$ (for
any number $k \ge 1$), implying $f \in \tI(Z)$ iff $f^k \in
\tI(Z)$; hence $\tI(Z)$ is a radical ideal. Finally, every element
is a root of each ghost polynomial, so $\tI(Z)$ contains all the
ghost polynomials.\end{proof}

This leads us to try to identify root sets with the ghost-closed
radical ideals of the polynomial semiring, which we study further
when considering the Nullstellensatz in the next section.

\begin{rem}\label{prime2}
 Any  ghost-closed  ideal $P$ of $F[\lambda]$  contains
$$(\la +\a^\nu )(\la^\nu + \a) = \nu(\la^2  + \a\la +\a^2) \in \tG[\la].$$
 It follows that the ghost ideal $\tGz[\la]$ is not
prime.\end{rem}

 As soon as one tries to dig deeper, one encounters many
potential pitfalls, which we illustrate with a few
 examples in one indeterminate.

\begin{example}\label{counterex} Suppose $F = D(\mathbb R).$ Thus, $\tT  =  \mathbb R.$

\begin{enumerate} \eroman
 \item No  ideal of $F \pl \la \pr$ defined by a set of
 tangible
 roots  contains both $\la ^2 + \la +2$ and $\la
^2 + 3\la +1$, since the latter has tangible roots $3$ and $-2$,
whereas the former only has the tangible root 1. \pSkip

\item Consider the ideal $A$ of polynomials having 2 as a root.
The polynomial $f = \la + 3^\nu\in A$, since $f(2) = 2 + 3^\nu =
3^\nu.$ Also $\la +2 \in A$ and  $f + (\la +2) = \la^\nu + 3^\nu,$
a ghost. On the other hand, by the same token, $f+(\la +3)$ is the
same ghost, although  $\la +3\notin A.$ (Actually,  every real
number $\le 3$ is a tangible root of $f$.) \pSkip

\item Besides being a root of $\la +2,$ the number $2$ is the
maximal tangible root of $\la + 2^\nu,$ and the minimal tangible
root of $\la^\nu +2$.  \ Every element of $F$ is a root of $\la
+2^\nu$ or $\la ^\nu  +2$ (cf.~Remark~\ref{prime2}). \pSkip

\item We would like the ideal of polynomials having $1,2$ as roots
to be generated by $(\la +1)(\la+2) = \la ^2 + 2 \la +3.$ But $\la
+ 3^\nu$ is in this ideal, and its degree is too small! \pSkip

\item $0$ is a root both of $3 \la +3$ and $\la ^2 + 3 \la +3,$
but not of the tangible part of their sum  $\la ^2 + 3^\nu \la
+3^\nu,$ which is $\la ^2.$ \pSkip

\item  The ideal of $F \pl \la \pr$ generated by all $\{\la + \a :
\a \in \R \}$ is not finitely generated in the classical sense.
(For any $S =\{\la + \a _1, \dots, \la + \a _m\},$ just take $\a <
\min \{ \a _1, \dots, \a _m\},$ and $\la +\a$ is not generated by
$S$.) \pSkip
\end{enumerate}
\end{example}

As we continue, we need to pick our way through these various
examples. One also wants to go in the other direction, from root
sets to polynomials. Different polynomials could define the same
root sets, and the same idea used in classical algebraic geometry
is applied here.

\subsection{Tropical varieties versus supertropical
root sets}\label{compare}

Our definition of root set encompasses the corner locus defined in
the introduction. In fact, one can apply this idea to any field
$\Fld$ with non-Archimedean valuation $\nVal: \Fld \to \Trop$.
Namely, given any affine hypersurface defined as the roots of the
polynomial $\kF\in \Fld[\La],$ one can pass to the tangible root
set of the polynomial $\trF$ defined over the supertropical
semifield $D(\nVal(\Fld)).$ (For example, one can take $\Fld$ to
be the field of Puiseux series.)


To incorporate the  ideas of tropical geometry into the
supertropical theory, we recall the special case of tangible
polynomials over the supertropical semifield $\Trop := D(\Real)$,
in which we recall $\tT = \Real$, $\tG = \Real^\nu$,
$\zero_{\Trop} = -\infty$ with addition and multiplication induced
by $\max$ and $+$, respectively.  For any tangible polynomial $f
\in D(\Real)[\lm_1,\dots, \lm_n]$ and point $\bfa \in \tT^{(n)}$,
$f(\bfa)$ is ghost iff the evaluation of $f$ on $\bfa$ is attained
by two monomials having the same dominant value (both tangible by
definition); namely $\bfa$ belongs to the corner locus $\Cor(f)$
of $f$. In other words, the domain of non-differentiability  of
$f$ is comprised precisely of the corner roots of $f$. In this
way, the corner locus of a polynomial is obtained as a tangible
root set, and we are poised to pass to a version of algebraic
geometry over the supertropical semifield.

Let us now describe how the tangible root sets (i.e., corner loci)
provide the {\it balancing condition} (Equation
\eqref{eq:balanceCond}). By Theorem \ref{eq:Kapranov}, any
tropical polynomial $\reF$ can be written as a convex piecewise
linear function function
\begin{equation}\label{eq:polyToFunc2} \reF(\bfa) \ = \ \max_{\bfi
\in\Omega \subset \Net^{(n)}}\{ \langle \bfa ,\bfi \rangle+
\al_\bfi \} ,\quad \bfa \in\zReal^{(n)}\ ,
\end{equation}
and the convex hull~$\Delta(\reF)$ of the set $\Omega$  is called
the \textbf{Newton polytope of~$\reF$}.  This is a lattice
polytope that also has a  subdivision
\begin{equation*}\label{eq:polySub} S({\reF})\ :\
\Delta(\reF)=\Delta_1\cup\dots\cup\Delta_N
\end{equation*}
into convex lattice polytopes, determined by projecting the upper
part of the convex hull of $(i_1, \dots, i_n, \al_\bfi)$ onto
$\Delta(\reF)$. This correspondence yields a duality
  between $\Cor({\reF})$ and $S(\reF)$ which
  inverts inclusion of faces;   the dual of a $k$-dimensional face
  in $\Cor({\reF})$ is an $(n-k)$-dimensional face in $S({\reF})$. In particular, the dual of
  an $(n-1)$-dimensional  face $\delta$  in $\Cor({\reF})$ is a lattice edge of
  $S(\reF)$ whose integral length provides the weight $m(\delta)$
  for $\delta$. The one-dimensional faces of the polytope
  come from binomials $\al_{\bfi} \la^{\bfi}+\al_{\bfj} \la^{\bfj}$ appearing in
  $f$, whose zero sets satisfy $\frac{\al_{\bfj}}{\al_{\bfi}} \la^{\bfj -\bfi} = \one_{\Trop},$
  and thus has normal vector of slope  ${\bfj -\bfi}$.

\begin{example} Consider the polynomial $f = \lm_1^2 \lm_2 + \lm_1 \lm_2^2 + 1 \lm_1 \lm_2 +
0$   in $\zReal[\lm_1, \lm_2]$. The corner locus $\Cor(f)$
consists of three line segments and three rays; these are exactly
the tangible root set $\tZ_{\tng}(f)$ as viewed in $\Trop[\lm_1,
\lm_2]$. (See Figure \ref{fig:curve}(a)).
\end{example}

 Supertropical geometry  permits  a   wider
scope for the definition  of polyhedral complexes, since we also
have non-tangible polynomials at our disposal, which enable us to
describe $n$-dimensional polyhedral complexes within
$n$-dimensional spaces (for example polytopes); a supertropical
hypersurface $\tZ_\tng(f) \subset \Real^{(n)}$ may have
(topological) dimension $n$ when $f$ has an essential ghost
monomial. For example, the tangible root set of $\la + 1^\nu$ in
$F[\la]$ is a ray. Here is a more interesting illustration of
supertropical varieties that previously were not available.

\begin{figure}
\begin{minipage}{0.9\textwidth}
\begin{picture}(10,140)(0,0)
\includegraphics[width=5.5 in]{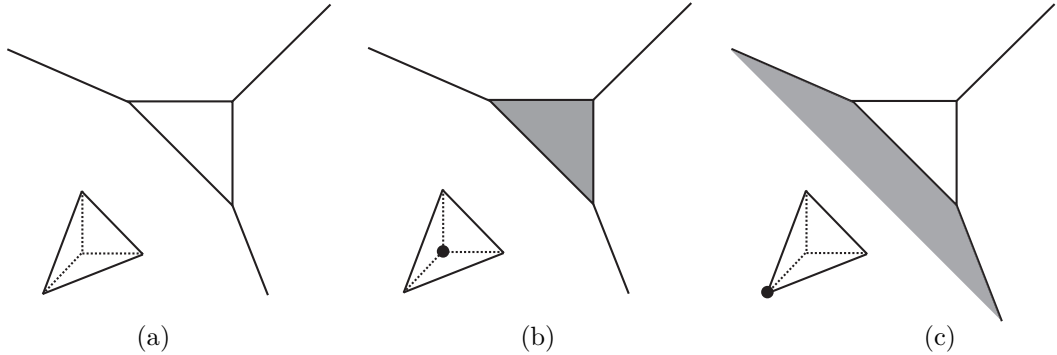}
\end{picture}
\end{minipage}
\begin{minipage}{0.3\textwidth}
\center{(a)}

\end{minipage}\hfil
\begin{minipage}{0.3\textwidth}
\center{(b) }

\end{minipage} \hfil
\begin{minipage}{0.3\textwidth}
\center{(c)}

\end{minipage}
\caption{\label{fig:curve} The tangible root sets (shaded) and the
corresponding Newton polytopes for $f = \lm_1^2 \lm_2 + \lm_1
\lm_2^2 + \al \lm_1 \lm_2 + \bt $ with $\al$ and $\bt$ having
tangible or ghost values. The dashed lines in the left diagrams
stand for the subdivision of $\Delta(f)$, and $\bullet$ marks the
ghost vertices.}
\end{figure}

\begin{example} Consider the essential polynomial $$f = \lm_1^2 \lm_2
+ \lm_1 \lm_2^2 + \al \lm_1 \lm_2 + \bt$$ and the tangible root
set $\tZ_\tng(f)$ of $f$. When $f$ is tangible, $\tZ_\tng(f)$ is
just a standard tropical curve of genus $1$ (see Figure
\ref{fig:curve}(a)). When $\al$ is a ghost, $\tZ_\tng(f)$ is of
dimension $2$ and has genus $0$ (see Figure \ref{fig:curve}(b)).
For $\bt$ ghost, $\tZ_\tng(f)$  has genus $1$ and dimension $2$
(see Figure \ref{fig:curve}(c)).

\end{example}

The Newton polytope $\Delta(f)$ of a supertropical polynomial $f$
is defined in the same manner, where vertices that correspond to
ghost monomials are designated  as ghosts. The duality between the
tangible root set of a polynomial and the subdivision $S(f)$ of
its Newton polytope is preserved; here, a ghost vertex of
$\Delta(f)$ corresponds to an $n$-dimensional face of
$\tZ_\tng(f)$ (see for example Figure \ref{fig:curve}).

 The ghost roots provide a new
dimension to the geometry, as illustrated in the following
examples  of root sets of polynomials in two indeterminates. We
display both the tangible and ghost parts of the root sets, by
taking each axis to represent tangible elements in one direction
(from $-\infty$) and ghost elements in the other direction. In
other words, each axis looks like:

 $$\xymatrix{
     \cdots & & & &  \overset{-\infty}{\cdot}
     \ar@{->}[llll]^{\text{ghosts in decreasing $\nu$-order}}  \ar@{->}[rrrr]_{\text{tangibles in increasing
      $\nu$-order}}  & &  & & \cdots \\
 }$$

$ $

\begin{example}\label{tangex0} $ $~
\begin{enumerate} \eroman

\item  Supertropical hypersurfaces: The tangible roots of  $f =
\la_1 + \la _2 + a$ in $F[\la]$ are

\pSkip  $\qquad \begin{cases} (a,b) \text{ for } b \le_\nu a; \\
(b,a) \text{ for } b \le_\nu a; \\ (b, b) \text{ for } b\ge_\nu a.
\end {cases}$
\pSkip Thus, $\tZ_{\tng} (f)$ is comprised of three rays, all
emanating from $(a,a)$, and its tangible complement has three
 components; cf.~Figure \ref{fig:linrconic}(a). \pSkip

\item
    The tangible complement of the  supertropical
    conic, given by $f= \lm_1^2 + \lm^2 + a^\nu \lm_1 \lm_2 + b$;
    cf.~ Figure~\ref{fig:linrconic}(b), is comprised of three components. This conic is of dimension $2$. \pSkip

    \item Tangible root sets of dimension $2$ in $F^{(2)}$; cf.~
    Figure \ref{fig:2dimset}: \pSkip
    \begin{enumerate}
        \item The unbounded strip $\tZ_{\tng} (\{ \lm_1 + a^\nu \lm_2, \, b^\nu\lm_1 + \lm_2\} )$,
        $b>_\nu  a^{-1}$ ; \pSkip
        \item The supertropical  hypersurface consisting of two half
        spaces,  $\tZ_{\tng} (\{ a^\nu \lm^2_1 + b \lm_1\lm_2 + a^\nu \lm^2_2\} )$,
        $b > _\nu  a$;  \pSkip
        \item A convex bounded tangible root set
        $\tZ_{\tng} (\{ \lm_1 + a^\nu, \, \lm_2 +a^\nu, \, (\lm_1\lm_2)^\nu + a \}
        )$;
    \end{enumerate}
    for  $a,b,c  \in \tT.$
\end{enumerate}
\end{example}

\begin{figure}
\begin{minipage}{0.7\textwidth}
\begin{picture}(10,160)(0,0)
\includegraphics[width=4.5 in]{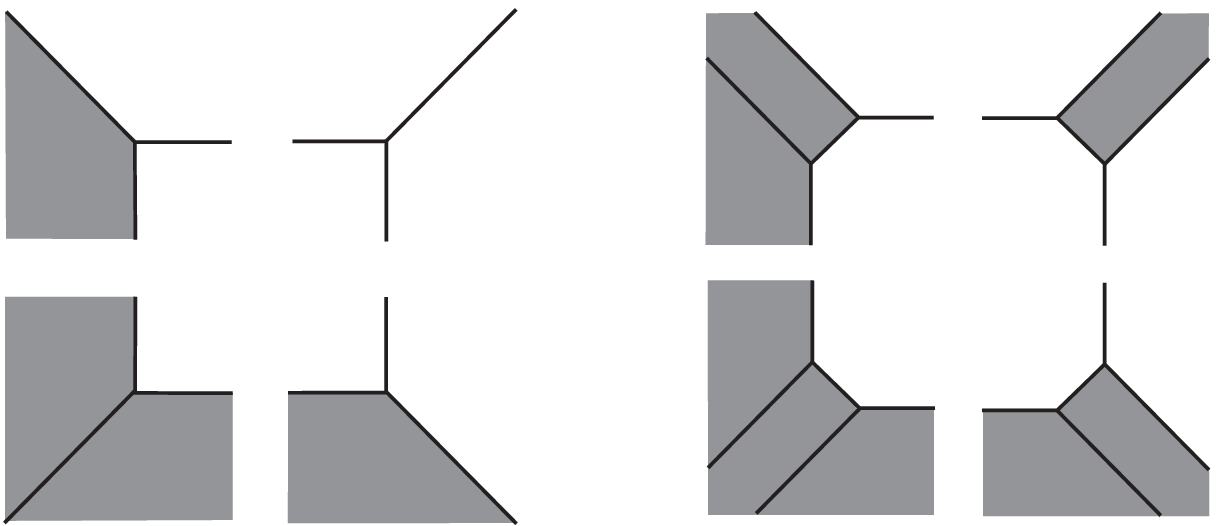}
\put(-10,145){$_{\tT  \times \tT}$} \put(-195,145){$_{\tT \times
\tT}$}

\put(-145,145){$_{\tG  \times \tT}$} \put(-332,145){$_{\tG \times
\tT}$}

\put(-10,-10){$_{\tT  \times \tG}$} \put(-195,-10){$_{\tT \times
\tG}$}

\put(-145,-10){$_{\tG  \times \tG}$} \put(-332,-10){$_{\tG \times
\tG}$}

\end{picture}
\end{minipage}
\begin{minipage}{0.7\textwidth}
\begin{picture}(10,20)(0,0)

\put(60,0){(a)} \put(250,0){(b)}
\end{picture}
\end{minipage}

\caption{\label{fig:linrconic} (a) The root set of the
supertropical tangible line $f = \lm_1 + \lm_2 + a$, $a \in \tT$.
(b) The root set of the supertropical conic  $f = \lm_1^2 +
\lm_2^2 + a^\nu\lm_1 \lm_2 + b$, $a, b \in \tT$.}
\end{figure}

\section{The Nullstellensatz}

One very basic goal (and perhaps the main result of this paper) is
to find an analog of Hilbert's Nullstellensatz, in order to
provide an algebraic foundation for tropical geometry.
Unfortunately, the naive tropical formulation just does not work,
even over a supertropical semifield.

The ``naive tropical Nullstellensatz'' would be that for any
divisibly closed, archimedean supertropical semifield $F$ and any ideal $A$ of
$F[\Lambda] = F[\la _1, \dots, \la_n],$ a polynomial $f$ satisfies
all common roots of $ A$ iff $f \in \sqrt A$. Unfortunately, there
are many counterexamples to this assertion, some of which are
given in Example~\ref{counterex}. This leaves us with a dilemma:
Do we want to hold on the notion of ideal and move our focus away
from root sets, or do we want to stay with  root sets and modify
our definition of ideal in the tropical sense? We deal with the
first approach, since it turns out to be more straightforward and
quite natural. It also turns out that the proofs are most easily
expressed in terms of the Laurent polynomial semiring $F[\Lambda,
\Lambda^{-1}].$

In this discussion, we view a supertropical semifield $F = \FGnu$,
with $\tT _{\zero}= F \setminus \tG,$  endowed with the
$\nu$-topology described in Definition \ref{ordertop}, and assume
that $F$ is divisibly closed and archimedean.

\begin{defn} Given a supertropical semifield $F = \FGnu$  and a polynomial
 $f\in F[\lm_1,\dots, \lm_n],$ we define the set
$$D_f := \{ \bfa = (a_1, \dots, a_n) \in \tT _{\zero}^{(n)} : f(\bfa) \in
 \tTz  \} = \tT _{\zero}^{(n)}\setminus \tZ_\tng(f) .$$

Refining this definition, writing $f = \sum f_\bfi$, a sum of
monomials, define $D_{f,\bfi}$ to be
$$D_{f,\bfi} := \{\bfa  \in \tT _{\zero}^{(n)} : f(\bfa) = f_\bfi (\bfa)
\in \tTz \}.$$ We call the $D_{f,\bfi}$ the \textbf{(tangible)
components} of $f$; we call ${f_\bfi}$ the \textbf{dominant
monomial} of $f$ on the component $D_{f,\bfi}$.

Likewise, we define the \textbf{closed components} of $f$ to be
$$\oD_{f,\bfi} := \{\bfa  \in \tT _{\zero}^{(n)} : f(\bfa) \nucong  f_\bfi
(\bfa) \}.$$
\end{defn}
\noindent Note that $f$ has finitely many components, since $f$ is
a sum of finitely many monomials.

 Components and
closed components are defined the same way for Laurent
polynomials, although here we only consider points in $\tT
 ^{(n)}$ (since a Laurent polynomial need not be defined at
 $\fzero$). Namely,
given  a Laurent polynomial
 $f\in F[\lm_1,\lm_1^{-1},\dots, \lm_n,\lm_n^{-1}],$ we define the set
$$D_f = \{ \bfa = (a_1, \dots, a_n) \in \tT ^{(n)} : f(\bfa) \in \tT  \};$$
thus  $\tT\setminus D_f$ is the set of tangible roots of $f$ in
$\tT ^{(n)}.$

Again, writing $f = \sum f_\bfi$, a sum of Laurent
 monomials, define $D_{f,\bfi}$ to be
$$D_{f,\bfi} = \{\bfa  \in \tT _{\zero}^{(n)} : f(\bfa) = f_\bfi (\bfa)
\in \tT \}.$$ We call the $D_{f,\bfi}$ the \textbf{(tangible)
components} of $f$; we call ${f_\bfi}$ the \textbf{dominant
Laurent monomial} of $f$ on the component $D_{f,\bfi}$.

By definition, $\bfa \in D_{f,\bfi}$ iff $f_\bfi(\bfa)$ is
tangible and $f_\bfi(\bfa) \nug f_\bfj(\bfa)$ for all $\bfj \ne
\bfi;$ hence $$D_f = \bigcup _{\bfi} D_{f,\bfi}.$$ On the other
hand, when $f$ is tangible,  $\bigcup _{\bfi} \oD_{f,\bfi} =
\tTz^{(n)},$ and one obtains a chain complex by taking the
intersections of closed components.

\begin{rem}\label{intersect} $\oD_{f,\bfi_1} \cap
\oD_{f,\bfi_2}\cap \dots \cap \oD_{f,\bfi_k} = \{ \bfa \in \tT
_{\zero}^{(n)} : f(\bfa) \nucong  f_{\bfi_u} (\bfa) , \ 1 \le u
\le k\}$.\end{rem}

\begin{figure}
\begin{minipage}{0.6\textwidth}
\begin{picture}(10,160)(60,60)
\includegraphics[width=4.0 in]{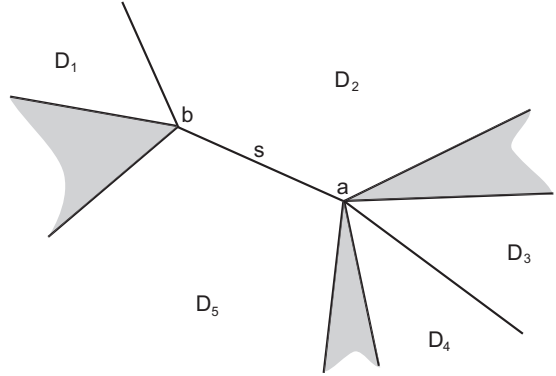}
\end{picture}
\end{minipage}

\caption{\label{fig:kborder} The extremal points $a$ and $b$ are
$4$-border
 and $3$-border, respectively.
 The line segment $s$ is $5$-border, while $s \setminus \{a ,b \}$
 borders only the two components   $D_2$ and~$D_5$.}
\end{figure}

We call any such  nonempty set a \textbf{$k$-border}, and we call $\{
\oD_{f,\bfi_1}, \dots, \oD_{f,\bfi_k}\}$ its \textbf{bordering
components}.  Note that $k$ does not necessarily describe the
codimension, since many components of the same dimension could
meet at a common border. Nevertheless, we can call a $k$-border
\textbf{extremal} if it does not have any other bordering
components. This means in the terminology of
Remark~\ref{intersect} that there is no $\bfa \in
\cap_{u=1}^k\oD_{f,\bfi_u}$ for which $f(\bfa) \nucong
f_{\bfi}(\bfa)$ for some $\bfi \ne \bfi_1, \dots, \bfi_k.$

\begin{rem}\label{finitebord} Since any $k$-border is determined by the monomials of the
components bordering it, there are at most $\binom m k$
$k$-borders, where $m$ is the number of monomials of $f$.
\end{rem}

The \textbf{border} between two components $\oD_{f,\bfi}$ and
$\oD_{f,\bfi'}$ is defined as $$\oD_{f,\bfi} \cap \oD_{f,\bfi'}
\setminus \left\{ \cup _{ \bfj \ne \bfi,\bfi'}
\oD_{f,\bfj}\right\},
$$
in other words,  the 2-border after we remove all 3-borders (which
include all the $k$-borders, $k > 3$). Two components having a
nonempty border are called \textbf{neighbors}. (Note that the
border is then dense in the 2-border.) See example in Figure
\ref{fig:kborder}.

 Keeping
Corollary~\ref{Prop:powOfPols} in mind, we have
\begin{rem} For any $m \in \Net,$ $D_f = D_{f^m}$, and $D_{f,\bfi}
= D_{f^m,m\bfi}$.\end{rem}

 Clearly, the $D_{f,\bfi}$ are open sets (with respect to the $\nu$-topology).
 Furthermore, they are convex, in the sense of Definition \ref{def:convexFunction}:
\begin{proposition} Each   component $D_{f, \bfi}$ is
convex, and every closed component is convex.
\end{proposition}
\begin{proof}
Follows immediately from Lemma~\ref{convex1}, applied to the
dominant monomial in the component.
\end{proof}

\begin{cor} Every $k$-border  is
convex.
\end{cor}

In other words, the common boundaries   of the tangible components
are convex. Thus, we can apply convexity arguments in our proofs,
even without any extra topological assumptions on the supertropical
semifield $F$.

\begin{defn} A closed component $\oD_{f,\bfi} $ is
\textbf{2-sided unbounded} if it contains a two-sided ray;
otherwise, it is \textbf{partially bounded}. The closed component
$\oD_{f,\bfi} $ is \textbf{bounded}  if it contains no one-sided
ray.
\end{defn}

\begin{lem} Any extremal border of a partially bounded component $\oD_{f,\bfi}$ must be a
point.\end{lem}
\begin{proof} Otherwise this extremal border contains two points, and thus the path
connecting them. But then the two-sided ray through this path must also
intersect some other border of $\oD_{f,\bfi}$, contrary to the
extremal hypothesis.\end{proof}

In view of this observation, we define an \textbf{extremal point}
of a partially bounded component to be an extremal border.

\begin{Note} The same sort of argument shows that any extremal
border must be the intersection of primitives, but we do not need
that fact.\end{Note}

\begin{prop} Any Laurent monomial $h$ takes on $\nu$-maximal and $\nu$-minimal values in any bounded
closed component $\oD_{f,\bfi}, $ and furthermore, for $h$ non-constant on $\oD_{f,\bfi}, $  the maximal and
minimal $\nu$-values are taken at extremal points.\end{prop}
\begin{proof}  By symmetry, we need only look for the $\nu$-minimal values on $\oD_{f,\bfi}. $
We claim that for any $\bfa$ in $D_{f,\bfi}$ there is a  point $\bfa'$ on some border of $\oD_{f,\bfi}$
such that $h(\bfa ') < _\nu h(\bfa)$. Indeed, take any point  $\bfb \in D_{f,\bfi} $
with  $h(\bfb) <  _\nu  h(\bfa).$ The path connecting $\bfa$ and $\bfb$ lies in the same
component, and continuing further along the ray until the border produces a
smaller value of $h$.

We continue with the same argument applied to
this border, showing that if $\bfa$ lies in a $k$-border, there is a point $\bfa'$ in a $k+1$-border with  $h(\bfb) <  _\nu  h(\bfa),$ unless the $k$-border is extremal. Thus,
 we reach an extremal border (since the number
of times we can apply this argument is at most the number of
monomials in $f$). By definition, this cannot contain a ray, so is
an extremal point, and since there are only finitely many extremal
points, the $\nu$-minimal value of $h$ on all of these must be the
$\nu$-minimal value on $\oD_{f,\bfi}$.
\end{proof}

Since Remark~\ref{finitebord} shows there are only finitely many
extremal points, we now see that the minimal and maximal
$\nu$-values of a monomial on the various bounded components are
all obtained at finitely many points. This is the key to our proof
of the Nullstellensatz below.

\begin{defn}\label{epst1}
Notation as above, we write $f \preceq _{D_{f,\bfi}} g$ if
$g(D_{f,\bfi}) \subseteq \tT;$ we write $f \epsc S$ for $S
\subseteq F[\lm_1, \dots, \lm_n],$ if for every essential monomial
$f_\bfi$   of $f$ there is some $g\in S$ (depending on
$D_{f,\bfi}$) with $f \preceq _{D_{f,\bfi}} g$.
\end{defn}

\begin{rem}\label{someex} $ $
\begin{enumerate}
\item The use of components is more precise than merely
considering tangible root sets. For example, the tangible root set
of the ideal $A$ of $F[\la]$  generated by $\la +2$ and $\la +3$
is $\{ 2 \} \cap \{ 3 \} = \emptyset.$ However, the constant
polynomial $1$ has no tangible roots, but does not belong to $A$.
On the other hand,   the   component of $1$   is all of $\tT$, and
is not contained in  the   component of any element of $A$, so $1
\not \epsc A.$

This example also shows that $\tI(Z_1 \cap Z_2)$ can be larger
than the radical of the ghost-closed ideal generated by $\tI(Z_1)$
and $\tI(Z_2)$.

\item Sometimes polynomials behave differently when viewed as
Laurent polynomials. Consider the polynomials $f =\la_1+1$ and $g
=(\la_1+1)\la_2$ in $F[\lm_1,\lm_2]$ .  Then $g(a,\fzero) =
\fzero$ whereas $f(a,\fzero) = a+1\ne \fzero,$ and $f \npreceq
_{D_{f,\bfi}} g$. But as Laurent polynomials, $f$ and $g$ take on
precisely the same tangible roots, so $f \preceq _{D_{f,\bfi}} g$.

 Nevertheless, Laurent
polynomials are useful in proving the Nullstellensatz for ordinary
polynomials.
\end{enumerate}
\end{rem}

Let us record some information about the components of a
polynomial.

\begin{rem} For any polynomials $g,h$, any  component of
$g$ contains a component of $gh$, since $\mathcal Z(g) \subseteq
\mathcal Z(gh).$\end{rem}

\begin{rem}\label{dommon} Any polynomial $f$ has a unique dominant monomial on any component $D$ of
~$f$. (Otherwise, there are two dominant monomials $g$ and $h$ of
$f$ on part of $D$, and $\{ \bfa \in \tT : g(\bfa) = h(\bfa)\} $
is contained in $D \cap Z(f) = \emptyset$, a
contradiction.)\end{rem}

\begin{prop}\label{dom2} If $f = \sum g_j$ for polynomials $g_j$,
then any component $D$ of $f$ is contained in some component of
some $g_j$; i.e., $f \preceq _D g_j.$
\end{prop}
\begin{proof} Suppose $\bfa \in D,$ and $h$ is the dominant monomial
of $f$ on $D$. Then $h$ is a monomial of some $g_j$. By
Remark~\ref{dommon}, $h$ is the dominant monomial of $f $ on all
of $D$, and thus is the dominant monomial of $g_j$ on all of $D$.
 Hence the
 component of $g_j$ contains $D$.
\end{proof}

\begin{cor} If, for some $k\ge 1,$ the polynomial $f^k$ belongs to the ideal $A$ generated by the set of polynomials $S = \{g_i : i \in I \},$
 then $f \epsc S.$\end{cor}
\begin{proof} Since $f$ and $f^k$ have the same components, we may
assume that $k = 1.$ Write $f = \sum q_i g_i.$ At each component
$D$ of $f$ we have suitable $j$ such that $f\preceq _D q_j g_j$,
so $f \preceq _D g_j;$  we conclude that  $f \epsc S.$\end{proof}

Our goal is to prove the following theorem:

\begin{thm}\label{Null2} \textbf{(Supertropical Nullstellensatz)}
Suppose $F$ is a divisibly closed, archimedean, supertropical semifield,  $A
\triangleleft F[\Lambda],$ and $f \in F[\Lambda].$ Then $$f \epsc
A \quad \text{ iff } \quad  f \in \tsqrt{A}.$$
\end{thm}

The rest of this section is devoted to the proof of the Theorem.
Since multiplying a polynomial $f$ by any element of $\tT$ does
not affect the root set $\tZ(f)$, and also does not affect whether
or not $f$ belongs to a given ideal, we often will replace a
polynomial by a scalar multiple.

\subsection{Proof of Theorem~\ref{Null2}}

In view of Proposition~\ref{Laurentembed}, we can embed
$F[\Lambda]= F[\la_1, \dots, \la _n]$ into $F[\Lambda,
\Lambda^{-1}] = F[\la_1,\la_1^{-1} \dots, \la _n,\la_n^{-1}].$ The
interplay between polynomials and Laurent polynomials is quite
useful, since it enables us to divide by monomials. The proof of
the Nullstellensatz is attained according to the following
sequence of steps, writing $f = \sum f_\bfi,$ a sum of monomials:

\begin{itemize}
\item [Step 1.] Take polynomials $g_{D_{\bfi}}\in A$ such that
$f\preceq _{D_{\bfi}} g_{D_{\bfi}}$, and modify them such that
$f_\bfi^m =  g_{D_{\bfi}}$ on $D_{\bfi}$, for each~$\bfi$.

\item [Step 2.] Increasing $m$ if necessary, one may assume for
every component $\bfi$ and all large enough $m$ that $f_{\bfi'}^m $ strictly dominates $
g_{D_{\bfi}}f_{\bfi}^{m-1}$ on $D_{\bfi'}$, for each component
$D_{\bfi'}$ neighboring $D_\bfi$.

\item [Step 3.] Increasing $m$ if necessary, one may assume for
$f$ tangible and all large enough $m$ that $f_{\bfi'}^m $  also strictly dominates
$g_{D_{\bfi}}f_{\bfi}^{m-1}$ on all components $D_{\bfi'}$
 not neighboring $D_\bfi$ (with respect to the same~$m$).

\item [Step 4.] Steps 2 and 3 imply for $f$ tangible that $f^m
\lmodgLa \sum g_{D_{\bfi}}f_{\bfi}^{m-1}$, where the sum is taken
over the components of the tangible essential monomials of $f$.
This means $f \in \tsqrt{\sum _\bfi F[\Lambda] g_{D_{\bfi}}}.$

 \item [Step 5.]
For general $f$, one may assume   that $f_{\bfi'}^m  \lmodgLa
g_{D_{\bfi}}f_{\bfi}^{m-1}$ on all components $D_{\bfi'}$ of the
essential-tangible part of $f$.

\item [Step 6.] Step 5 implies $f^m \lmodgLa \sum
g_{D_{\bfi}}f_{\bfi}^{m-1}$, where the sum is taken over the
components of the tangible essential monomials of $f$. This means
$f \in \tsqrt{\sum _\bfi F[\Lambda] g_{D_{\bfi}}}.$

\end{itemize}

There is a  version of the Nullstellensatz, with somewhat easier
proof, using the Laurent polynomial semiring. Although the Laurent
version is a bit different from the polynomial version, as seen in
Example~\ref{someex}, we end this section with
Proposition~\ref{reduct1}, which could be used to link the two
versions.

\subsubsection*{Step 1 of the proof}

The idea is to match some power $f^m$ at each component with some
polynomial of the ideal $A$.

\begin{lem}\label{leadingcomp} Suppose  $D
= D_{f,  \bfi}$, $f \preceq _{D} g$ in $F[\Lambda],$ and $\la_j $
divides $ g.$ Then $\la_j $ divides $f_\bfi.$\end{lem}
\begin{proof} Suppose $\la _j$ does not divide $f_\bfi$, and $\bfa = (a_1, \dots, a_n) \in D .$ Let
$\bfb$ be the point obtained by specializing $a_j \mapsto \fzero.$
Then $f(\bfb)\le_\nu  f(\bfa)  = f_\bfi(\bfa),$ but $f_\bfi(\bfb)
= f_\bfi(\bfa)$, implying $\bfb \in  {D }.$ On the other hand,
$g(\bfb) = \fzero,$ contrary to hypothesis.
\end{proof}

For each component $D_{f,  \bfi}$ we choose a polynomial $
g_{D_{f, \bfi}} \in A$ for which $f \preceq _{D_{f, \bfi}}
g_{D_{f, \bfi}}$. Write $h_\bfi$ for the dominant monomial of
$g_{D_{f, \bfi}}$ on $D_{f, \bfi}$.
 By the lemma, taking $m$ large enough, we have $\deg_j f_\bfi ^m
> \deg _j g$ for each $j$, so $f_\bfi ^m = q_\bfi h_\bfi$ for a suitable monomial
$q_\bfi $. Replacing $g_{D_{f, \bfi}}$ by $q_\bfi g_{D_{f, \bfi}},$ we may
assume that $q_\bfi = \one,$ and $f_\bfi ^m = g_{D_{f, \bfi}}.$ This
achieves Step 1.

\subsubsection*{Step 2 of the proof}

Before proceeding to Step 2 -- Step 5, we note that Step 4 follows
formally from Steps~2 and 3, whereas Step 6 follows formally from
Step 5. Since we may replace $f $ by $f^m = \sum f_\bfi^m$ without
affecting the outcome of the theorem, we thus may assume that
$f_\bfi = h_\bfi$ (notation as in the proof of Step 1); in other
words, the leading monomials of $f$ and $g_{D_{f, \bfi}}$ on
$D_{f, \bfi}$ are the same. It is convenient at this stage to move
to the Laurent polynomial ring, replacing $f,$ $g_{D_{f, \bfi}}$
respectively by
$$ \frac f {f_\bfi},\quad \frac {g_{D_{f, \bfi}}}{f_\bfi}.$$
Thus, we assume that $f_\bfi = h_\bfi = \fone.$

We aim to verify Steps 2 and 5 for all sufficiently large $m$;
since there are only finitely many components, this means that we
need only prove these steps for a given single component $\bfi$,
which we fix for the remainder of the proof. This simplifies the
notation, since we can write $D = D_{f,\bfi},$ and $g = g_D.$ For
each component $D'$, we write $f_{D'}$ for the leading component
of $f$ and $h_{D'}$ for the leading component of $g_{D'}$. Thus
$f_D = h_{D} = \fone,$ which strictly dominates each $f_{D'}$ and
each $h_{D'}$ on~$D$.

\begin{figure}
\begin{minipage}{0.6\textwidth}
\begin{picture}(0,160)(60,70)
\includegraphics[width=4.0 in]{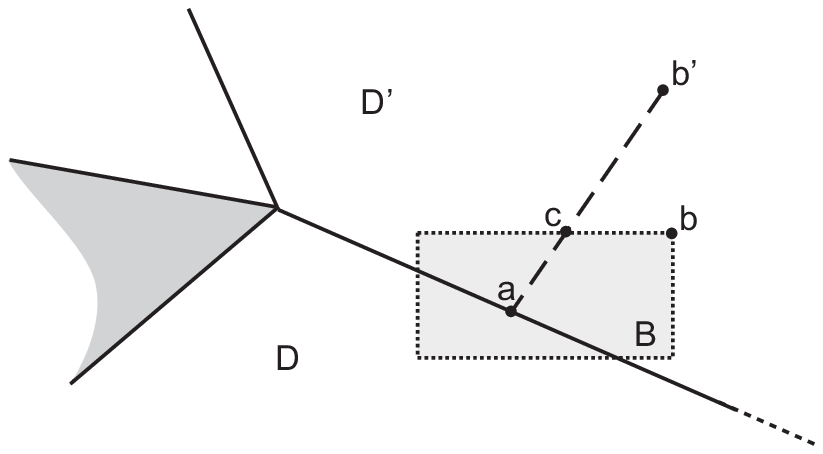}
\end{picture}
\end{minipage}

\caption{\label{fig:Box} }
\end{figure}

We need $m_0$ such that $f_{D'}^m$ strictly dominates  $h_{D'}$ on
$D'$ for all $m\ge m_0$, for each component $D'$ neighboring $D$.
We pick a point $\bfa$ on the border, and take a small enough
closed $\bfa$-box $B$, cf.~Definition~\ref{box}, such
that its extremal points lie in $D' \cup D$. Hence, some extremal
point $\bfb$ of the box $B$ lies in $D'$. By definition, $$
f_{D'}(\bfa) = f_D(\bfa) = \fone = h_D(\bfa)\ge_\nu
h_{D'}(\bfa),$$ whereas
$$f_{D'}(\bfb) >_\nu f_D(\bfb) = \fone.$$
Thus, since $F$ is archimedean, there is $m_0$ such that
\begin{equation}\label{boxed} f_{D'}(\bfb)^{m}
>_\nu h_{D'}(\bfb)\end{equation} for all $m \ge m_0$. Since there are only
finitely many extremal points on the box $B$, we may assume that
\eqref{boxed} holds for all extremal points of $B$ and thus, by
Lemma~\ref{convex1}, for all points in   $B \cap \overline{D'}$.

But now for any point $\bfb' \in D',$ the path from $\bfa$ to
$\bfb'$ passes through $B \cap \overline{D'}$. It starts at
$\bfa,$ where $ f_{D'}(\bfa)^m = \fone = h_{D'}(\bfa),$ and then
passes through some point $\bfc$ where $f_{D'}(\bfc)^{m}
>_\nu h_{D'}(\bfc)$ (see Figure \ref{fig:Box}), so applying Lemma~\ref{convex1} to $\frac
{f_{D'} ^{m}}{h_{D'}},$ we also have \eqref{boxed} at the point
$\bfb',$ as desired.

\subsubsection*{Step 3 of the proof}

This is the subtlest part of the proof. Pick a point $\bfa \in D.$
Then for every monomial $f_{\bfi'} \ne f_\bfi,$
$$f_{\bfi'}(\bfa) < _\nu  f_{D}(\bfa) = \fone,$$ so picking $m$ large
enough, we have \begin{equation}\label{xed} f_{\bfi'}(\bfa)^m <_\nu
h(\bfa)\end{equation} for {\it every} monomial $h$ of $g.$
Furthermore, since there are only finitely many components, we may
assume that \eqref{xed} holds for every monomial $f_{\bfi'}$ of
$f$ (other than $f_\bfi$).

Now we fix $m$ for the remainder of the proof, and consider the
polynomial $\tilde f = f^m +g.$ We take the components with
respect to $\tilde f$; this is just a subdivision of the
components with respect to $f$. We call a component $D_{\bfi'}$
\textbf{good} if $f_{\bfi'}^m $ strictly dominates the polynomial
$g_{D_{\bfi}}f_{\bfi}^{m-1}$ on $D_{\bfi'}$. Otherwise, we say
that  $D_{\bfi'}$  is \textbf{bad}, which means
$g_{D_{\bfi}}f_{\bfi}^{m-1}$ dominates $f_{\bfi'}^m $ on
$D_{\bfi'}$.

Our aim is to prove that all components are good; we assume that
some component $D'$ is bad and reach a contradiction.  Let $L$ be
the set of good components. Take a point $\bfa$ in $D$, and a
point $\bfb $ in $D$; adjusting them if necessary, we may assume
that the path $p$ from $\bfa$ to $\bfb$ passes only from neighbor
to neighbor. (In other words, any point on $p$ not lying inside a
component lies on the common border of two neighbors.)

By definition, $D$ and its neighboring components lie in $L$; we
take the first bad component traversed by our path $p$, and
clearly may replace this component by $D'$. Thus, we may assume
that $p$ contains a point $\bfb'$ on the border of $D'$ with some
component of $L$, which we call $D''$. Say $f_{\bfi'}$ is the
dominant monomial of $f$ on $D'$, and  $f_{\bfi''}$ is the
dominant monomial of $f$ on $D''$. Hence $f_{\bfi'}(\bfb') =
f_{\bfi''}(\bfb')$ since it lies on the border.

We know that $$f_{\bfi'}^m (\bfa)<_\nu
g_{D_{\bfi}}(\bfa)f_{\bfi}^{m-1}(\bfa).$$ By hypothesis, since
$D'$ is bad,
$$f_{\bfi'}^m (\bfb) \le_\nu
g_{D_{\bfi}}(\bfa)f_{\bfi}^{m-1}(\bfa).$$ On the other hand, since $D''$
is good,
$$f_{\bfi'}^m (\bfb') = f_{\bfi''}^m(\bfb') \ge_\nu
g_{D_{\bfi}}(\bfa)f_{\bfi}^{m-1}(\bfa).$$

Since $\bfb'$ lies between   $\bfa$ and $\bfb$ on the path $p$, we
have a contradiction to Lemma~\ref{convex1}.

\subsubsection*{Step 5 of the proof}  Let $\hat f$
denote the tangible polynomial having the same $\nu$-value as $f$.
We relabel the components of $\hat f$   as $D_1, \dots, D_q.$ Some
of these remain as components of $\tZ_{\operatorname{tan}} (f)$;
we call these components ``true''. Other components are in the
root set of $f$ (because of its extra ghost coefficients) and thus
belong to $Z_{\operatorname{tan}} (f)$; and
 we call these components ``fictitious.'' For each true
component $D_{f,  \bfi},$ take a polynomial $g_{D_{f,  \bfi}}\in
A$ with a component   containing $D_{f,  \bfi},$ and let $\bar f$
be the (tangible) sum of these monomials $ f_ \bfi,$ from the true
components. Applying Step 3 to $\bar f,$ we see that $f_{\bfi'}^m
\lmodgLa \hat f_{\bfi'}^m $ strictly dominates
$g_{D_{\bfi}}f_{\bfi}^{m-1}$ on all true components $D_{\bfi'}$.

\subsubsection*{Step 6 of the proof} This is formal: Clearly $f^m \lmodgLa \bar f^m = \sum
g_{D_{\bfi}}f_{\bfi}^{m-1}$ (summed on the true components). This
concludes the proof of the Nullstellensatz.

\begin{example}\label{exdiv} Let $R = F[\lambda]$, and  consider the polynomial $f = \la^2 + 6 ^\nu \la +7$, whose tangible root set is the
interval $[1,6]$.
\begin{enumerate} \eroman
    \item If $g = \la  + 4,$ whose tangible root set is $\{ 4 \},$ then
$$f  = (\la + 3)g + 6^\nu \la,$$  so $f\lmodgLa (\la + 3)g $. \pSkip

    \item $g = \la^2 + 4 ^\nu \la +6$, whose tangible root set is the
interval $[2,4]$, then $$f^2 =  \la^4 + 6 ^\nu \la^3 + 12^\nu
\la^2 +13^\nu \la + 14$$ and $(\la^2+8)g  = \la^4 +4^\nu \la^3 + 8
\la^2 + 12^\nu \la + 14,$ implying $$f^2  \lmodgLa (\la^2+8)g .$$
\end{enumerate}

\end{example}

\begin{example}\label{quad1} Generalizing Example~\ref{exdiv}, suppose $f = \la^2 + a_2 ^\nu \la
+a_1a_2,$ for $a_1, a_2$ tangible. Then, for $a$ tangible, $\la
+a$ supertropically divides $f$ iff $a_1 \nule a  \nule a_2 .$
Indeed, suppose $f \lmodgLa (\la +a)q .$ Then $q = \lm + b$,
where, comparing constant terms, we see that $a b\nucong a_1a_2.$
Now matching the coefficients of $\la$ shows that $\max\{b^\nu,
a^\nu\} \le a_2^\nu,$ and thus $\min\{b^\nu, a^\nu\} \ge
a_1^\nu.$\end{example}

 \subsubsection{An explicit connection to the Laurent polynomial
semiring}

The following result links the Nullstellensatz to Laurent
polynomials, and could be used to provide an alternate proof.

\begin{prop}\label{reduct1} Suppose $hf^m =g = \sum _\bfi h_\bfi g_\bfi,$ for some
monomial $h$ and some $m$, where $g_\bfi$ are polynomials such
that
 $f \preceq _{D_\bfi} g_\bfi$ in $F[\Lambda],$ for each $\bfi$.
 Then there is
some $m'
>m$ for which $f^{m'} \in \sum _\bfi F[\Lambda] g_\bfi.$ \end{prop}
\begin{proof}  Write $f = \sum f_\bfi$ as a finite sum of
monomials, and write $h = \prod _{u=1}^n \la _{j_u}^{k_u}$. We
proceed by induction on $\deg h$. Let
$$U = \{ u : k_u \ne 0\} \subseteq \{1, \dots, n\}.$$ Pick $u \in
U,$ and write $\la = \la_{j_u},$ and $k = k_u$. By
Proposition~\ref{divide2}, $\la$ divides $h_\bfi g_\bfi$ for each
$\bfi.$

Let $J_u = \{ \bfi:   \la$ divides  $h_\bfi\}$. If $\bfi \in J_u,$
then we can write   $h_\bfi = \la h'_\bfi,$ for some monomial
$h'_\bfi,$ and $f^{mk} = {h'_\bfi}^k g_\bfi^k.$

If $\bfi \notin J_u,$ i.e,  $\la$ does not divide $h_\bfi,$ then
$\la$ divides $g_\bfi;$ hence,
 by
Lemma~\ref{leadingcomp}, $\la$ divides $f_\bfi,$ and  $\la^{k}$~
divides $f_\bfi^{k}$. Writing $f_\bfi^{k} = \la^k f'_\bfi,$ we
then have $f_\bfi^{m+k} = \la^k f_\bfi^m  {f'_\bfi}^k.$ Applying
this for each $u$ and taking $\tilde k = \sum k_u$ yields
$$f_\bfi^{m+\tilde k} = h_\bfi {f'_\bfi}^{\tilde k} g_\bfi.$$

Hence, taking $m'$ to be the maximum of $\{m+\sum k_u, \  m\prod k_u: u \in
U\},$ we see that each $f_\bfi^{m'} \in F[\Lambda] g_\bfi,$
implying by Corollary~\ref{Prop:powOfPols} that $$f^{m'} = \sum
{f_\bfi}^{m'} \in \sum  F[\Lambda] g_\bfi.$$\end{proof}

\section{Factorization of polynomials}

One way of determining  the roots of a polynomial is by factoring
it first into irreducible factors, which is the main theme of this
section. Unfortunately, factorization over the max-plus algebra is
quite cumbersome. As noted in the introduction, the polynomial
$\la^2 +4$ is irreducible even though 2 is a root. One can bypass
this difficulty by factoring up to $e$-equivalence. Since a
polynomial $f$  has precisely the same roots as its essential
part, we  always study divisibility and factorization in the sense
of e-equivalence and $(\nu,e)$-equivalence, cf.~
Definition~\ref{def:R-equivalent}.

Difficulties are still  encountered when studying
factorization, especially if one wants to understand polynomials
having cluster roots, so let us start this section with a brief
guide to its results. We embark on a thorough investigation of
factorization of polynomials over a supertropical semifield $F$,
with emphasis on the factorization of a polynomial $f(\la)$ in one
indeterminate $\la$. This requires finding the appropriate
representative of $f$ in $F\pl \la \pr$. Although one could work
with essential polynomials, the computations do not match so well,
as we see in Example \ref{exmp:eq} (iii), and we look for a more
convenient representative. The answer comes from a description of
the polytope of a polynomial in \S\ref{sec:polytope}, which leads
us to the notion of a \textbf{full polynomial} (Definition
\ref{full0}).

Tangible full polynomials behave quite like polynomials in
classical algebra, having unique factorization into linear
factors; cf.~Theorem~\ref{factor1}. However, nontangible
polynomials behave more poorly, and unique factorization is
violated in Example~\ref{exmp:factorization}. We recover unique
factorization by turning to the factorization ``minimal in
ghosts,'' which is interpreted in Proposition \ref{mingho} in
terms of the root set of the polynomial. So from this point of
view, one can ``understand'' factorization in terms of roots, even
when some of the roots are not \singular.

The situation for several indeterminates is more disturbing at
first, since a serious violation of unique factorization is given
in Example~\ref{rem:vander}. However, this also can be understood in
terms of root sets, and by rewriting factorizations in terms of binomials, cf.
Theorem~\ref{permprime}, which has the geometric interpretation
that every root set can be embedded naturally in a union of
hyperplanes. The remainder of this section contains examples which
clarify the geometric content of this theorem.

\subsection{General observations about factorization}

\begin{defn}\label{def:rRed}  A polynomial $g\in R\pl \Lambda \pr$ \textbf{$e$-divides} $f$, written $g \rDiv
f$, if $f \eqR qg$ for some polynomial~$q.$ (In other words, the
image of $g$ in $R[\Lambda]$ divides the image of $f$.) A
polynomial $f $ is said to be \textbf{e-reducible} if $f \eqR gh$
for some $g,h \in R\pl \Lambda\pr $ each not $e$-equivalent to a
nonconstant; otherwise $f$ called is \textbf{e-irreducible}. The
product $f \eqR q_1 \cdots q_s$ is called a \textbf{factorization}
of~$f$ \textbf{into irreducibles} if each of the $q_i$'s is
$e$-irreducible.
\end{defn}

\begin{rem}
For $f,g\in  R \pl \Lambda \pr$, if $f  \rDiv g$ and $g \rDiv h,$
then $f \rDiv h.$
\end{rem}

\begin{example}\label{exmp:div} (Logarithmic notation)
\begin{enumerate}
 \item $(\lm +1) \rDiv (\lm^2 +  2 \lm + 3)$, since $\lm^2 +
2 \lm + 3 = (\lm +1)(\lm+2)$;
\pSkip \item $(\lm +1) \rDiv (\lm^2 + 2)$, in view of Example
\ref{exmp:eq} (iii);
 \pSkip \item $\sum {f_i}^k \rDiv \sum f_i ^m$ for each $m\ge k \ge 1,$
in view of Corollary~\ref{Prop:powOfPols}.
\end{enumerate}
\end{example}

\begin{prop}\label{prop:div} The polynomial $g \rDiv f$, iff the essential part of $qg$
equals the essential part of $f$ for some polynomial $q.$
\end{prop}
\begin{proof} For each condition, the essential parts have to be $e$-equivalent, and thus equal, monomial for monomial.
\end{proof}

\begin{cor} The polynomial $g \rDiv f$, iff the essential part of $g$
e-divides the essential part of $f$ with respect to the
multiplication of $R \pl \Lambda \pr$.
\end{cor}
\begin{proof} By Proposition \ref{prop:div}.
\end{proof}

\begin{example}\label{linearirr} The following is a list of
all $e$-irreducible polynomials in $F[\la]$, together with their tangible root
sets  and their tangible complements.
 We normalize, to assume that the
leading coefficient is $\fone$ or $\fone^\nu$. For convenience, we
also assume throughout  that $\nu_{\tT}: \tT \to \tG$ is 1:1 and
$a,b \in \tT$.
\begin{description}
\item[\emph{Type I.a}] $f = \la$; $\mathcal Z (f) =\tGz$, whereas
$\tZ_{\tng} (f) = \{\fzero\}$.  The tangible complement is all
of~$\tT $. \pSkip
 \item[\emph{Type I.b}] $f = \la + a;$ $\tZ_{\tng} (f) = \{ a \}.$
 The tangible
 complement is the union of two open rays. \pSkip
 \item[\emph{Type II} (right ghost)] $f = \la + a^\nu$;
 $\tZ_{\tng} (f) = \{ b \in \tT :  b \nule a \},$
 the closed left ray up to $a$. The tangible  complement is the open right ray from
 $a$. \pSkip
 \item[\emph{Type III} (left ghost)] $f = \la ^\nu + a $ ; $\tZ_{\tng} (f) = \{  b\in \tT : b \nuge a \},$
 the closed right ray from $a$. The tangible complement is the open left ray to
 $a$. \pSkip
 \item[\emph{Type IV}] $f = \la ^2 + b ^\nu \la + ab  $ for $a  \nul b$;
   $\tZ_{\tng} (f) =  \{ d \in \tT :  a \nule d \nule b \},$
 the closed interval from $a$ to $b$. The tangible complement is
 comprised of
 two open rays, one open left to $a$ and the other open right from~$b$.
\end{description}
\end{example}

\subsection{The geometry of  polynomials}\label{sec:polytope0}

Important as they are to our theory, essential polynomials miss
the mark when computing factorizations, since we have to continue
to take essential parts when computing products.  We want a
different representative inside $F\pl \Lambda \pr$ that will more
accurate reflect this product. In order to put the algebraic
theory into perspective, we turn to a key geometric interpretation
of polynomials, which enables us to overcome this difficulty.

\subsubsection{The polyhedron of a polynomial}\label{sec:polytope}

We identify each monomial $\a_\bfi \lm^{\bfi}$ (for $\bfi = (i_1,
\dots, i_n)$) with the point
$$ (\bfi, \a _\bfi^\nu)= (i_1, \dots, i_n,  \a _\bfi^\nu) \in \Net
^{(n)} \times \tG \subset \R^{(n)} \times \cltG,$$ where $\cltG$
is the divisible closure of $\tG.$ For any polynomial $f =
\sum_\bfi \a _\bfi \Lm^\bfi \in R \pl \Lambda \pr,$ we define the
polytope $C_f$ determined by the convex hull of the points
$$\{ (\bfi, \a _\bfi^\nu): \bfi \in \supp (f)\}.$$
 The upper part of $C_f$ is called the \textbf{essential polyhedron of} $f$,
and is denoted $\oC_f$, whose
 vertices we call
the \textbf{upper vertices} of $C_f.$ The points of $\oC_f$ of the
form $\{(\bfi, \a ^\nu): \bfi \in \N ^{(n)}\}$ are called
\textbf{lattice points of} $f$. For example, when $f = \la^2 + 2,$
its lattice points are $(2,0^\nu)$, $(1,1^\nu)$, and $(0,2^\nu)$.
A vertex $(\bfi, \a _\bfi^\nu)$ of $\oC_f$ is called a
\textbf{tangible vertex} if $ \a _\bfi$ is tangible; otherwise the
vertex is called a \textbf{ghost vertex}.

(The  essential polyhedron  of $f$ should not be confused with the
graph of $f$ itself, which is in a sense dual; in the graph of
$f$, the vertices themselves correspond to ordinary roots of $f$.

The structure described above can be stated in the context of the
Newton polytope as described in \S\ref{compare}. In this sense the
convex hull, $\Delta(f)$, of the $\bfi$'s in $\supp(f)$ describes
the Newton polytope of $f$ and, by taking the projection, by
deleting the last coordinates, of the non-smooth part of $\oC_f$
(that is a polyhedral complex) on $\Delta(f)$, the induced
polyhedral subdivision $S(f)$ of $\Delta(f)$ is obtained. A dual
geometric object having combinatorial properties is thereby
produced. This object plays a major rule in the classical tropical
theory; cf.~
\cite{Itenberg03,MikhalkinEnumerative,Mikhalkin01,Shustin1278,SpeyerSturmfels2004}.)

The following result shows how the roots of a polynomial
correspond to its essential polyhedron.  As mentioned earlier,
when studying the polyhedron, we use the additive (logarithmic)
notation for $\tG$.

\begin{prop}\label{convex} Over a divisibly closed supertropical domain
$R$, any polynomial $f$ is weakly $(\nu,e)$-equivalent to the
polynomial corresponding to~$\oC_f$, and $\oC_{\ef} = \oC_f$.
\end{prop}

\begin{proof} We claim that we may discard any monomial whose corresponding
point is not a vertex of $\oC_f$. Indeed, we may pass to the
$\mathbb N$-divisible group $\overline{\tG},$ and suppose
$(\bfj,\a_\bfj^\nu)$ lies below the simplex connecting the
$(\bfi_u,\a_{\bfi_u}^\nu)$; that is,  $\bfj = \sum _u t_u \bfi_u,$
where each $t_u = \frac {m_u}m,$ for $m, m_u \in \Net$ with each
$m_u \le m,$ but also with  $$m \a_\bfj \nule \sum m_u \a_{i_u}
.$$ But then, using  logarithmic notation, for any point $\bfa =
(a_1, \dots, a_n)$, the $\nu$-value of $\a _\bfj \La ^\bfj$ at
$\bfa$ is
\begin{equation}\label{thisone}\a_\bfj \bfa^\bfj \nucong
\a_\bfj + \sum _{\ell =1}^n {j_\ell}  a_\ell \nule \sum _u t_u
\a_{\bfi_u} + \sum _u \sum _{\ell =1}^n {t_u}{i_u}_\ell a_\ell
\nucong  \end{equation} $$\sum _u t_u \sum _{\ell =1}^n
(\a_{\bfi_u} + {i_u}_\ell a_\ell) \nucong \sum _u t_u \,
\prod_{\ell =1}^n \a_{i_u} a_\ell^{i_{u\ell}} \nucong \sum _u t_u
\a _{\bfi_u} \bfa^{\bfi_u}.$$
   This shows that
any point under $\oC_f$
 is superfluous.

 On the other hand, for any \singular\ root $\bfa$, we need
 there to be $\bfi\ne \bfj$ for which
 $$\a_\bfi \bfa^\bfi  \nucong  \a _\bfj
 \bfa^\bfj,$$
 or $\sum _{\ell = 1}^n (j_\ell-{i}_\ell) a_\ell \nucong  \a_\bfi - \a _\bfj;$ this
 defines a hyperplane, which corresponds to a face of
$\oC_f$. Conversely, any two vertices
 $(\bfi,\a_\bfi^\nu)$ and $(\bfj,\a_\bfj^\nu)$ define the same
 hyperplane (of dimension $n-1)$ of roots, implying that these roots are \singular.

 It remains to show that $\oC_f$ is also the essential polyhedron of
 $\ef.$ This is true since the vertices
 $(\bfi,\a_\bfi^\nu)$ and $(\bfj,\a_\bfj^\nu)$ defining a given
 hyperplane are essential, in view of Remark~\ref{arch1} , since
 any small increase of $\a_\bfi^\nu$ in the appropriate direction will make
 $\nu(\a_\bfi \bfa^\bfi)$ greater than the value of any other monomial
 of~$f$.
\end{proof}

\begin{rem}\label{rmk:eClass} (For $R$ a divisibly closed supertropical domain, with $\nu_\tT$ 1:1.)
Proposition \ref{convex} shows that $f \eqR g$  if and only if
$\oC_f =\oC_g$ with vertices of the same parity, iff $\ef = \eg$.

\end{rem}

 The \singular\
roots of $f$ correspond (up to $\nu$-equivalence) to the faces of
$\oC_f$.
 In general, the inessential part of
$f$ does not appear  as vertices of  $\oC_f$.  We say that the
monomial $h=\a _\bfi \La ^\bfi$ is \textbf{quasi-essential} for~
$f$ if $(\bfi,\a^\nu _\bfi)$ lies on $\oC_f$ and is not a vertex.
This has the following interpretation:

\begin{rem}
An inessential monomial is  quasi-essential if any (arbitrarily
small) increase of the $\nu$-value of its coefficient makes it
essential.
\end{rem}

\begin{rem}\label{rmk:quasiEssen}
Given a polynomial $f = \sum \al_\bfi \Lm^\bfi$, and assume that
$h_\bfi = \al_\bfi \Lm^\bfi$ is a monomial of $f$ for which
$(\bfi,\al_\bfi^\nu) = \sum_j t_u(\bfi_u,\al^\nu_{\bfi_u})$ for
some $\bfi_1,\dots,\bfi_m$, where $t_u \in \Q^+$ and $\sum_u
t_u~=~ 1$. Then $h_\bfi$ is inessential for~$f$; when all the
corresponding $h_{\bfi_u}$ are essential, then $h_\bfi$ is
quasi-essential. This means that $(\bfi,\al_\bfi^\nu)$ lies on
$\oC_f$ (or under) but is not a vertex. (The proof is as in
Equation \Ref{thisone}.)
\end{rem}

\subsubsection{Full polynomials}  Having the geometric
interpretation in hand, we  are ready for our main class of
polynomials. Essential polynomials slightly miss the mark, since
the polyhedron of an essential polynomial may lack interior
lattice points.

\begin{defn}\label{full0} A polynomial $f \in R \pl \Lambda \pr$ is called
\textbf{full} if every lattice point lying on $\oC_f$ corresponds
to a monomial of $f$ that is either essential or quasi-essential,
and furthermore, the coefficient of each quasi-essential monomial
is a ghost; a full polynomial $f$ is \textbf{tangibly-full} if $f$
is also essential-tangible. The \textbf{full closure} $\tilde f$
of $f$ is the sum of $\ef$ with all the quasi-essential ghost
monomials interpolated from the polyhedron $\oC_f$.
\end{defn}

In this paper, we only consider full polynomials in the case that
$n=1;$ i.e., $f = \sum _{i=0}^m r_i \la^i.$ Here $f(\zero) = r_0,$
which thus is essential in $f$. whenever $r_0 \ne \zero,$ and
likewise the monomial $r_m \la^m$ is essential in $f$. The
polynomial $f$ is full iff the intermediate monomials $f_i \la^i$
are essential or quasi-essential for all $0 < i <m.$

\begin{rem} Geometrically, the full closure $\tilde f$ has a monomial
corresponding to each lattice point of the essential polyhedron of
$f$. However, one needs to take care: The full closure is only
defined over $\cltG.$ For example, if $ F = D(\tG)$ where $\tG =
(\Z, +),$ then the essential polynomial $\la ^2 +1$ is defined
over $F$ but its full closure, $\la ^2 + {\frac 12}^\nu \la + 1$,
is  defined not over $F$, but  over $\clF$.
\end{rem}

Thus, by definition, the full closure of a tangible polynomial is
tangibly-full.

\begin{example} The polynomials $\lm^2 + 2^{\nu}\lm +
4$, $ \lm^2 + 2^{\nu}\lm + 4^\nu$, and $ 0^\nu \lm^2 + 2^{\nu}\lm
+ 4^\nu$ are full. However, the polynomial $f = \lm^2 + 2\lm + 4$
is not full, since the middle term is not essential but is
tangible; the monomial $2\lm$ is quasi-essential for $f$, and the
full closure of $f$ is $\lm^2 + 2^{\nu}\lm + 4,$ which is
tangibly-full.

The polynomial $\la ^2 + 3^\nu \la +4$ is full, and essential, but
not tangibly-full.\end{example}

\begin{rem}\label{Fullofit} The full closure $\tilde f$ is $e$-equivalent to $f$,
for any polynomial $f$. Conversely, different full polynomials
cannot be $e$-equivalent. Thus, any class of polynomials in $R \pl
\lm_1, \dots, \lm_n \pr$ has a unique \textbf{full representative}
$\tilde f$, and we can view $R[\lm_1, \dots, \lm_n]$ as the set of
full polynomials, under the operations
$$f+g = \widetilde {f\oplus g}, \qquad fg = \widetilde {f\odot
g}.$$
\end{rem}
Thus,   we have identified another canonical representative for
each $e$-equivalence class in $R \pl\lm_1, \dots, \lm_n \pr$, cf.
Remark \ref{rmk:eClass}.

\subsection{The essential graph of coefficients}

We utilize  the results of the previous subsection, for the case
$n=1$. For a polynomial $f = \sum_{i=0}^t \a _i \la ^i\in F \pl
\la \pr,$ take the sequence $\a_0^\nu, \dots, \a_t^\nu$, and the
graph $G_f$ whose vertices are the
  points
\begin{equation}\label{eq:coefficientsGraph}
    (0,\a_0^\nu),
(1,\a_1^\nu),\quad \dots, \quad  (t,\a_t^\nu).
\end{equation}
 In the case
of the polynomial semiring over a supertropical semifield, any
polynomial of degree $t$ is determined by the graph $G_f$ (having
at most $ t$ edges). Note that $C_f$  is the convex hull of $G_f$,
cf. \S\ref{sec:polytope}. The \textbf{essential graph  of
coefficients}, $\oC_f$, is constructed as the top edges of $C_f$.
(This is the essential polyhedron of a polynomial in one
indeterminate.) When the polynomial $f$ is full, the graph of
coefficients is already essential. As we shall see, these edges
correspond to ordinary roots of $f$.

\begin{rem}\label{incsl} The slopes of the edges
of  the graph $\oC_f$ of $f = \sum_i f_i$  decrease as we move to
the right, since $\oC_f$ is convex.\end{rem}

\begin{example}\label{Exess} $f = (\la+1)^2(\la +2) = (\la^2 +1^{\nu}\la
+2)(\la +2) = \la^3 +2\la^2 +3^\nu \la + 4.$ Then the graph of
coefficients has upper vertices $  (0,4^\nu), (1,
3^{\nu}),  (2,2^\nu), (3,0^\nu),  $ and the  convex hull is determined by the
upper vertices $  (0,4^\nu), (2,2^\nu), (3,0^\nu),  $ thereby
corresponding to the polynomial $\la^3 +2\la^2  + 4,$ the
essential part of $f$. \end{example}

Note that $\oC_f$ may contain lattice points not corresponding to
monomials of the original polynomial~$f$. For instance, in Example
\ref{Exess}  the point  $(1, 3^{\nu})$  lies on an edge of
$\oC_f$, although it is not a vertex.

\begin{prop}\label{convex2} For  $f \in F \pl \lm \pr$, the $\nu$-equivalence classes of \singular\
roots
  correspond to the negations of the slopes of the edges
of $\oC_f$, as to be described in the proof. Such roots exist
whenever $F$ is divisibly closed.
\end{prop}

\begin{proof}
For any \singular\ root $\ra$ of $f$, we need
 $i  < j$ for which
\begin{equation}\label{root11} \a _j{\ra^{j-i}} \ \nucong \ \a_i ,\end{equation}
 i.e.,
 in logarithmic notation,
 \begin{equation*}\label{root1}(j-i) \ra \ \nucong \ \a_i - \a _j.\end{equation*}
  This means  $\ra$ must satisfy
\begin{equation*}\label{root1}\ra \nucong \frac { \a_i - \a
_j}{j-i} \ \nucong \  -\frac { \a_j - \a _i}{j-i},\end{equation*}
the negation of the slope of an edge of the graph of coefficients;
conversely, any such tangible root $\ra$ is \singular.\end{proof}


\subsubsection{Factoring tangible polynomials in one indeterminate}
Assume that $f \in R \pl \lm \pr,$ for a supertropical domain
$R$.

The tropical theory of  polynomials in one indeterminate is rather
close to the classical theory, when we work with
tangibly-full polynomials.

\begin{rem}\label{easyf} $ $
\begin{enumerate} \eroman
    \item
 Suppose $f = pq$ for $f,p,q \in R[\la]$. Then $\ra$ is a root
of $f$ iff $\ra$ is a root of $p$ or $q$. (Indeed $f(\ra) =
p(\ra)q(\ra)$, which is in $\tG$ iff one of the factors is in
$\tG$.) \pSkip

\item  As a special case of (i), if $f = (\la + \ra)q$ for $f,q \in
R[\la]$, then $\ra$ is a root of $f$. \end{enumerate}
\end{rem}

To start a theory of factorization, we need a converse for
Remark~\ref{easyf}: Given a tangible root $\ra$ of $f$, we would
like $\la+\ra$ to divide $f$. This issue is surprisingly tricky,
and also leads us to the question of ``multiple roots,'' so the
following calculation will be useful.

\begin{example}\label{power} Write $\a^2$ for $\a\a$ (which is
computed in logarithmic notation as $2\a$); likewise $\a^3 =
\a\a\a.$
\begin{enumerate} \eroman
    \item
    $(\la + \a)^2 = \la^2 + \a \la + \a \la + \a ^2 = \la^2 +
    \a ^\nu
\la + \a^2.$
\pSkip \item
 By Proposition~\ref{lem:powOfPol},
   $(f+g)^m  \eqR f^m + g^m;$ in particular, $(\la + \a)^m \eqR \la ^m + \a ^m$.
   \end{enumerate}
\end{example}

\begin{lem}\label{factor0} Suppose $R=\RGnu$ is a supertropical domain.
If $\ra$ is an \individual\ root   of
  $f\in R \pl \la \pr$, then $(\la +\ra)$ $e$-divides $b f$ for some   $b \in \tT$.
  In particular, when $R$ is a supertropical semifield, $(\la +\ra)$ $e$-divides $f$.
\end{lem}
\begin{proof} Write $f = \sum _{i=0}^t \a
_i \la ^i$. By Proposition~\ref{convex2}, explicitly
Equation~\eqref {root11}, there are $j<k$ such that, in semiring
notation, $$\a _j \nucong     \a_k {\ra}^{k-j} .$$

First we assume that $k = t,$ which means $a$ is the $\nu$-maximal
root of $f$, and
$$f = \a _t \la ^t +
b\la ^j +g,$$   where $b\nucong  \a _{t}  a^{t-j} $ and $g= \sum
_{i=0}^{j-1} \a _i \la ^i.$ Note that $b$ is tangible since the
root $a$ is \individual. One computes that
$$\a_t(\la + a)^{t-j}( b \la ^{j}+g) \eqR  (\a_t\la^{t-j} + b)( b \la ^{j}+g) \eqR b(\a _t \la^t +  b
\la^{j} + g) +\a _t\la^{t-j} g = bf+ \a _t\la^{t-j} g ,$$ so we
need only show that $\a _t\la^{t-j} g$ is inessential in the right
hand side. When $c<_\nu  a,$   $$ \a _t c^{t-j}  g(c) <_\nu \a _t
a^{t-j}  g(c) \nucong b g(c).$$ When $c \ge_\nu  a ,$ then  $
g(c)\le_\nu bc^{j}$ (since the monomial $b\la^{j}$
dominates $g$ for all substitutions to elements of $\nu$-value
greater than the largest root), so $$ \a _t c^{t-j} g(c) \le _\nu
\a _t c^{t-j} b c^{j} \le_\nu b\a_t c^t.$$ We conclude in each
case that $\a _t \la ^{t-j} g$ is inessential.

For $k <t$,  Theorem~\ref{convex2} implies $\ra$ is a root of $f_1
= \sum _{i=0}^{t-1} \a _i \la ^i,$ (since $\ra$ is the negation of
the slope of some other edge of $\oC_f$), so by induction on
degree, there is $g(\la)$ of degree $t-2$ such that $(\la + \ra
)g$ has essential part $bf_1$. But then $(\la + \ra  )(\a _t b\la
^{t-1} +g)$ has the same essential part as $\a _t b\la ^t  + (\la
+ \ra)g$, which has the same essential part as $b f$.
\end{proof}

Iterating Lemma \ref{factor0}, we get

\begin{thm}\label{factor1} Suppose $\FGnu$ is an $\Net$-divisibly closed, supertropical semifield. Then
any  polynomial $f \in \tT  \pl \la \pr$ is $e$-equivalent to the
tangible part of a product $\prod_j (\la +\ra _j)^{i_j},$ where
the $\ra _j$ range over \singular\ roots of $f$.\end{thm}

\begin{cor}\label{corfactor1} If $F = \clF$, then any essential-tangible polynomial  $f$ can be
factored uniquely to a product $\prod_j (\la +\ra _j)^{i_j},$
where the $\ra _j$  range over the \individual\ roots of $f$.
\end{cor}
\begin{proof} The polynomial $\prod_j (\la +\ra _j)^{i_j}$ is
full, and uniqueness is clear. \end{proof}

 We would like to think of the $i_j$ as the multiplicities of the
roots, but, as usual, care is required. (We only handle essential-tangible polynomials
here, since the general case is considerably subtler.)

\begin{example} Suppose $F = (\R^\times,{\R^+}, \nu),$ where $\nu$ is the usual absolute value on
$\R^\times.$  The polynomial $\la ^ 2 + (-4)$ is $e$-equivalent to
$\la ^2 + 2^\nu \la + (-4) = (\la +2)(\la +(-2)),$ but
intuitively, since $(-2)^\nu = 2^\nu,$ we should say that the root
$2$ has multiplicity $2.$\end{example}

\begin{defn} For $f \eqR \prod (\la +\ra_j)^{i_j},$ the
\textbf{multiplicity} of some root $a$ of $f$ is $$\sum \{ i_j:
\ra_j \nucong a \}.$$\end{defn}

\begin{rem}
When $\nu_{\tT} $ is 1:1, the multiplicities are just the $i_j.$
\end{rem}

\begin{cor}\label{factor2} The multiplicities  of the roots
of $f$ in Theorem~\ref{factor1} (and Corollary \ref{corfactor1})
are precisely the lengths of the edges of the convex hull of the
essential  graph of coefficients $\oC_f$ of~$f$.\end{cor}
\begin{proof} Just repeat the proof of Lemma~\ref{factor0}, noting
that the result holds for $\la ^i + \a ^i$ by Example~\ref{power}.
\end{proof}

When $F$ is not divisibly closed, the following reduction is
useful.

\begin{prop}\label{enlarge2} Suppose $F$ is a supertropical semifield, and $f,g
\in F[\Lambda]$. If $f$ $e$-divides $g$ in $\clF[\Lambda]$, then
$f$ $e$-divides $g$ in $F[\Lambda] .$
\end{prop}
\begin{proof} Otherwise write $\eg = \ef \eh$ and let $\a_\bfk \Lm^{\bfk}$ be the lowest
order monomial (under the lexicographical order of $\N^{(n)}$) of
$\eh$ for which $\a_{\bfk} \notin \tT.$ Since it is essential,
there must be some value $\bfa$ for which $h(\bfa) = \a_\bfk
\bfa^\bfk.$ But then $f(\bfa) = g(\bfa)\a_\bfk \bfa^\bfk,$
implying some monomial of $f$ has the form $g_{\bfi} \a_{\bfk}
\Lm^{\bfk},$ for a suitable monomial $g_{\bfi}$ of $g.$ Thus, we
may assume that $f$ and $g$ are monomials, and we have a
contradiction since $\tG = \nu (\tT )$ is assumed to be a group.
\end{proof}


\subsubsection{Factoring tangibly-full polynomials in one
indeterminate}\label{factanfull}

Having obtained decisive (albeit easy) results for tangible
polynomials in one indeterminate, we turn towards the general
case, focusing first on tangibly-full polynomials (such as $\la^2
+ 2^\nu \la + 4)$. \pSkip

 We
recall that for any full essential polynomial $f$ of degree $t$,
we get a sequence of ghost elements $m_1^\nu \ge \dots \ge
m_t^\nu$,
  defined uniquely by the slopes of the series of edges
of $\oC_f$, each determined by the
pair $(i-1,\al_{i-1}^\nu)$ and $(i,\al_{i}^\nu)$ for $ 1 \le i \le t.$  Recall that a
monomial
 $h = \al_{i}\la ^i$ of $f$ is essential iff $(i,\al_{i}^\nu)$ is a vertex of $\oC_f,$ which is true
 iff $m_i^\nu \ne m_{i+1}^\nu.$ We have three possibilities for  a monomial
 $h = \al_{i}\la ^i$ of $f$: \pSkip
\begin{enumerate} \ealph
    \item  $h$ is tangible essential; \pSkip
    \item   $h$ is ghost essential; or \pSkip
    \item  $h$ is quasi-essential (at a lattice point which is
    not a vertex of $\oC_f$),   which is the case iff $m_i^\nu =
    m_{i+1}^\nu$.
\end{enumerate}
The following  observation explains how to factor a tangibly-full
polynomial.

\begin{lemma}\label{lem:fact2} Assume that $F$ is a supertropical
semifield, with $f = \sum _j \a_j \la ^j \in F \pl \la \pr.$
 If   $\a_i \lm^i$ is a tangible essential monomial of $f$, then  $$f
=
 (\al_t\lm^{t-i} + \al_{t-1}\lm^{t-i-1} + \cdots +
\al_{i+1}\lm + \a_i)(\lm^i + \frac{\al_{i-1}}{\a_i}\lm^{i-1} +
\cdots + \frac{\al_0}{\a_i}).$$
 \end{lemma}
\begin{proof} Denote the right side  by $p(\la)$, and let $h_j$ be
the monomial of degree $j$ of $p$. We need to show that  $h_j =
\al_j\lm^j$ for all $j$. Note that each $h_j$ is a tropical sum of
monomials, one of which is $\al_j\lm^j$, so we need to check this
is always the one (and only) monomial having the largest
$\nu$-value. We do this case by case.

For $j = i,$ this is clear unless  $\a_i  \la^i \le_\nu
  \al_{i+k}\lm^{k} \frac{\al_{i-k}}{\a_i}\lm^{i-k} $, for
some $0 < k \le i$. But then $\a_i \le_\nu
 \al_{i+k}\frac{\al_{i-k}}{\a_i} $, and thus $  \a_i^2 \le_\nu
 \al_{i+k}\al_{i-k}$. But this contradicts the fact that
$\a_i\lm^i$ is essential for~$f$.

For $j >i$, we are done unless $\a _j \la^j  \le_\nu
  \al_{j+k}\lm^{j-i+k} \frac{\al_{i-k}}{\a_i}\lm^{i-k} $, for
some $0 < k \le i$. Then $\al_{j}\le_\nu
 \al_{j+k}\frac{\al_{i-k}}{\a_i}, $ implying  $
\frac{\a_i}{\al_{i-k}} \le_\nu \frac{\al_{j+k}}{\al_j} $. Since
  $\oC_f$ is convex, we must have
equality, and $\a_{i-k}, \a _i, \a_j,$ and $\a _{j+k}$ all lie on
the same edge; again this contradicts the essentiality of
~$\a_i\lm^i$.

For $j < i$, we are done unless  $ \a _j \la_j \le_ \nu
\al_{i+k}\lm^{k}\frac{\al_{j-k}}{\a_i}\lm^{j-k}$ for some $0 < k
\le i$, yielding
 $ \frac{\al_j}{\al_{j-k}} \le_\nu   \frac{\al_{i+k}}{\a_i},
$   a contradiction  by the same consideration as in the previous
paragraph.
\end{proof}

\begin{cor}\label{tangfact} Suppose $f = \sum _{j=0}^t \al _j \la ^j$ is full,
 with $ \al
_i  $   tangible for some $0 < i <t$. Then $f = g_1g_2,$ where
$g_1 = \sum _{j=0}^{t-i} \al_{i+j}\la^j$ and $g_2 = \sum
_{j=0}^{i} \frac{\al_{j}}{\al_i}\la^j$ are full.
\end{cor} We  call this a factorization \textbf{along a
tangible vertex}; note in this case that $f$ is tangibly-full iff
$g_1$ and $g_2$ are both tangibly-full. Using Corollary
\ref{tangfact} repeatedly, we have:

\begin{prop}\label{strongfull}  Suppose $F = \clF.$ Any tangibly-full polynomial $f\in F \pl \la \pr$ is
the  product of some power of $\la$ with a product of tangible
binomials.\end{prop}

 \begin{Note} One also could prove Proposition \ref{strongfull} geometrically, which provides the dividend that the factorization
 is unique up to $(\nu,e)$-equivalence: We subdivide the graph of $f$ to its lines of
different slopes. In other words, if $f = \sum _{i=0}^n \a _i \la
^i$ where the slope changes at $\la ^t$, then one sees easily that
$f = gh$ where $g = \sum  _{i=0}^{n-t} {\a _{i+t}} \la ^i$ and $h
= \sum _{j=0}^{t} \frac{\a _{j}}{\a_t} \la ^j.$ Different products
of  binomials clearly produce different graphs, and thus the
factorization is unique. \end{Note}

\begin{cor} When $F = \clF,$ any irreducible tangibly-full
polynomial in one indeterminate must be a binomial.
 \end{cor}

\subsubsection{Factoring arbitrary full polynomials in one
indeterminate}\label{facfull}

When considering  full polynomials that are not necessarily
tangibly-full, we  must face the fact that not every nonlinear
polynomial $f$ is $e$-reducible; for example, one can easily check
that $f = \lm^2 + \uuu{2}\lm + 3$ is $e$-irreducible. We need an
intermediate notion.

\begin{defn} A polynomial $f = \sum _{i=0}^t \al _i \la
^i$ is \textbf{semitangibly-full} if $f$ is full with $\al_t$ and
$\al _0$ tangible, but $ \al _i$ are ghost for all
$0<i<t$.\end{defn}

 Dividing out by $\a _t,$ we may assume that our
semitangibly-full polynomial is monic.
 We have the following
observation:
\begin{lem}\label{nontangfact} If $f = \la ^ t + \al _0 + \sum _{i=1}^{t-1} \al _i^\nu \la ^i$ is monic
semitangibly-full for $t>2$ (where $\a_i$ are taken tangible),
 then
 taking  $$\delta  = \frac {\a_0}{\a _1}, \qquad \beta _i =  \frac {\a_i}{\a
_{t-1}},$$  (both tangible), we have
\begin{equation}\label{ugh} f = ( \la ^2 + \al _{t-1}^\nu\la +
\delta \al _{t-1}) g,\end{equation} where $g = \la^{t-2} +  \bt_1
+\sum _{i=2}^{t-2} \beta_{i} ^\nu \la^{i-1} $. Thus, we can
 extract a quadratic factor from any monic semitangibly-full
polynomial of degree $\ge 2$.
\end{lem}
\begin{proof}
The verification is along the same lines as Lemma~\ref{lem:fact2}.
Namely, the constant terms match, and in the middle, the term
$(\al _{t-1}\la )(\bt _{i}\la^{i-1})= \al_i$ strictly dominates
$\la^2 (\bt _{i-1}\la^{i-2})$ and $(\delta \al _{t-1})(\bt
_{i+1}\la^{i})$, because the slopes of the graph decrease.
(Explicitly, we see that $\al _{i}$ strictly dominates $\bt _{i-1}
= \frac {\al_{i-1}}{\al_{t-1}}$ since $\al_{t-1}\nug   \frac
{\al_{i-1}}{\al _{i}} ,$ and $\al _{i}$ strictly dominates $\delta
\al _{t-1}\bt _{i+1} = \frac {\a_0}{\al_1} \a_{i+1}$ since $ \frac
{\al_{i}}{\al _{i+1}}  \nug \frac {\a_0}{\al_1} .)$\end{proof}

 A qualitative way of
obtaining Equation~\Ref{ugh} is by taking the $\nu$-equivalent
polynomial $\tilde f$ obtained by making each coefficient
tangible, taking the product $\tilde h$ of two linear factors of
$\tilde f$ (we took the first and the last in descending order of
$\nu$-values), writing $\tilde f = \tilde h \tilde g,$ and then
making the inner coefficients ghosts.

It remains to factor a polynomial to semitangibly-full
polynomials, which we do by means of the following observation.
\begin{rem}\label{nontangfact1} Suppose $f = \sum _{i=0}^t \al _i
\la ^i$, where $\a_t= \rone^\nu$.  Then
$$f = (\la^\nu  +{\al _{t-1}})\sum _{i=0}^{t-1} \frac {\al _i}{\al
_{t-1}} \la ^i.$$  \end{rem}

 We call a linear
polynomial $\la^\nu +a$ a \textbf{linear left ghost}. Thus,
whenever the leading terms are ghost we can use
 Remark \ref{nontangfact1} to factor out linear left ghosts until we reach a tangible leading term. But if we do this
twice, we observe for tangible $a,b$ with $a \ge_\nu b$ that
$$(\la^\nu +a)(\la^\nu + b) = \rone^\nu \la ^2 + a^\nu \la + a  b = (\la+a)(\la^\nu +
b).$$ Thus, we always can adjust our factorization to have at most
one linear left ghost factor $\la ^\nu +b$ for $b$ tangible, and
this is the  $b$ having the minimal $\nu$-value for those factors
$\la^\nu + b$ which can appear as linear left ghosts.

This reduces our considerations to the case where $f$ is monic but
with the constant term  ghost. Now we define a \textbf{linear
right ghost} to have the form
 $\la + a^\nu$.
When the constant term is ghost we can factor out some linear
right ghost, and arrange for $f = (\la + a^\nu)h,$ where $h$ can
be factored along tangible vertices to get semitangibly-full
factors, and we continue as above.

Putting together
Corollary~\ref{tangfact} with Lemma~\ref{nontangfact}, we see that
any irreducible full polynomial $f$ must have no tangible interior
vertices, and at most one interior lattice point (whose
corresponding vertex must be  nontangible); thus, $f$ must either
be linear or quadratic, of the form
\begin{equation}\label{eq:quadCom} \al_2 \la^2 + \al_1^\nu \la +
\al_0,
\end{equation} where $\al_1^\nu \la $ is essential.
 Iterating, we have the following result:
\begin{prop}\label{fullfact} Every full polynomial is the product
of at most one linear factor of the form $(\la +a^\nu)$ (namely
with $a^\nu$ maximal possible), at most one linear factor of the
form $(\la^\nu + b)$ (with $b$ tangible and $b^\nu$ minimal
possible), tangible linear factors, and
 semitangibly-full quadratic polynomials. \end{prop}

\begin{rem} \label{tangfactyy} Suppose both the leading and constant
coefficients of $f$ are ghosts, so that we have extracted the
right ghost $(\la +a^\nu)$ and  left ghost $(\la^\nu + b)$. When $a \ge_\nu b,$ we also have
$$(\la + a^\nu )(\la ^\nu + b) = \rone^\nu \la ^2 + a ^\nu\la + (a
b)^\nu.$$\end{rem}

\subsubsection{Uniqueness(?) of factorizations of polynomials in one
indeterminate}

Assume throughout this subsection that $F = \clF$. Having shown
that any full polynomial in one indeterminate has a factorization
into irreducibles of degree $\le 2$, we turn in earnest to the
companion question, of uniqueness of factorization of a (not
necessarily full)
 polynomial into irreducibles. The answer turns out to be
quite interesting, involving subtleties that do not exist in the
classical theory of polynomials. Although   unique factorization
fails in $F[\la]$, there is a version of unique factorization
``minimal in ghosts,'' which is seen to have a natural connection
to the set of roots of the polynomial.

We immediately encounter new difficulties.
\begin{example}\label{exmp:factorization}$ $
\begin{enumerate} \eroman
    \item
 The  factorization into $e$-irreducibles need not necessarily be
unique, even up to $\eqR$; for example $\lm^2 + \uuu{2} \eqR (\lm
+ \uuu{1})^2$ and at the same time $\lm^2 + \uuu{2} \eqR (\lm
+1)(\lm + \uuu{1})$, whereas $\lm + \uuu{1} \neqR \lm + 1$. \pSkip
\item  Another violation of unique factorization: for $\ra \nuge
\rb,$ we have $(\la + \ra^\nu)(\la + \rb) = \la ^2 + \ra^\nu \la +
(\ra \rb)^\nu = (\la + \ra^\nu )(\la + \rb^\nu).$
\item  The previous examples still have unique
$(\nu,e$)-factorization. A more serious violation of unique
factorization:$$
\begin{array}{lll}
    \la ^4 + 4^\nu\la^3 + 6^\nu\la ^2 +5^\nu\la +3  &= & (\la^2 + 4^\nu \la + 2)(\la ^2 + 2^\nu \la +1) \\
                                                     & =  & (\la^2 + 4^\nu \la + 2)(\la  + 2)(  \la +(-1)) \\
                                                     & = &  (\la ^2 +4 ^\nu \la +3)(\la ^2 + 2^\nu \la + 0), \\
                                                         \end{array}
$$
 all of which are factorizations into $e$-irreducibles. \end{enumerate}
\end{example}

The last example is an illustration that the factorization
procedure  of Lemma~\ref{nontangfact} is not unique; we could
factor out any two tangible roots of $\tilde f$ to produce the
first factor, just so long as their $\nu$-values are not both
maximal or both minimal (in which case this trick does not work).
Since we may permute the factors, we may always assume that the
tangible root of highest $\nu$-value belongs to the first factor.

\begin{example}\label{manyfacts} This method explains the different factorizations in
the   polynomial
$$   f = \la ^4 + 4^\nu \la^3 + 6^\nu \la^2
+5^\nu \la +3 $$
of Example~\ref{exmp:factorization} (iii).
  Clearly $f$ is semitangibly-full and
 has the four corner
  roots $-2,-1,2,$ and $4$, so we can take the first quadratic
  factor to be
$\la^2 + 4^\nu \la + 2$ or $\la^2 + 4^\nu \la + 3$. In the first
case, the second quadratic factor is $\la ^2 + 2^\nu \la +1$, but
we could use $\la ^2 + 2  \la +1$ instead, which factors to $(\la
+ 2)(\la + (-1)).$ (This will be explained in Proposition~
\ref{mingho}.)

Had we tried $\la^2 + 4^\nu \la + 6$ for the first factor, we
would need $\la^2 +(-1)^\nu\la + (-3)$ for the second factor, but then
the product is $$\la ^4 + 4^\nu\la^3 + 6\la ^2 +5^\nu\la +3, $$
which is not quite $f$ (since it has a tangible inner
coefficient).
\end{example}

 Nevertheless, there is a version of
unique factorization in one indeterminate.  In
conjunction with Remark~\ref{nontangfact1} and Proposition
\ref{fullfact}, we have proved the following result concerning
unique factorization:

\begin{thm}\label{fullfacty} Any full polynomial in one indeterminate is the unique product of a full tangible
polynomial (which can be factored uniquely into tangible linear
factors), a linear left ghost, a linear right ghost, and
semitangibly-full
 polynomials of maximal possible degree.\end{thm}
\begin{proof} Just factor at each tangible vertex, and multiply
together the full tangible factors.\end{proof}

This brings us back to semitangibly-full polynomials. By Remark
\ref{nontangfact1}, for $F = \bar F,$ any semitangibly-full
polynomial can be factored  into tangible linear and
semitangibly-full quadratic factors. Despite Example
\ref{exmp:factorization}, we also get uniqueness of a sort here,
when we count the number of non-tangible quadratic components, cf.
Equation \Ref{eq:quadCom}. We say a factorization is
\textbf{minimal in ghosts} if it has the minimal number of
irreducible quadratic components having essential ghost terms;
this type of factorization turns out to be unique.

\begin{example} In Example \ref{manyfacts},   the latter is the factorization of
 $f$ which is minimal in ghosts, having only one ghost component.
\end{example}

\begin{lemma}\label{lem:fact3}
Suppose that  $a_i \la ^i$   and $a_{i+1}\la^{i+1}$ are  essential
monomials of
 $f$, such that $\delta a_{i+1} = a_i$ for $\dl$ tangible. (This
 means that $ a_i, a_{i+1}$ are both tangible or both ghost.)
Then  $$f =
 \left(\al_t\lm^{t-1} +  \cdots +
\al_{i+1}\lm^i + \frac{\a_{i-1}}{\dl} \lm^{i-1} +
\frac{\a_{i-2}}{\dl} \lm^{i-2} + \cdots +
\frac{\a_{0}}{\dl}\right)\left(\lm + \dl\right).$$
 \end{lemma}

 \begin{proof} Denote the product by $g$ and let $h_j$ be its monomial of degree
 $j$. To see that $h_j = (\al_j\lm^{j-1})\lm$ for $j>i$, note that if $h_j
 \nucong
  (\al_{j+1} \lm^j)\dl  \ge_\nu  (\al_j\lm^{j-1})\lm$, then
 $ \frac{a_{i}}{a_{i+1 }} \nucong \dl \ge _\nu \frac{a_j}{a_{j+1}} $, which
 contradicts the fact that the sequence of slopes determined by
 the coefficients is descending.

 For $j = i$, we have
 $(a_{i+1} \lm^{i}) \dl = (a_{i+1} \lm^{i})\frac{a_{i}}{a_{i+1 }} =
 a_{i}\lm^i$, which strictly dominates $(\frac{a_{i-1}}{\dl} \lm^{i-1})\lm $ since the slopes of the graph decrease.
 Hence, $h_i = \al_i \lm^i.$

When $j <i $, $h_j = (\frac{\a_{j}}{\dl}
 \lm^{j})\dl$ since otherwise $h_j^\nu \ge_\nu\left(\frac{\a_{j-1}}{\dl}
 \lm^{j-1}\right)\lm$ by the same argument as for $j > i$, which leads
 to the analogous
 contradiction.
 \end{proof}

Putting all these results together yields:

\begin{thm}\label{fullfact2} When $F = \bar F$, any full polynomial in one indeterminate has a factorization into tangible linear
factors, quadratic semitangibly-full factors, at most one linear
left ghost and at most one linear right ghost, and the
factorization which is minimal in ghosts is unique.\end{thm}
\begin{proof} Just factor at each tangible vertex, then factor
inductively at pairs of adjacent ghost vertices, and multiply
together the full factors.\end{proof}

Here is another way of viewing Theorem \ref{fullfact2}.
\begin{cor}\label{rmk:fact}
  Any full polynomial $f$
 can be written as the product $f = f_t f_m$ where $f_t$
is tangible and $f_m$ is the product in Theorem \ref{fullfact2} of
(perhaps) a linear left ghost, a linear right ghost, and
semitangibly-full
 polynomials; $f_m$ has alternating tangible
and ghost coefficients, seen by applying Lemmas \ref{lem:fact2}
and \ref{lem:fact3} inductively for pairs of adjacent  tangible or
ghost monomials that are essential. We call this procedure
\textbf{extracting a minimal ghost factor}; note the minimality is
in essential ghosts. Accordingly, $f_t$ can be factored into
linear components, and the factorization of $f_m$ has at most two
linear components while all the others are quadratic.
\end{cor}

%

We can understand Theorem \ref{fullfact2} better, by considering
the tangible roots of a polynomial. In view of Remark~\ref{easyf},
these roots are determined by the tangible roots of its
$e$-irreducible factors. The case of one indeterminate is given in
Example \ref{linearirr}.

\begin{remark} Working backwards in Type IV of Example \ref{linearirr}, given a closed interval (or point)
$W$ in~$\tT,$ one can write the $e$-irreducible polynomial of
degree $\le 2$ whose   set of tangible roots is precisely $W$.

In general, given a closed subset $W$  of $\tT$, we write $W$ as a
finite union $W_1 \cup \cdots \cup W_t$ of disjoint closed
intervals (or points), and take an $e$-irreducible polynomial
$f_k$ of degree $\le 2$ whose set of tangible roots is precisely
$W_k$, for $1 \le k \le t.$ Then taking $  f = f_1 \cdots f_t,$ we
see that $\tZ_{\tng} ( f) = W$.
\end{remark}

Let us apply this process to an arbitrary semitangibly-full
 polynomial $f$.
 \begin{prop}\label{mingho} Suppose  $f$ is a semitangibly-full
 polynomial of degree $t$, and $\a _1, \dots, \a_t$ are
corner roots of $f$, arranged in ascending $\nu$-value. Then $$ f=
(\la^2 + \a_t ^\nu \la + \a_t \a_1)\prod_{k = 2}^{t-1} (\la +
\a_k).$$
 \end{prop}
 \begin{proof} Consider the tangibly-full polynomial $\tilde f$
 whose coefficients have the same $\nu$-value as those of $f$.
 Then $\a _1, \dots, \a_t$ are the \individual\ roots of $f$, so $$\tilde f =
 \prod _{k = 1}^{t} (\la +
\a_k) = (\la^2 + \a_t  \la + \a_t \a_1)\prod_{k = 2}^{t-1} (\la +
\a_k).$$ It remains to note that all the interior coefficients of
$(\la^2 + \a_t ^\nu \la + \a_t \a_1)\prod_{k = 2}^{t-1} (\la +
\a_k)$ are ghost, since the coefficient of $\la^j$ is $\a_t^\nu
\a_{t-1}\cdots \a _{t-j}.$ \end{proof}

Obviously this is the factorization with the minimal number of
ghosts (namely, just one). Note that the corner roots $\a_2,
\dots, \a_{t-1}$ are interior points in
$\tZ_{\operatorname{tan}}(f)$, and all appear in tangible linear
factors. When $\deg f
>t$, the statement of the result becomes more complicated since
one needs to deal with multiple roots, but the proof is analogous, to be treated in another paper.

\begin{rem}\label{factorback} Reversing the logic of Proposition~\ref{mingho}, take an arbitrary semitangibly-full polynomial $f =
\la^t + \sum _{i=1}^{t-1} \a _i^\nu \la^i +  \a_0\a_1$, where each
$\a_i \in \tT$. The tangible root set of $f$ is the interval
$[\a_0, \a _{t-1}],$ so we can factor
$$f = (\la^2 + \a_{t-1}^\nu \la + \a_0\a_{t-1})g,$$ where $g = \sum _{i=0}^{t} \a _{i+1}^\nu
\la^{i}$ is a tangible polynomial which can thus be factored into
linear factors.\end{rem}

\begin{rem} We are now in a position to explain geometrically the various factorizations of a full
polynomial $f\in F[\la]$ of degree $n$. Namely, we take the set $S
= \{a_1, \dots, a_n\}$ of tangible corner roots of $f$, and
partition $S$ into $n_1$ pairs $(a_{i_1}, a_{i_2})$, $1 \le i \le
n_1$, and $n_2$ single roots, where $2n_1 +n_2 = n,$ such that
$$\cup (a_{i_(a_{i_1}, a_{i_2})1}, a_{i_2}) = \tZ_{\tan}(f).$$
(The union need not be disjoint.) Each of the closed intervals
$\left[a_{i_1}, a_{i_2}\right]$ is the root set of a polynomial
$$f_i = \la ^2 + a_{i_2}^\nu \la + a_{i_1}  a_{i_2},$$ whereas
each single root $a_j$ is the root set of the linear polynomial
$\la + a_j,$ and the product of all of these polynomials can be
seen to be $f$. Each of these subdivisions corresponds to a factorization of $f$
  into irreducibles.

  There will be only one such partition in which each interval $\left[a_{i_1},
  a_{i_2}\right]$ is a connected component of the tangible
  root set of $f$, and this is the (unique) ``preferred''
  factorization.
\end{rem}

\begin{example}
 Let us apply Remark~\ref{factorback}  to explain Example \ref{manyfacts}.
Since $f$ is semitangibly-full of degree 4, and its corner
  roots are $-2,-1,2,$ and $4$, of which $-1$ and $-2$ are in the interior of $\tZ_{\operatorname{tan}}(f)$,
   we have the factorization $$f = (\la^2 + 4^\nu \la + 2)(\la
  + (-1))(\la +2).$$
  This is the ``preferred'' factorization; the other
  factorizations are obtained by taking the partitions
  ${[-2,2],[-1,4]}$ and ${[-2,4],[-1,2]}$. The ``near miss'' of Example \ref{manyfacts} comes
  from $([-2,-1].[2,4]$ which is not quite a subdivision of $[-2,4]$.
\end{example}

Here is a satisfactory numerical algorithm for factoring a
polynomial $f$ into e-irreducibles: First factor out the linear
left ghost and/or  right ghost if necessary, then factor $f$ into
a product of $m$ semitangibly-full factors, and then apply
Proposition~\ref{mingho} (or Remark \ref{factorback}) to obtain
the factorization minimal in ghosts (one ghost for each of the $m$
semitangibly-full factors).

\subsection{Binomial factorization in several indeterminates}

We turn to factorization in $F[\lm_1,\dots, \lm_n]$ for $n>1.$
Although the thrust of this subsection is an analysis of how
unique factorization fails, first we note a positive cancellation
result as consolation. Recall that we write $F[\Lm]$ for
$F[\lm_1,\dots, \lm_n]$ and $F[\Lambda, \Lambda^{-1}]$ for
$F[\lm_1,\lm^{-1}_1\dots, \lm_n,\lm_n^{-1} ]$.

\begin{rem} If $f\in F[\Lambda, \Lambda^{-1}]$ is tangible and $fg
= fh$, then $g=h.$ (Indeed, $g$ and $h$ take on the same values on
a dense subset of $F^{(n)}$ in the $\nu$-topology,  so are
identically equal.)\end{rem}

Note that this topological argument could fail when $f$ is not
tangible, even for one indeterminate, as evidenced by the
factorizations
$$\la^2 + 1^\nu \la + 1^\nu = (\la +1^\nu)(\la +0^\nu) = (\la
+1^\nu)(\la +0);$$
 $$ \begin{array}{lll}
    \la ^4 + 4^\nu\la^3 + 6^\nu\la ^2 +5^\nu\la +3  &= & (\la^2 + 4^\nu \la + 2)(\la ^2 + 2^\nu \la +1) \\
                                                     & =  & (\la^2 + 4^\nu \la + 2)(\la ^2 + 2  \la +1) \\
                                                         \end{array}$$

For several variables, we confront a most severe violation of
unique factorization.

\begin{rem}\label{rem:vander} Suppose $f_1, f_2, f_3 \in
F \pl \Lambda \pr.$
\pSkip
\begin{enumerate} \eroman
    \item
 $f_1+f_2+f_3$ is a factor of $(f_1+f_2)(f_1+f_3)(f_2+f_3).$
Indeed,
$$\begin{array}{l} (f_1+f_2+f_3)(f_1f_2+f_1f_3 + f_2f_3) = \\ f_1^2 f_2+ f_1^2 f_3 + f_1
f_2^2 + f_1 f_3^2 +  f_2^2 f_3 +  f_2 f_3^2 +\nu(f_1f_2f_3) = \\
(f_1+f_2)(f_1+f_3)(f_2+f_3).\end{array}$$ Note that the polynomial
$\nu(f_1f_2f_3)$ is inessential.  Thus, full-tangible polynomials
need not have unique factorization. \pSkip
%
 Two variants of (i), for later use, which one
checks  by matching the tangible parts: \pSkip

\item  $(f_1+f_2
+f_3^\nu)(f_1f_2+f_1f_3 + f_2 f_3^\nu) =
(f_1+f_2)(f_1+f_3^\nu)(f_2+f_3);$

\pSkip \item
$(f_1+f_2^\nu +f_3^\nu)(f_1f_2+f_1f_3 + f_2f_3) =
(f_1+f_2^\nu)(f_1+f_3^\nu)(f_2+f_3).$

\end{enumerate}
\end{rem}

\begin{example}\label{nonuniquefact}
$$\begin{array}{l} (0+\la_1 + \la_2)(\la_1 + \la _2 + \la_1\la_2) =
\\ \la _1 + \la _2 + \la_1^2 + \la _2 ^2 + \nu (\la_1\la_2) + \la_1^2
\la _2 + \la _1 \la _2^2 = \\ (0+\la_1)(0+\la_2)(\la_1 +\la_2).
\end{array}$$
\end{example}

In fact,  any polynomial $f = \sum_{i=1}^m f_i$ $e$-divides
$\prod_{i \neq j} (f_i+f_j),$ which leads us to the main theorem
of this section.

\begin{theorem}\label{permprime} Suppose $f = \sum_{i=1}^m f_i \in \FunR
$, for $m \ge 2$. Then
\begin{equation}\label{twofacts}
 \prod_{i < j } (f_i + f_j)  = g_1\cdots g_{m-1}\,
\end{equation}
where $g_1 = f= \sum_i f_i,$ $g_2 = \sum_{i < j } f_i f_j,$ $
\cdots$, and $g_{m-1} = \sum_{i }\prod_{j \neq i} f_j$.
\end{theorem}

Our applications of this theorem are for the sub-semiring
$R[\Lambda, \Lambda^{-1}]$ of $\FunF$, in which this result could
be viewed as the utter collapse of unique factorization, since
\emph{every} polynomial~$f$ which is a sum of at least three
distinct monomials is part of a factorization that is not unique.
Specifically, if $f_i$ are the monomials of $f$, then $f$ is a
factor of  the product $\prod_{i \ne j} (f_i+f_j).$ However,
Theorem~\ref{permprime} casts considerable light on the geometry,
and has a positive geometric interpretation:

\begin{rem}\label{geom1} Every tropical variety~$\Var$  can be ``completed'' to a
variety $\tP(\Var)$ comprised of various hyperplanes, which in
turn can be decomposed into a union $\Var_1 \cup \dots \cup
\Var_{m-1}$, where $\Var_i = \tZ_{\tan}(g_i)$ for $1 \le i \le
m-1$. After proving Theorem ~\ref{permprime}, we shall see how
such geometric decompositions provide an assortment of
factorizations.
\end{rem}

 The factorization in
Theorem ~\ref{permprime} involves many inessential terms, so, to
avoid excessive computation in the proof of the theorem, we
consider
  how inessential terms often  arise.

\begin{lemma}\label{lem:essential} If  $h_2 ^2 = h_1  h _3\in \FunR$, then $h_2$ is inessential for $h_1+h_2+h_3$.
\end{lemma}
\begin{proof}  Apply Lemma~\ref{compaid}(i) to each  $\bfa = (a_1,\dots, a_n) \in R^{(n)}$. \end{proof}

For the next lemma, we let $\mathcal I_m \subset \Net^{(m)}$
denote the set of all $m$-tuples $\bfi = (i_1, \dots, i_{m})$ for
which each $0 \le i_u <m$ and $\sum _{u=1}^{m} i_u = \binom m 2$.
Such $m$-tuples include $(0,1,\dots, m-1)$ or any permutation of
the components. For any $ \bfi = (i_1, \dots, i_{m}) \in \mathcal
I_m$ and $0 \leq j \leq m-1$, we define the $j$-index $\iota
_j(\bfi)$ to be the number of $i_u$'s that equal $j$; define
$\iota (\bfi) = (\iota _{m-1}(\bfi), \dots, \iota _{0}(\bfi))$.

 Let $S_m$
  denote the set of permutations of $(0,1, \dots, m-1).$   Thus, $\bfi \in S_m$ iff $\iota  (\bfi)
= (1,1,\dots, 1)$.
 We say $\bfi $ is \textbf{admissible} if for each number $k$, the
 sum of the largest $k$ components of $\bfi$ is at most $(m-1) + \dots +(m-k) = km - \frac {k(k+1)}2.$
 Thus, all $\bfi \in S_m$ are admissible.

\begin{lemma} Lexicographically,  $\iota
(\bfi) \le (1,1,\dots, 1)$ for each admissible $\bfi \in \mathcal
I_m$.
\end{lemma}
\begin{proof}  For any admissible   $\bfi$, the sum of the largest two components  is
at most $2m-3,$ which means that at most one component is $m-1$,
so $\iota  _{m-1}(\bfi)\le  1$. We are
done unless $\iota  _{m-1}(\bfi) = 1,$  and conclude by induction on $m$.
\end{proof}

 Given $f_1, \dots, f_m \in \FunR,$ for each $\bfi\in \mathcal I_m,$ we define the
function $$h_\bfi = f_1^{i_1}\cdots f_{m}^{i_{m}}.$$ For any permutation $\sig \in S_{m},$ we denote
$$h_\sig =f_1^{\sig (0)}\cdots f_{m-1}^{\sig (m-1)}.$$

\begin{lemma}\label{lem:essential2}  $  \sum _{\sig \in S_m} h_\sig= \sum_{\sig \in \mathcal I_m} h_\sig$.
\end{lemma}
\begin{proof} Let $p = \sum _{\sig \in \mathcal I_m}
h_\sig $. We need to show that  $h_\bfi$ is inessential in $p$
 whenever $\iota  (\bfi) <(1,1,\dots, 1).$ The proof is by reverse
induction on the lexicographic order of $\iota (\bfi).$ Since
$\iota (\bfi) <(1,1,\dots, 1),$  some $j$-index $\iota_{j} (\bfi)$  is 0, and we take
the largest such $j$. Then for some $j'<j,$  the $j'$-index
$\iota_{j'} (\bfi)\ge 2$; in other words, $\bfi$ has components $
i_s = i_t = j'$ for suitable $s \ne t$.

Take $\bfi' = (i_1', \dots, i_m')$ to be the $m$-tuple in which  $i'_s = j'+1$ and   $i'_t
= j'-1$ (with all other components the same as for $\bfi$), and likewise
let $\bfi''$ be the $m$-tuple in which $i''_s = j'-1$ and $i''_t  =
j'+1$. By Lemma~\ref{lem:essential}, $h_\bfi $ is inessential in
$h_{\bfi'}+h_{\bfi''}.$

We claim that  $\bfi'$ and $\bfi''$ are admissible and $\le
(1,1,\dots, 1)$. Indeed, this is clear when $j'< j-1,$ since then
$\iota_j (\bfi')= \iota_j (\bfi'') = 0.$ Thus, we may assume that
$j' = j-1.$ Clearly $ \iota(\bfi'')=  \iota(\bfi'),$ since the
roles of $s$ and $t$ are interchanged, so it suffices to prove the
claim for ${\bfi'}$.

First assume that $\iota_{j-1} (\bfi) \ge 3.$ Then $j\ge 2$, and
the
 sum of the largest $k = m-j+2$ components of $\bfi$ is equal to
 $$(m-1) + \dots +(j+1)+0  +3(j-1),$$ which is
 greater than $ km - \frac {k(k+1)}2$ unless $3(j-1) = j + (j-1),$
 which implies $j=2$. But then $\iota(\bfi) =  (1,1,\dots, 0,3,0)$
and so $\iota(\bfi') =  (1,1,\dots, 1,1,1)$, as desired.

Thus, we are done unless $\iota_{j-1} (\bfi) = 2.$ Since $\bfi'$
increases   the component  $i_s$ from $j-1$ to $j$, and decreases
the component  $i_t$ from $j-1$ to $j-2,$ we see that $\iota_j
(\bfi') =1$ and $\iota_{j-1} (\bfi')=0,$ proving $\iota (\bfi')<
(1,1, \dots, 1),$ as desired.

 Clearly, $\iota(\bfi') = \iota(\bfi'')$ is
of higher lexicographic order than $\iota(\bfi)$ (since
$\iota_{j'+1}(\bfi')= \iota_{j'+1}(\bfi)+1 $), so, by reverse
induction, either $\bfi' \in S_m$ or $h_{\bfi'}$ is inessential in
$p$, and likewise for $\bfi'',$  implying that $h_\bfi$ is
inessential in $p$.
\end{proof}

Our main tool in proving Theorem~\ref{permprime} is the
\textbf{tropical Vandermonde  matrix} $V_f$ of  $f = \sum_{i=1}^m
f_i$. Define $V_f$ to the $m \times m$ matrix with entries $v_{i,j}
= f_i^{j-1}$. Since the signed determinant is not available in
tropical algebra (because it involves negative signs), one
substitutes the permanent, which we still notate as
$$\per{{V_f}} = \sum_{\sig \in S_m}  f_1^{\sig(0)} \cdots  f_m^{\sig(m-1)},
$$
 We
can compute the permanent in two ways:

\begin{lemma}\label{lem:Van} If $V_f = (f_i^{j-1})$ is the $m \times m$ Vandermonde matrix
for $f = \sum f_i$, then
\begin{enumerate}
    \item $\per{V_f} =
\prod_{i < j } (f_i + f_j) $ and,\pSkip
    \item $\per{V_f} = (\sum_i f_i) (\sum_{i < j }
f_i f_j) \cdots (\sum_{i }\prod_{j \neq i} f_j).$
\end{enumerate}
\end{lemma}
\begin{proof}  Let $p = \per{V_f}= \sum_{\sig \in S_m} h_\sig$,
the function of Lemma~ \ref{lem:essential2}, which says that $p =
\sum_{\sig \in \mathcal I_m} h_\sig$.

But it is easy to see that each summand of   $q_1 =  \prod_{i < j
} (f_i + f_j)$ has the form $h_\bfi$ where $\bfi$ is admissible,
 and thus $h_\bfi$  is dominated by $p$, by
Lemma~\ref{lem:essential2}. Since each summand of $p$ appears in
$q_1$, we get $p = q_1.$
Likewise, expanding  $q_2 = \left(\sum_i f_i\right) \left(\sum_{i
< j }f_i f_j\right) \cdots \left(\sum_{i }\prod_{j \neq i}
f_j\right) $ clearly each term has the form $h_\bfi$ where $\bfi$
is admissible, and each summand of $p$ appears in $q_2$, implying
$p =  q_2.$
\end{proof}

The proof of Theorem \ref{permprime} now becomes quite
transparent:

\begin{proof}[Proof of Theorem \ref{permprime}] 
By parts (i) and (ii) of  Lemma \ref{lem:Van}.
\end{proof}

 Algebraically, Theorem
\ref{permprime} shows that the  factorization of $\per{V_f} \in
R[\Lambda]$ into irreducible polynomials is not unique.

\begin{example} Suppose $f= \lm_1^i + \lm_2^j + \al $, with $\al \in R$. Then $$V_f = \vvMat{0}{\al}{\al^2}{0}{\lm_1^i }{\lm_1^{2i}
}{0}{\lm_2^j }{\lm_2^{2j} } \quad \text{and}$$
$$\per{V_f} \eqR  (\lm_1^i + \lm_2^j + \al)(\al\lm_1^i + \al \lm_2^j + \lm_1^i  \lm_2^j) \eqR
(\lm_1^i + \lm_2^j )(\lm_1^i + \al)(\lm_2^j + \al) \  .
$$
This yields two different tropical factorizations  of $\per{V_f}$
into irreducible polynomials. (The right factorization  is a
binomial factorization.)
\end{example}

Thus, in this version of supertropical algebra, perhaps ``unique
factorization'' is the wrong emphasis, but rather we should
emphasize factorization of
 $\per{V_f} \subset R[\lm_1, \dots, \lm_n]$ into binomials.

\begin{rem}\label{dual2} Lemma~\ref{lem:essential} is clearly self-dual in the sense of Remark \ref{add5}; hence Theorem~\ref{permprime} also holds over the dual
supertropical semifield $\dual{F}$. Explicitly, suppose  $f = \sum
f_i$, written as a sum of monomials. Taking the isomorphism
$\Phi_{\Fun}$ of Remark~\ref{add5} and putting $\bar f = g_1\cdots
g_{m-1}$, we have $\Phi_\Fun (g_i) ={(\bar f)^{-1}}g_{m-i}$ for $1
\le i \le m-1.$

 Thus, $\Phi_\Fun$ also induces an action $\hat \Phi$ on tangible root sets, given by
$$\hat \Phi (\Var_i) =(\tZ_{\tng}( \Phi_\Fun(g_i)) = (\tZ_{\tng}(
\Phi_\Fun(g_{m-i}))= \Var_{m-i}.$$ Likewise, $$\Phi_\Fun (f_i +
f_j) = {f_i}^{-1} + {f_j}^{-1} = {(f_if_j)^{-1}}(f_i+f_j),$$ so
$\hat \Phi$ preserves binomials in the sense that $$\hat \Phi
(\tZ_{\tng}(f_i + f_j)) = \hat \Phi
(\tZ_{\tng}({(f_if_j)^{-1}}(f_i+f_j))) =\tZ_{\tng}(f_i + f_j).$$
Thus, in Remark~\ref{geom1}, $\hat \Phi$ induces a partition into
dual pairs of root sets.
\end{rem}

In this way, the algebraic structure again is reflected in the
geometry.

\begin{rem}\label{power2} We also can  go in the other direction, and illustrate the Nullstellensatz. If
$f= \sum _{i=1}^m  f_i$ then, arguing as before (since $f_if_j$ is
dominated by $f_i^2 + f_j^2$), $$f^2 \eqR \sum f_i ^2 \eqR f_m(f_m
+f_1) + \sum _{i=1}^{m-1}f_i(f_i +f_{i+1}),$$ which is in the
ideal $\langle f_m+f_1, f_i + f_{i+1}: 1 \le i \le m-1
\rangle$.
\end{rem}
\begin{figure}[h]
\setlength{\unitlength}{0.8cm}
\begin{picture}(10,7)(0,0)

\thicklines \put(6,3){\line(-2,1){2.5}}
\put(6,3){\line(0,-1){2.5}}\put(6,3){\line(1,-1){2}}

\dashline{0.2}(6,3)(8.5,1.8)\dashline{0.2}(6,3)(4,5)\dashline{0.2}(6,3)(6,5.5)
%

\put(2.3,4.3){\scriptsize{$\la _1\la
_2^2+0$}}\put(5.5,0.2){\scriptsize{$\la _1+0$}}
\put(7.7,0.6){\scriptsize{$ \la _1\la _2 + 0$}}
\put(5.9,5.6){$h_2$} \put(8.5,1.2){$h_3$}\put(3.7,5.2){$h_1$}

\put(6.1,2.9){\scriptsize{$(0,0)$}}
\end{picture}
\caption{The tangible roots in Example \ref{exm:1}. \label{fig:1}}
\end{figure}
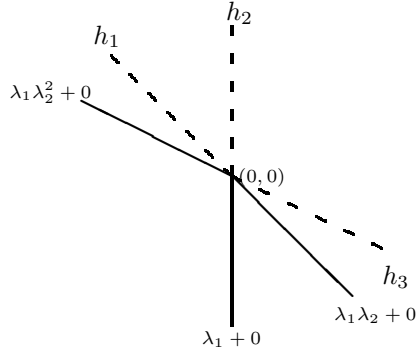

\begin{example}\label{exm:1}
(Illustrating Theorem \ref{permprime}.) Let $f = \la _1^2\la _2 +
\la _1 + 0$ (see Fig. \ref{fig:1}), a polynomial over~$D(\Real)$.
Then, notation as in Theorem~\ref{permprime},  $ g_1= f$ and $g_2
= \la_1(\la _1^2\la _2 + \la _1\la _2 + 0).$   Defining the
binomials
   $q_1 = \la _1^2\la _2 + \la_1,\  q_2 = \la _1^2\la _2 + 0,$ and $ q_3 = \la _1 + 0$,   we have the
   equality
$$ g_1g_2 = q_1 q_2 q_3 = \la _1 + \la _1^2 + \la _1^2\la _2 + 0^\nu \la _1^3 \la _2 + \la _1^4 \la _2 + \la _1^4 \la _2^2 + \la _1^5 \la _2^2.$$
$\Var_2= \tZ_{\tng}(g_2)$ can be viewed as the complement of
$\Var= \tZ_{\tng}(f)$ along $(0,0)$.
\end{example}

In the next example, we can ``improve'' the factorization of
Theorem \ref{permprime}.

\begin{example}\label{exm:2}  Let $f= \la _1^2 + \la
_2^2 + \a \la _1\la _2 + 0$ (see Fig. \ref{fig:2}) be a polynomial
over $D(\Real)$, where $\a
> 0$. Note here that $\tZ_{\operatorname{tan}} (\la_1^2 +
\la_2^2)$ does not affect $\tZ_{\operatorname{tan}} (f)$,
since whenever $(\la_1^2)^\nu = (\la_2^2)^\nu,$ these are both
less than $(\a \la_1 \la_2)^\nu$.

Let $f_i = \la _i^2 + \a \la _1\la _2 + 0$, for $i=1,2.$ Also,
define the binomials $q_1= \a \la _1\la _2 + 0$, $q_2= \la _1^2 +
0$, $q_3 = \la _2^2 +0$, $q_4 = \la _1 + \a \la _2$, and $q_5 = \a
\la _1 + \la _2$. Algebraically, Theorem~\ref{permprime} applied
to $f_1$ and $f_2$ in turn yields
$$f_1g_1 = q_1q_3q_5 ; \qquad f_2g_2 = q_1q_2q_4,$$
where $g_1= \la _2 + \a \la _1\la _2^2 + \a \la _1$ and $g_2= \la
_1 + \a \la _1\la _2^2 + \a \la _2$. Furthermore,
$$q_1f = \a \la_1^3\la_2 + \a \la_1 \la_2^3 + \a
^2 \la_1^2 \la_2^2 + \a^\nu \la _1 \la _2 + \la_1^2 + \la_2^2 + 0
= f_1f_2,$$ since $\la_1^2\la_2^2$ is strictly dominated by
$\a\la_1^2\la_2^2.$ Consequently,
$$q_1fg_1g_2 = f_1g_1f_2g_2 = q_1^2 q_2q_3 q_4 q_5 ,$$
implying $fg_1g_2 =  q_1 q_2q_3 q_4 q_5,$ which is actually an
improvement of Theorem \ref{permprime}.

Taking tangible root sets, we have $$\tZ_{\operatorname{tan}}
(q_1)\cup \tZ_{\operatorname{tan}} (f) = \tZ_{\operatorname{tan}}
(f_1f_2)= \tZ_{\operatorname{tan}} (f_1 )\cup
\tZ_{\operatorname{tan}} (f_2).$$

Geometrically, $\tZ_{\operatorname{tan}} (f)$ is contained in the
five lines which are respectively the tangible root sets of $q_1$,
$q_2$, $q_3$, $q_4$, and $q_5$. The tangible root sets of  $g_1$
and $g_2$ are the complements of $f$ along the respective vertices
$(-\a,0)$ and $(0,-\a)$.
\end{example}

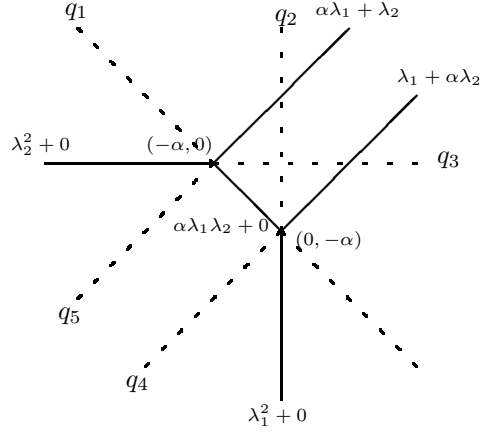
\begin{figure}[!h]
\setlength{\unitlength}{0.9cm}
\begin{picture}(10,7)(0,0)

\thicklines
\put(5,4){\line(-1,0){2.5}}\put(5,4){\line(1,-1){1}}\put(5,4){\line(1,1){2}}
\put(6,3){\line(0,-1){2.5}}\put(6,3){\line(1,1){2}}

\dashline{0.1}(5,4)(3,2)\dashline{0.1}(5,4)(8,4)\dashline{0.1}(5,4)(3,6)
\dashline{0.1}(6,3)(8,1)\dashline{0.1}(6,3)(4,1)\dashline{0.1}(6,3)(6,6)
%

\put(2,4.2){\scriptsize{$\la
_2^2+0$}}\put(5.5,0.2){\scriptsize{$\la _1^2+0$}}
\put(4.4,3){\scriptsize{$\a \la _1\la _2
+0$}}\put(6.5,6.2){\scriptsize{$\a \la _1 +\la _2$}}
\put(7.7,5.2){\scriptsize{$ \la _1 + \a \la _2$}}
\put(2.8,6.2){$q_1$}\put(5.9,6.1){$q_2$}
\put(8.3,4){$q_3$}\put(2.7,1.7){$q_5$}\put(3.7,0.7){$q_4$}

\put(4,4.2){\scriptsize{$(-\a,0)$}}\put(6.2,2.8){\scriptsize{$(0,-\a)$}}
\end{picture}
\caption{Illustration to Example \ref{exm:2} of the tangible
roots. \label{fig:2} }
\end{figure}

The explanation of Example~\ref{exm:2} is that one of the
binomials of $f$ (namely $\la_1^2 + \la_2^2 \eqR (\la_1 +
\la_2)^2$) is ``fictitious,'' since its tangible set does not
exist in the graph (as we showed above). Thus we can separate $f$
into two polynomials whose factorizations do not involve the
fictitious binomial. Continuing this process inductively, one can
find a factorization that displays $f$ as a divisor of a product
of $m$ binomials, where $m$ is the minimal number of hyperplanes
whose union contains the graph of $f$. The precise description of
an algorithm for this process seems to involve an investigation of
the Newton polytope, which we do not pursue here.

\section{Prime ideals of polynomial semirings}

Since the Nullstellensatz translates supertropical geometry to
radical ideals, and every radical ideal is the intersection of
prime ideals, we would like to classify the prime ideals of the
 polynomial semiring $F[\Lambda]$ over a
supertropical semifield $F$.
The factorization in Theorem \ref{permprime} clearly affects prime
ideals.

\begin{example}\label{nonuniquefact2}
 The  ghost-closed ideal  $A = \langle \la _1 + \la _2 + 0\rangle $ of $F [ \la_1, \la_2 ]$
is not prime! Indeed, if $A$ were prime, Example
\ref{nonuniquefact} would imply that $A$ contains one of
$0+\la_1$, $0+\la_2$, and $\la_1 +\la_2$, which is absurd.

Likewise, reading Example \ref{nonuniquefact} from the other
direction shows that the ghost-closed ideal  $A = \langle \la _1 +
\la _2\rangle$ of $F[ \la_1, \la_2 ]$ is not prime.
\end{example}

Given a polynomial $f = \sum _\bfi f_\bfi$ written as a sum of
monomials, we define the set of   \textbf{binomials of} $f$ to be
$\{ f_\bfi + f_\bfj :  \bfi, \bfj \in \supp (f), \bfi \neq \bfj
\}.$ The  role of binomials is found in the following key
observation.

\begin{rem} It follows at once from Theorem \ref{permprime}  that if $P$ is a prime
ideal of $F[\Lambda]$ and $f\in P$, then some binomial of $f$
belongs to $P$. \end{rem}

The key to binomials is found in the following observation, which
is a converse to Example~\ref{exmp:div}; we already treated the
case $n=1$ (twice) in Proposition~\ref{strongfull}.

\begin{prop}\label{factorbin} If $F = \clF$ and $f\in F \pl \Lambda \pr$ is a
 \eBin, then $f$ can be factored as a product of a
monomial times a power of an irreducible binomial.\end{prop}

\begin{proof} Let us write $\ef = \a \Lm^\bfi + \beta \Lm^\bfj$.
Factoring out $\beta,$ we may assume that $\beta = \fone.$ It is
convenient to work in $F \pl \Lambda, \Lambda^{-1} \pr$, since
then we may divide by $\Lm^\bfj$ and assume that $\ef$ has the
form $\a \Lm^\bfi +\fone$. We are done unless the full closure of
$\ef$ has some monomial on the line connecting $\bfi$ to
$(0,\dots, 0)$. In other words, $f$ has some monomial $\gamma
\Lm^{\bfk},$  where $\bfi = m\bfk$ for suitable $m$. But then
$\frac {\a}m \La^{\bfk}$ is a monomial of $f$, which is   the
$m$-th power of $h = \frac {\a}m\La^{\bfk}+ \fone$. We are done if
$h$ is $e$-irreducible, and continue by induction if $h$ is
$e$-reducible. (One has to check that the factorization in $F \pl
\Lambda, \Lambda^{-1} \pr$ matches a factorization in $F \pl \la
_1, \dots, \la _n \pr ,$ by clearing denominators.)\end{proof}

\subsection{Ghost-closed prime ideals of polynomials in one
indeterminate  }

The classification of  ghost-closed prime ideals is difficult even
for the case $n=1,$ because there are many more of them than in
classical ring theory. In this paper we content ourselves with the
result that the tangible part of any f.g.~prime ideal of $F[\la]$ is generated by at most
two polynomials. We start with the list of $e$-irreducible
polynomials given in Example~\ref{linearirr}.

 \begin{example}\label{linear2} $ $ Suppose $F$ is a supertropical semifield, and $\a\in F$ and $ \beta \in \tT $
 with $\a  < _\nu  \beta  .$  \pSkip
\begin{enumerate} \eroman
    \item
 The ghost-closed ideal generated by $\la + \a$  contains $\la +
\beta^\nu = (\la + \a) + \beta^\nu $.\pSkip
\item
 Any ideal containing $f_1 =  \la + \a$ and $f_2  = \la +
 \beta$ also contains $\lm + \gamma$ for all $\gamma \in \tT $ with $\a^\nu
 < \gamma ^\nu < \beta^\nu,$ since
 $$\la +\gamma = (\la  + \a) + \frac {\gamma}{\beta}(\la +
 \beta).$$

\item
 Any ideal containing $f_1 =  \la + \a$ and $f_2  = \la^\nu +
 \beta$ also contains $\lm + \gamma$ for all $\gamma \in \tT $ with $\a^\nu
 < \gamma ^\nu < \beta^\nu,$ by the same computation as in (ii).
\pSkip
 \item  If $\a_1^\nu \le \a_2^\nu \le \beta_1^\nu \le \beta_2^\nu,$
 then the polynomial $\la ^2 + \beta_2 ^\nu \la + \a  _1\beta  _2
 $ is contained in
 the  ghost-closed radical ideal generated by  $\la ^2 + \beta_1 ^\nu \la + \a_1\beta_1 $ and
 $\la ^2 + \beta_2 ^\nu \la + \a_2\beta_2  $, as seen by the Nullstellensatz (or by direct computation:
$$(\la ^2 + \beta_2 ^\nu \la + \a_1\beta_2)^2 \lmodgla (\la ^2 + \beta_2 ^\nu \la + \a  _2\beta
_2)(\la^2+ \frac {\a_1\beta_2} {\a_2} \la + \frac
{\a_1^2\beta_2}{\a_2}).)
$$
\end{enumerate}
 \end{example}

Recall the ``types'' of irreducible polynomials given in Example~\ref{linearirr}.

\begin{rem}\label{gen1} (i) It follows from Remark~\ref{linear2}(ii) that any finite set
of type I polynomials is generated by at most two type I
polynomials.

(ii) By Remark~\ref{prime2}, any  ghost-closed prime ideal  $P$
contains a polynomial of type II or type III.

(iii) Applying (ii) for each $\a$ shows that any
  ghost-closed prime ideal either contains all possible polynomials
of type II or type III (and thus is not f.g.), or contains  a
polynomial each of type~II and type~ III.
\end{rem}

\begin{lem}\label{lingen} The linear polynomials of any f.g.~  ghost-closed prime ideal of $F[\la]$ are generated by at most
two linear polynomials, the extreme case being  $\la +\a$ and
either $\la +\beta$ or $\la^\nu  + \beta,$ where $\a\in F$ and
$\beta \in \tT$ and $\a  <_\nu \beta.$
\end{lem}
\begin{proof}  Since $P$ is f.g., it cannot contain all polynomials of type III.
Thus, in view of Remark \ref{prime2}, we can
 take $\la +\a^\nu$
with $\a^\nu$ minimal possible, and $\la^\nu  + \beta$ or $\la +
\beta$ with $\beta^\nu$ maximal possible. By Example
\ref{linear2}, these generate all other linear factors.
\end{proof}

\begin{prop} The tangible part of any f.g.~ prime ideal $P$ of $F[\la]$ is generated by at most two
linear polynomials.
\end{prop}
\begin{proof} In view of Theorem~\ref{factor1}, the tangible part of $P$ is generated by linear polynomials, and so we conclude by the
lemma.
\end{proof}

\begin{prop} Any f.g.~ prime ghost-closed ideal of $F[\la]$ is supertropically generated by at most two polynomials.
\end{prop}
\begin{proof} Suppose  $P$ is a prime ideal of $F[\la]$.  By Proposition~ \ref{fullfact},  any polynomial in $P$
factors as a product of linear and irreducible quadratic factors $
\la ^2 + \gamma^\nu \la + \gamma\delta,$ where $\gamma, \delta \in
\tT$ with $ \delta^\nu < \gamma^\nu$. $P$ is tropically generated
by at most four polynomials.  In view of
 Lemma \ref{lingen}, the
  linear polynomials in~$P$
  are bounded by $\la+\a$ and $\la + \beta$ (or $\la^\nu + \beta$), and it
remains to consider irreducible quadratic factors.

One could continue using direct computations, but it is easiest to
apply the Nullstellensatz. Each quadratic polynomial has a
tangible root set consisting of a tangible interval, and the
nonempty intersections of these intervals give us root sets
corresponding to quadratic polynomials in $P$. Thus, the quadratic
factors of $P$ are generated by finitely many quadratic factors
whose roots sets are $\nu$-disjoint. We can discard $f=\la ^2 +
\gamma^\nu \la + \gamma\delta$ unless $\gamma^\nu \ge a^\nu$ or
$\beta^\nu \le \delta^\nu,$ as seen by checking the tangible
complements. But then we replace $\la +\a$ by the leftmost such
quadratic factor, and $\la^\nu  + \beta$ (or $\la^\nu  +
\beta^\nu$ by the rightmost such quadratic factor), where
appropriate, and these generate an ideal whose radical is $P$.
\end{proof}

Unfortunately, this  kind of argument shows that any chain of
prime ideals is contained in a chain of infinite length. Indeed,
the prime ideal generated by  $\{\la + \a, \la + \bt \}$ (for
$\a^\nu < \bt^\nu$) is contained in the prime ideal generated by
$\{\la + \a_1, \la + \bt_1\}$, whenever ${\a_1}^\nu \le \a^\nu$
and $\bt^\nu \le {\bt_1}^\nu$.



\end{document}